\documentclass[11pt]{amsart}

\usepackage[a4paper,hmargin=3.5cm,vmargin=4cm]{geometry}
\usepackage{amsfonts,amssymb,amscd,amstext}
\usepackage{graphicx,float}
\usepackage[dvips]{epsfig}
\usepackage{pstricks,pst-plot}

\usepackage{fancyhdr}
\input xy
\xyoption{all}

\usepackage{times}

\usepackage{enumerate}
\usepackage{titlesec}
\usepackage{mathrsfs}

\titleformat{\section}%[display]
{\filcenter\bfseries\large} {\thesection{.}}{0.2cm}{}%[$\vspace*{-1.0cm}$]
\titleformat{\subsection}[runin]
{\bfseries} {\thesubsection{.}}{0.15cm}{}[.]
\titleformat{\subsubsection}[runin]
{\em}{\thesubsubsection{.}}{0.15cm}{}[.]
\usepackage[up,bf]{caption}
\newtheorem{theorem}{Theorem}[section]
\newtheorem{proposition}[theorem]{Proposition}

\newtheorem{lemma}[theorem]{Lemma}
\newtheorem{corollary}[theorem]{Corollary}

\theoremstyle{definition}
\newtheorem{definition}[theorem]{Definition}
\newtheorem{remark}[theorem]{Remark}

\newtheorem{notation}[theorem]{Notation}

\numberwithin{equation}{section}
\numberwithin{figure}{section}

\newcommand\Acal{\mathcal{A}}

\newcommand\Kcal{\mathcal{K}}

\newcommand\Mcal{\mathcal{M}}
\newcommand\Ncal{\mathcal{N}}
\newcommand\Ocal{\mathcal{O}}
\newcommand\Pcal{\mathcal{P}}

\newcommand\Vcal{\mathcal{V}}

\newcommand\Ascr{\mathscr{A}}

\newcommand\Cscr{\mathscr{C}}

\newcommand\Oscr{\mathscr{O}}

\def\c{\mathbb{C}}

\newcommand\cp{\mathbb{CP}}

\def\n{\mathbb{N}}
\renewcommand\r{\mathbb{R}}

\newcommand\z{\mathbb{Z}}
\renewcommand\k{\mathbb{K}}

\newcommand\igot{\mathfrak{i}}

\renewcommand\igot{\mathfrak{i}}

\newcommand\pgot{\mathfrak{p}}
\newcommand\qgot{\mathfrak{q}}

\newcommand\Agot{\mathfrak{A}}

\newcommand\Sgot{\mathfrak{S}}
\newcommand\Zgot{\mathfrak{Z}}

\renewcommand\imath{\igot}

\newcommand\Psf{\mathsf{P}}

\newcommand\wt{\widetilde}
\newcommand\wh{\widehat}
\newcommand\di{\partial}

\newcommand\dist{\mathrm{dist}}
\renewcommand\span{\mathrm{span}}
\newcommand\length{\mathrm{length}}

\newcommand\Flux{\mathrm{Flux}}

\newcommand\supp{\mathrm{supp}}
\newcommand\Qcal{\mathcal{Q}}

\usepackage{color}

\begin{document}

\fancyhead[LO]{Interpolation by minimal surfaces and directed holomorphic curves}
\fancyhead[RE]{A.\ Alarc\'on and I.\ Castro-Infantes}
\fancyhead[RO,LE]{\thepage}

\thispagestyle{empty}

%% Title
\vspace*{7mm}
\begin{center}
{\bf \LARGE Interpolation by conformal minimal surfaces and directed holomorphic curves}

\vspace*{5mm}

%% Authors
{\large\bf Antonio Alarc\'on \; and \; Ildefonso Castro-Infantes}

\vspace*{3mm}

{\large\em Dedicated to Franc Forstneri\v c on the occasion of his sixtieth birthday}
\end{center}

\vspace*{7mm}

\begin{quote}
{\small
\noindent {\bf Abstract} \hspace*{0.1cm}
Let $M$ be an open Riemann surface and $n\ge 3$ be an integer. We prove that on any closed discrete subset of $M$ one can prescribe the values of a conformal minimal immersion $M\to\r^n$. Our result also ensures jet-interpolation of given finite order, and hence, in particular, one may in addition prescribe the values of the generalized Gauss map. Furthermore, the interpolating immersions can be chosen to be complete, proper into $\r^n$ if the prescription of values is proper, and injective if $n\ge 5$ and the prescription of values is injective. We may also prescribe the flux map of the examples.

We also show analogous results for a large family of directed holomorphic immersions $M\to\c^n$, including null curves.

\medskip

\noindent{\bf Keywords} \hspace*{0.1cm} 
minimal surface, directed holomorphic curve, %interpolation theory, 
Weierstrass theorem, Riemann surface, Oka manifold.

\medskip

\noindent{\bf MSC (2010)} \hspace*{0.1cm} 
53A10, 32E30, 32H02, 53A05}
\end{quote}

%%%%%%%%%%
%%%%%%%%%%
%%%%%%%%%%
%%%%%%%%%%  Section: Introduction
%%%%%%%%%%
%%%%%%%%%%

\section{Introduction and main results}\label{sec:intro}

The theory of interpolation by holomorphic functions is a central topic in Complex Analysis which began in the 19th century with the celebrated Weierstrass Interpolation Theorem: 
{\em on a closed discrete subset of a domain $D\subset\c$, one can prescribe the values of a holomorphic function $D\to\c$} (see  \cite{Weierstrass1886}). Much later, in 1948, Florack extended the Weierstrass theorem to arbitrary open Riemann surfaces (see \cite{Florack1948SMIUM}).
In this paper we prove an analogous of this classical result for conformal minimal surfaces in the Euclidean spaces. 

%
% Simplified Main Theorem
% 
\begin{theorem}[Weierstrass Interpolation theorem for conformal minimal surfaces]\label{th:main-intro}
Let $\Lambda$ be a closed discrete subset of an open Riemann surface, $M$, and let $n\ge 3$ be an integer. Every map $\Lambda\to\r^n$ extends to a conformal minimal immersion $M\to\r^n$.
\end{theorem}
Let $M$ be an open Riemann surface and $n\ge 3$ be an integer. By the Identity Principle it is not possible to prescribe values of a conformal minimal immersion $M\to\r^n$ on a subset that is not closed and discrete, hence the assumptions on $\Lambda$ in Theorem \ref{th:main-intro} are necessary. 

Recall that a conformal immersion $X=(X_1,\ldots,X_n)\colon M\to\r^n$ is minimal 
if, and only if, $X$ is a harmonic map. % in the classical sense: $\triangle X=0$. 
If this is the case then, denoting by $\di$ the $\c$-linear part of the exterior differential $d=\di+\overline\di$ on $M$ 
(here $\overline\di$ denotes the $\c$-antilinear part of $d$), the $1$-form $\di X=(\di X_1,\ldots,\di X_n)$ with values in $\c^n$ is holomorphic, has no zeros, and satisfies
$\sum_{j=1}^n (\di X_j)^2=0$ everywhere on $M$. Therefore, $\di X$ determines the Kodaira type holomorphic map
\[%\begin{equation}\label{eq:GX}
	G_X\colon M\to\cp^{n-1},\quad M\ni p\mapsto G_X(p)=[\di X_1(p)\colon\cdots\colon \di X_n(p)],
\]%\end{equation}
which takes values in the complex hyperquadric 
\[%\begin{equation}\label{eq:Q}
	Q_{n-2}=\big\{[z_1\colon\cdots\colon z_n]\in\cp^{n-1}\colon z_1^2+\cdots+z_n^2=0\big\}
	\subset\cp^{n-1}
\]%\end{equation}
%of the complex projective space $\cp^{n-1}$
 and is known as {\em the generalized Gauss map of $X$}; conversely, every holomorphic map $M\to Q_{n-2}\subset\cp^{n-1}$ is the generalized Gauss map of a conformal minimal immersion $M\to\r^n$ (see Alarc\'on, Forstneri\v c, and L\'opez \cite{AlarconForstnericLopez2016Gauss}). 
 %The study of the generalized Gauss map of minimal surfaces is a classical topic with a large literature (we refer to the monograph by Hoffman and Osserman \cite{HoffmanOsserman1980MAMS} and, among many others, the articles \cite{Fujimoto1983JMSJ,Ru1991JDG,OssermanRu1997JDG}).
%
The real part $\Re(\di X)$ is an exact $1$-form on $M$; the {\em flux map} (or simply, the {\em flux}) of $X$ is the group homomorphism $\Flux_X\colon H_1(M;\z)\to\r^n$, of the first homology group of $M$ with integer coefficients, given by
\[
    \Flux_X(\gamma)=\int_\gamma\Im(\di X)=-\imath\int_\gamma \di X,\quad \gamma\in H_1(M;\z),
\]
where $\Im$ denotes the imaginary part and $\imath=\sqrt{-1}$.

Conversely, every holomorphic $1$-form $\Phi=(\phi_1,\ldots,\phi_n)$ with values in $\c^n$, vanishing nowhere on $M$, satisfying the nullity condition 
\begin{equation}\label{eq:nullity}
      \sum_{j=1}^n (\phi_j)^2=0\quad \text{everywhere on $M$},
\end{equation}
and whose real part $\Re(\Phi)$ is exact on $M$, determines a conformal minimal immersion $X\colon M\to\r^n$ with $\di X=\Phi$ by the classical Weierstrass formula
\begin{equation}\label{eq:Weierstrass}
     X(p)=x_0+2\int_{p_0}^p \Re(\Phi),\quad p\in M,
\end{equation}
for any fixed base point $p_0\in M$ and initial condition $X(p_0)=x_0\in\r^n$. (We refer to Osserman \cite{Osserman-book} for a standard reference on Minimal Surface theory.)
This representation formula has greatly influenced the study of minimal surfaces in $\r^n$ by providing powerful tools coming from Complex Analysis in one and several variables. In particular, Runge and Mergelyan theorems for open Riemann surfaces (see Bishop \cite{Bishop1958PJM} and also \cite{Runge1885AM,Mergelyan1951DAN}) and, more recently, the modern {\em Oka Theory} (we refer to the monograph by Forstneri\v c \cite{Forstneric2017} and to the surveys by L\'arusson \cite{Larusson2010NAMS}, Forstneri\v c and L\'arusson \cite{ForstnericLarusson2011NY}, Forstneri\v c \cite{Forstneric2013AFSTM}, and Kutzschebauch \cite{Kutzschebauch2014SPMS}) have been exploited in order to develop a uniform approximation theory for conformal minimal surfaces in the Euclidean spaces which is analogous to the one of holomorphic functions in one complex variable and has found plenty of applications; see
\cite{AlarconLopez2012JDG,AlarconLopez2014TAMS,AlarconForstneric2014IM,AlarconLopez2015GT,DrinovecForstneric2016TAMS,AlarconForstnericLopez2016MZ,AlarconForstnericLopez2016Pre1,ForstnericLarusson2016Pre1} and the references therein. In this paper we extend some of the methods invented for developing this approximation theory %, by adding conceptually new ideas,
 in order to provide also interpolation on closed discrete subsets of the underlying complex structure.

Theorem \ref{th:main-intro} follows from the following much more general result ensuring not only interpolation but also {\em jet-interpolation of given finite order}, approximation on holomorphically convex compact subsets, control on the flux, and global properties such as completeness and, under natural  assumptions, properness and injectivity. If $A$ is a %smoothly bounded 
compact domain in an open Riemann surface, by a {\em conformal minimal immersion $A\to\r^n$ of class $\Cscr^m(A)$}, $m\in\z_+=\{0,1,2,\ldots\}$, we mean an immersion $A\to\r^n$ of class $\Cscr^m(A)$ whose restriction to the interior $\mathring A=A\setminus bA$ is a conformal minimal immersion; we use the same notation if $A$ is a union of pairwise disjoint such domains. 

%
% Main Theorem on minimal surfaces
%

\begin{theorem}[Runge Approximation with Jet-Interpolation for conformal minimal surfaces]\label{th:main-intro2}
Let $M$ be an open Riemann surface, $\Lambda\subset M$ be a closed discrete subset, and $K\subset M$ be a smoothly bounded compact domain such that $M\setminus K$ has no relatively compact connected components.
For each $p\in\Lambda$ let $\Omega_p\subset M$ be a %smoothly bounded 
compact neighborhood of $p$ in $M$, assume that $\Omega_p\cap \Omega_q=\emptyset$ for all $p\neq q\in \Lambda$, and set $\Omega:=\bigcup_{p\in\Lambda}\Omega_p$. Also let  $X\colon K\cup\Omega\to\r^n$ %and $Y\colon \Omega\to\r^n$ 
$(n\ge 3)$ be a conformal minimal immersion of class $\Cscr^1(K\cup\Omega)$ %and $\Cscr^1(\Omega)$, respectively, 
and let $\pgot\colon H_1(M;\z)\to\r^n$ be a group homomorphism satisfying %the following conditions:
\[%\begin{enumerate}[\rm (i)]
%$X=Y$ on $K\cap\Omega$\quad and\quad 
\Flux_X(\gamma)=\pgot(\gamma)\quad \text{for all closed curves $\gamma\subset K$.}
\]%\end{enumerate}
Then, given $k\in\z_+$, $X$ may be approximated uniformly on $K$ by complete conformal minimal immersions $\wt X\colon M\to\r^n$ enjoying the following properties:
\begin{enumerate}[\rm (I)]
\item $\wt X$ and $X$ have a contact of order $k$ at every point in $\Lambda$.
\item $\Flux_{\wt X}=\pgot$.
\item If the map $X|_\Lambda\colon\Lambda\to\r^n$ is proper then we can choose $\wt X\colon M\to\r^n$ to be proper.
\item If $n\ge 5$  and the map $X|_\Lambda\colon\Lambda\to\r^n$ is injective, then we can choose $\wt X\colon M\to\r^n$ to be injective.
\end{enumerate}
\end{theorem}

Condition {\rm (I)} in the above theorem is equivalent to $\wt X|_\Lambda=X|_\Lambda$ and, if $k>0$, the holomorphic $1$-form $\di(\wt X-X)$, assuming values in $\c^n$, has a zero of multiplicity (at least) $k$ at all points in $\Lambda$; in other words, {\em the maps $\wt X$ and $X$ have the same $k$-jet at every point in $\Lambda$} (see Subsec.\ \ref{subsec:Jets}).
This is reminiscent to the generalization of the Weierstrass Interpolation theorem provided by Behnke and Stein in 1949 and asserting that on an open Riemann surface one may prescribe values to arbitrary finite order for a holomorphic function at the points in a given closed discrete subset  (see \cite{BehnkeStein1949MA} or \cite[Theorem 2.15.1]{NapierRamachandran2011Book}). In particular, choosing $k=1$ in Theorem \ref{th:main-intro2} we obtain that {\em on a closed discrete subset of an open Riemann surface, $M$, one can prescribe the values of a conformal minimal immersion $M\to\r^n$ $(n\ge 3)$ and of its generalized Gauss map $M\to Q_{n-2}\subset\cp^{n-1}$} (see Corollary \ref{co:}). The case $\Lambda=\emptyset$ in Theorem \ref{th:main-intro2} (that is, when one does not take care of the interpolation) was recently proved by Alarc\'on, Forstneri\v c, and L\'opez (see \cite[Theorem 1.2]{AlarconForstnericLopez2016MZ}).

Note that the assumptions on $X|_\Lambda$ in assertions {\rm (III)} and {\rm (IV)} in Theorem \ref{th:main-intro2} are necessary. We also point out that if $\Lambda$ is infinite then there are injective maps $\Lambda\to\r^n$ which do not extend to a topological embedding $M\to\r^n$; and hence, in general, one can not choose the conformal minimal immersion $\wt X$ in {\rm (IV)} to be an  {\em embedding} (i.e., a homeomorphism onto $\wt X(M)$ endowed with the subspace topology inherited from $\r^n$). On the other hand, since proper injective immersions $M\to\r^n$ are embeddings, we can choose $\wt X$ in Theorem \ref{th:main-intro2} to be a proper conformal minimal embedding provided that $n\ge 5$ and $X|_\Lambda\colon\Lambda\to\r^n$ is both proper and injective.

%
% Methods
%
 
Let us now say a word about our methods of proof. 
Given a holomorphic $1$-form $\theta$ on $M$ with no zeros (such exists by the Oka-Grauert principle; see \cite{Grauert1957MA,Grauert1958MA} or \cite[Theorem 5.3.1]{Forstneric2017}), any holomorphic $1$-form $\Phi=(\phi_1,\ldots,\phi_n)$ on $M$ with values in $\c^n$ and satisfying the nullity condition \eqref{eq:nullity} can be written in the form $\Phi=f\theta$ where $f\colon M\to\c^n$ is a holomorphic function taking values in the {\em null quadric} (also called the {\em complex light cone}):
\begin{equation}\label{eq:nullquadric}
    \Agot:=\{z=(z_1,\ldots,z_n)\in\c^n\colon z_1^2+\cdots+z_n^2=0\}.
\end{equation}
Therefore, in order to prove Theorem \ref{th:main-intro} one needs to find a holomorphic map 
\[
	f\colon M\to\Agot\setminus\{0\}\subset\c^n
\]
 such that $\Re(f\theta)$ is an exact real $1$-form on $M$ %(equivalently, $f\theta$ has no real periods on $M$; see \eqref{eq:norealperiods}) 
and
\[
     2 \int_{p_0}^p \Re(f\theta)= \Zgot(p)\quad \text{for all $p\in \Lambda$},
\]
where $p_0\in M\setminus \Lambda$ is a fixed base point and $\Zgot\colon\Lambda\to\r^n$ is the given map.
Then the formula \eqref{eq:Weierstrass} with $x_0=0$ and $\Phi=f\theta$ provides a conformal minimal immersion satisfying the conclusion of the theorem. The key in this approach is that the {\em punctured null quadric}
\begin{equation}\label{eq:punctured}
     \Agot_*:=\Agot\setminus\{0\}
\end{equation}
%(cf.\ \eqref{eq:nullquadric})
is a {\em complex homogeneous manifold} and hence an {\em Oka manifold} (see \cite[Example 4.4]{AlarconForstneric2014IM}); thus, there are {\em many} holomorphic maps $M\to \Agot_*$  (see Subsec.\ \ref{ss:Oka} for more information). 

The proof of Theorem \ref{th:main-intro2} is much more involved and elaborate. It requires, in addition to the above, a subtle use of the Runge-Mergelyan theorem with jet-interpolation for holomorphic maps from open Riemann surfaces into Oka manifolds (see Theorem \ref{th:MTJI}) to achieve condition {\rm (I)}, a conceptually new intrinsic-extrinsic version of the technique by Jorge and Xavier from \cite{JorgeXavier1980AM} to ensure completeness of the interpolating immersions (see Lemma \ref{lem:complete-S} and Subsec.\ \ref{ss:(I)}), and, in order to guarantee assertion {\rm (III)}, of extending the recently developed methods in \cite{AlarconLopez2012JDG,AlarconForstneric2014IM,AlarconForstnericLopez2016MZ} for constructing proper minimal surfaces in $\r^n$ with arbitrary complex structure (see Lemma \ref{lem:proper-S} and Subsec.\ \ref{ss:(II)}). Moreover, in order to prove {\rm (IV)} we adapt the transversality approach by Abraham \cite{Abraham1963BAMS} (cf.\  \cite{AlarconForstneric2014IM,AlarconDrinovecForstnericLopez2015PLMS,AlarconForstnericLopez2016MZ} for its implementation in Minimal Surface theory); see Theorem \ref{th:gp-S}.

The above described method for constructing conformal minimal surfaces in $\r^n$, based on Oka theory, was introduced by Alarc\'on and Forstneri\v c in \cite{AlarconForstneric2014IM} and it also works in the more general framework of {\em directed holomorphic immersions} of open Riemann surfaces into complex Euclidean spaces. Directed immersions have been the focus of interest in a number of classical geometries such as symplectic, contact, Lagrangian, totally real, etc.; we refer for instance to the monograph by Gromov \cite{Gromov1986Book}, to Eliashberg and Mishachev \cite[Chapter 19]{EliashbergMishachev2002Book}, and to the introduction of \cite{AlarconForstneric2014IM} for motivation on this subject. Given a (topologically) closed conical complex subvariety $\Sgot$ of $\c^n$ $(n\ge 3)$, a holomorphic immersion $F\colon M \to \c^n$ of an open Riemann surface $M$ into $\c^n$ is said to be {\em directed by} $\Sgot$, or an {\em $\Sgot$-immersion}, if its complex derivative $F'$ with respect to any local holomorphic coordinate on $M$ assumes values in 
\[
	\Sgot_*:=\Sgot\setminus\{0\}
\]
%; $F$ is said to be an {\em $\Sgot$-embedding} if it is an injective $\Sgot$-immersion 
(cf.\ \cite[Definition 2.1]{AlarconForstneric2014IM}). If $A$ is a %smoothly bounded 
compact domain in an open Riemann surface, or a union of pairwise disjoint such domains, by an {\em $\Sgot$-immersion $A\to\c^n$ of class $\Ascr^m(A)$} $(m\in\z_+)$ we mean an immersion $A\to\c^n$ of class $\Cscr^m(A)$ whose restriction to the interior, $\mathring A$, is a (holomorphic) $\Sgot$-immersion. % (see Subsec.\ \ref{ss:RS}).
 Among others, general existence, approximation, and desingularization results were proved in \cite{AlarconForstneric2014IM} for certain families of directed holomorphic immersions, including {\em null curves}: holomorphic curves in $\c^n$ which are
 directed by the null quadric $\Agot\subset\c^n$ (see \eqref{eq:nullquadric}). It is well known that the real and imaginary parts of a null curve $M\to\c^n$ are conformal minimal immersions $M\to\r^n$ whose flux map vanishes everywhere on $H_1(M;\z)$; conversely, every conformal minimal immersion $M\to\r^n$ is locally, on every simply-connected domain of $M$, the real part of a null curve $M\to\c^n$ (see \cite[Chapter 4]{Osserman-book}). 

The second main theorem of this paper is an analogue of Theorem \ref{th:main-intro2} for a wide family of directed holomorphic curves in $\c^n$  which includes null curves. Given integers $1\le j\le n$ we denote by $\pi_j\colon\c^n\to\c$ the coordinate projection $\pi_j(z_1,\ldots,z_n)=z_j$.

%
% Main Theorem on directed immersions
%

\begin{theorem}[Runge Approximation with Jet-Interpolation for directed holomorphic curves]\label{th:main-intro3}
Let $\Sgot$ be an irreducible closed conical complex subvariety of $\c^n$ $(n\ge 3)$ which is contained in no hyperplane and such that $\Sgot_*=\Sgot\setminus\{0\}$ is smooth and an Oka manifold. Let $M$,  $\Lambda$, $K$, and $\Omega$ be as in Theorem \ref{th:main-intro2} and  
let  $F\colon K\cup\Omega\to\c^n$ %and $G\colon \Omega\to\c^n$ 
be an $\Sgot$-immersion of class $\Ascr^1(K\cup\Omega)$. %, respectively, such that
%\[
%     F=G\quad \text{on $K\cap\Omega$}.
%\]
 Then, given $k\in\n$, $F$ may be approximated uniformly on $K$ by $\Sgot$-immersions $\wt F\colon M\to\c^n$ such that $\wt F-F$ has a zero of multiplicity (at least) $k$ at every point in $\Lambda$. Moreover, if the map $F|_\Lambda\colon\Lambda\to\c^n$ is injective, then we can choose $\wt F\colon M\to\c^n$ to be injective.

Furthermore:
\begin{enumerate}[\rm (I)]
\item If $\Sgot\cap\{z_1=1\}$ is an Oka manifold and $\pi_1\colon \Sgot\to\c$ admits a local holomorphic section $h$ near $\zeta=0\in\c$ with $h(0)\neq 0$, then we may choose $\wt F$ to be complete.
\item If $\Sgot\cap\{z_j=1\}$ is an Oka manifold and $\pi_j\colon \Sgot\to\c$ admits a local holomorphic section $h_j$ near $\zeta=0\in\c$ with $h_j(0)\neq 0$ for all $j\in\{1,\ldots,n\}$, and if the map $F|_\Lambda\colon\Lambda\to\c^n$ is proper, then we may choose $\wt F\colon M\to\c^n$ to be proper.
\end{enumerate}
\end{theorem}

In particular, if we are given $\Sgot$, $M$, and $\Lambda$ as in Theorem \ref{th:main-intro3} then {\em every map $\Lambda\to\c^n$ extends to an $\Sgot$-immersion $M\to\c^n$}. When the subset $\Lambda\subset M$ is empty the above theorem except for assertion {\rm (I)} is implied by \cite[Theorems 7.2 and 8.1]{AlarconForstneric2014IM}. It is perhaps worth mentioning to this respect that, if $\Sgot$ is as in assertion {\rm (I)} and $F|_\Lambda\colon\Lambda\to\c^n$ is not proper, Theorem \ref{th:main-intro3} provides {\em complete} $\Sgot$-immersions $M\to\c^n$ which are not proper maps; these seem to be the first known examples of such apart from the case when $\Sgot$ is the null quadric. Let us emphasize that the particular geometry of $\Agot$ has allowed to construct complete null holomorphic curves in $\c^n$ and minimal surfaces in $\r^n$ with a number of different asymptotic behaviors (other than proper in space); see \cite{AlarconLopez2013MA,AlarconForstneric2015MA,AlarconDrinovecForstnericLopez2015Pre,AlarconDrinovecForstnericLopez2015PLMS,AlarconCastro2016Pre1} and the references therein.

Most of the technical part in the proof of Theorems \ref{th:main-intro2} and \ref{th:main-intro3} will be furnished by a general result concerning periods of holomorphic $1$-forms with values in a closed conical complex subvariety of $\c^n$ (see Theorem \ref{th:MtT} for a precise statement). With this at hand, the proofs of Theorems \ref{th:main-intro2} and \ref{th:main-intro3} are very similar; this is why, with brevity of exposition in mind, we shall spell out in detail the proof of Theorem \ref{th:main-intro3} (which is, in some sense, more general) but only briefly sketch the one of Theorem \ref{th:main-intro2}.

This paper is, to the best of our knowledge, the first contribution to the theory of interpolation by conformal minimal surfaces and directed holomorphic curves in a Euclidean space. 

%
% Organization
%

%\smallskip\noindent{\bf Organization of the paper.}
\subsection*{Organization of the paper}
In Section \ref{sec:prelim} we state some notation and the preliminaries which are needed throughout the paper; we also show an observation which is crucial to ensure the jet-interpolation conditions in Theorems \ref{th:main-intro2} and \ref{th:main-intro3} (see Lemma \ref{lem:jet}). Section \ref{sec:paths} is devoted to the proof of several preliminary results on the existence of period-dominating sprays of maps into conical complex subvarieties $\Sgot_*$ of $\c^n$; we use them in Section \ref{sec:jet} to prove the non-critical case of a Mergelyan theorem with jet-interpolation and control on the periods for holomorphic maps into a such $\Sgot_*$ being Oka (see Lemma \ref{lem:main-noncritical}), and the main technical result of the paper (Theorem \ref{th:MtT}). In Section \ref{sec:plus} we prove a general position theorem, a completeness lemma, and a properness lemma for $\Sgot$-immersions, which enable us to complete the proof of Theorem \ref{th:main-intro3} in Section \ref{sec:maintheorem}. Finally, Section \ref{sec:maintheoremMS} is devoted to explain how the methods in the proof of Theorem \ref{th:main-intro3} can be adapted to prove Theorem \ref{th:main-intro2}. 

%%%%%%%%%%
%%%%%%%%%%
%%%%%%%%%% Section: Preliminaries
%%%%%%%%%%   
%%%%%%%%%%
%%%%%%%%%%

\section{Preliminaries}\label{sec:prelim}

We denote $\imath=\sqrt{-1}$, $\z_+=\{0,1,2,\ldots\}$, and $\r_+=[0,+\infty)$.
Given an integer $n\in\n=\{1,2,3,\ldots\}$ and $\k\in\{\r,\c\}$, we denote by $|\cdot|$, $\dist(\cdot,\cdot)$, and $\length(\cdot)$ the Euclidean norm, distance, and length in $\k^n$, respectively. If $K$ is a compact topological space and $f\colon K\to \k^n$ is a continuous map, we denote by 
\[
	\|f\|_{0,K}:=\max\{|f(p)|\colon p\in K\}
\]
the maximum norm of $f$ on $K$. Likewise, given $x=(x_1,\ldots,x_n)$ in $\k^n$ we denote 
\[
	|x|_\infty:=\max\{|x_1|,\ldots,|x_n|\} \quad\text{and}\quad \|f\|_{\infty,K}:=\max\{|f(p)|_\infty\colon p\in K\}.
\]
If $K$ is a subset of a Riemann surface, $M$, then for any $r\in\z_+$ we shall denote by $\|f\|_{r,K}$  the standard $\Cscr^r$ norm of a function $f\colon K\to\k^n$ of class $\Cscr^r(K)$, where the derivatives are measured with respect to a Riemannian metric on $M$ (the precise choice of the metric will not be important).

Given a smooth connected surface $S$ (possibly with nonempty boundary) and a smooth immersion $X\colon S\to\k^n$, we denote by $\dist_X\colon S\times S\to\r_+$ the Riemannian distance induced on $S$ by the Euclidean metric of $\k^n$ via $X$; i.e.,
\[
	\dist_X(p,q):=\inf\{\length(X(\gamma))\colon 
	\text{$\gamma\subset S$ arc connecting $p$ and $q\}$},\quad p,q\in S.
\]
Likewise, if $K\subset S$ is a relatively compact subset we define
\[
	\dist_X(p,K):=\inf\{ \dist_X(p,q)\colon q\in K \},\quad p\in S.
\]

An immersed open surface $X\colon S\to \k^n$ $(n\ge 3)$ is said to be {\em complete} if the image by $X$ of any proper path $\gamma\colon [0,1)\to S$ has infinite Euclidean length; equivalently, if the Riemannian metric on $S$ induced by $\dist_X$ is complete in the classical sense. On the other hand, $X\colon S\to \k^n$ is said to be {\em proper} if the image by $X$ of every proper path $\gamma\colon [0,1)\to S$ is a divergent path in $\k^n$.

%%%%%%%%%%
%%%%%%%%%%
%%%%%%%%%% Subsection: Riemann surfaces
%%%%%%%%%%   
%%%%%%%%%%
%%%%%%%%%%

\subsection{Riemann surfaces and spaces of maps}\label{ss:RS}

Throughout the paper every Riemann surface will be considered connected if the contrary is not indicated.

Let $M$ be an open Riemann surface. Given a subset $A\subset M$ we denote by $\Ocal(A)$ the space of functions $A\to \c$ which are holomorphic on an unspecified open neighborhood of $A$ in $M$. If $A$ is a smoothly bounded compact domain, or a union of pairwise disjoint such domains, and $r\in\z_+$, we denote by $\Ascr^r(A)$ the space of $\Cscr^r$ functions $A\to\c$ which are holomorphic on the interior $\mathring A=A\setminus bA$; for simplicity we write $\Ascr(A)$ for $\Ascr^0(A)$. Likewise, we define the spaces $\Ocal(A,Z)$ and $\Ascr^r(A,Z)$ of maps $A\to Z$ to any complex manifold $Z$. Thus, if $\Sgot$ is a closed conical complex subvariety of $\c^n$ $(n\ge 3)$, by an $\Sgot$-immersion $A\to\c^n$ of class $\Ascr^r(A)$ we simply mean an immersion of class $\Ascr^r(A)$ whose restriction to $\mathring A$ is an $\Sgot$-immersion. In the same way, a conformal minimal immersion $A\to\r^n$ of class $\Cscr^r(A)$ will be nothing but an immersion of class $\Cscr^r(A)$ whose restriction to $\mathring A$ is a conformal minimal immersion.

By a {\em compact bordered Riemann surface} we mean a compact Riemann surface $M$ with nonempty boundary $bM$ consisting of finitely many pairwise disjoint smooth Jordan curves. The interior $\mathring M=M\setminus bM$ of $M$ is called a {\em bordered Riemann surface}. It is well known that every compact bordered Riemann surface $M$ is diffeomorphic to a smoothly bounded compact domain in an open Riemann surface $\wt M$. The spaces $\Ascr^r(M)$ and $\Ascr^r(M,Z)$, for an integer $r\in\z_+$ and a complex manifold $Z$, are defined as above.

A compact subset $K$ in an open Riemann surface $M$ is said to be {\em Runge} (also called {\em holomorphically convex} or {\em $\Oscr(M)$-convex}) if every continuous function $K\to\c$, holomorphic in the interior $\mathring K$, may be approximated uniformly on $K$ by holomorphic functions on $M$; by the Runge-Mergelyan theorem \cite{Runge1885AM,Mergelyan1951DAN,Bishop1958PJM} this is equivalent to that $M\setminus K$ has no relatively compact connected components in $M$. 
The following particular kind of Runge subsets will play a crucial role in our argumentation.
\begin{definition}\label{def:admissible}
A nonempty compact subset $S$ of an open Riemann surface $M$ is called {\em admissible} if it is Runge in $M$ and of the form $S=K\cup \Gamma$, where $K$ is the union of finitely many pairwise disjoint  smoothly bounded compact domains in $M$ and $\Gamma:=\overline{S\setminus K}$ is a finite union of pairwise disjoint smooth Jordan arcs and closed Jordan curves meeting $K$ only in their endpoints (or not at all) and such that their intersections with the boundary $bK$ of $K$ are transverse. 
%See Figure \ref{fig:admissible}.
\end{definition}
%\begin{figure}[ht]
%	\includegraphics[width=10cm]{admissible.eps}
%	\caption{An admissible set.}\label{fig:admissible}
%\end{figure}

If $C$ and $C'$ are oriented arcs in $M$, and the initial point of $C'$ is the final one of $C$, we denote by $C*C'$ the product of $C$ and $C'$, i.e., the oriented arc $C\cup C'\subset M$ with initial point the initial point of $C$ and final point the final point of $C'$.

Every open connected Riemann surface $M$ contains a $1$-dimensional embedded CW-complex $C\subset M$ such that there is a strong deformation retraction $\rho_t\colon M\to M$ $(t \in [0,1])$; i.e., $\rho_0={\rm Id}_M$, $\rho_t|_C={\rm Id}|_C$ for all $t\in [0,1]$, and $\rho_1(M) = C$. It follows that the complement $M\setminus C$ has no relatively compact connected components in $M$ and hence $C$ is Runge.  Such a CW-complex $C\subset M$ represents the topology of $M$ and can be obtained, for instance, as the Morse complex of a Morse strongly subharmonic exhaustion function on $M$. Recall that the first homology group $H_1(M;\z)=\z^l$ for some $l\in\z_+\cup\{\infty\}$. It is not difficult to see that, if $M$ is finitely-connected (for instance, if it is a bordered Riemann surface), i.e., if $l\in\z_+$, then, given a point $p_0\in M$ there is a CW-complex $C\subset M$ as above which is a bouquet of $l$ circles with base point $p_0$; i.e., $\{p_0\}$ is the only $0$-cell of $C$, and $C$ has $l$ $1$-cells $C_1,\ldots, C_l$ which are closed Jordan curves on $M$ that only meet at $p_0$. 

%%%%%%%%%%
%%%%%%%%%%
%%%%%%%%%% Subsection: Jets
%%%%%%%%%%   
%%%%%%%%%%
%%%%%%%%%%

\subsection{Jets}\label{subsec:Jets}

Let $\Mcal$ and $\Ncal$ be smooth manifolds without boundary, $x_0\in \Mcal$ be a point, and $f,g\colon \Mcal\to\Ncal$ be smooth maps. The maps $f$ and $g$ have, by definition, a {\em contact of order $k\in\z_+$} at the point $x_0$ if their Taylor series at this point coincide up to the order $k$.  An equivalence class of maps $\Mcal\to\Ncal$ which have a contact of order $k$ at the point $x_0$ is called a {\em $k$-jet}; see e.\ g.\ \cite[\textsection 1]{Michor1980SP} for a basic reference. Recall that the Taylor series %of order $k$ 
at $x_0$ of a smooth map $f\colon \Mcal\to\Ncal$ does not depend on the choice of coordinate charts on $\Mcal$ and $\Ncal$ centered at $x_0$ and $f(x_0)$ respectively. Therefore, fixing such a pair of coordinates, %on $\Mcal$ and $\Ncal$ in a neighborhood of $x_0$ and $f(x_0)$, respectively, 
we can identify the $k$-jet of $f$ at $x_0$, which is usually denoted by $j_{x_0}^k(f)$, with the set of derivatives of $f$ at $x_0$ of order up to and including $k$; under this identification of jets we have
\[
     j_{x_0}^0(f)=f(x_0),\; j_{x_0}^1(f)=\big( f(x_0),\frac{\di f}{\di x}\big|_{x_0}\big), 
     \; j_{x_0}^2(f)=\big( f(x_0),\frac{\di f}{\di x}\big|_{x_0},\frac{\di^2 f}{\di x^2}\big|_{x_0}\big),
     \;\ldots
\]
Analogously, if $\Mcal$ and $\Ncal$ are complex manifolds then we consider the complex (holomorphic) derivatives with respect to some local holomorphic coordinates. It is clear that the definition of the $k$-jet of a map at a point is local and hence it can be made for germs of maps at the point. Moreover, if a pair of maps have the same $k$-jet at a point then, obviously, they also have the same $k'$-jet at the point for all $k'\in\z_+,$ $k'\le k$.

In particular, if $\Omega$ is a neighborhood of a point $p$ in an open Riemann surface $M$ and $f,g\colon \Omega\to\c^n$ are holomorphic functions, then they have a contact of order $k\in\z_+$, or the same $k$-jet, at the point $p$ if, and only if, $f-g$ has a zero of multiplicity (at least) $k+1$ at $p$; if this is the case then for any distance function ${\sf d}\colon M\times M\to \r_+$ on $M$ (not necessarily conformal) we have 
\begin{equation}\label{eq:O}    
      |f-g|(q)=O({\sf d}(q,p)^{k+1})\quad \text{as $q\to p$}.
\end{equation}

If $f,g\colon \Omega\to\r^n$ are harmonic maps (as, for instance, conformal minimal immersions), then we say that they have a contact of order $k\in\z_+$, or the same $k$-jet, at the point $p$ if assuming that $\Omega$ is simply-connected there are harmonic conjugates $\wt f$ of $f$ and $\wt g$ of $g$ such that the holomorphic functions $f+\imath\wt f,g+\imath\wt g\colon \Omega\to\c^n$ have a contact of order $k$ at $p$; this is equivalent to that $f(p)=g(p)$ and, if $k>0$, the holomorphic $1$-form $\di(f-g)$ has a zero of multiplicity (at least) $k$ at $p$. Again, if such a pair of maps $f$ and $g$ have the same $k$-jet at the point $p\in \Omega$ then \eqref{eq:O} formally holds. 

The following observation will be crucial in order to ensure the jet-interpolation in the main results of this paper.
\begin{lemma}\label{lem:jet}
Let $V$ be a holomorphic vector field in $\c^n$ $(n\in\n)$, vanishing at $0\in\c^n$, and let $\phi_s$ denote the flow of $V$ for small values of time $s\in\c$. Given an open Riemann surface $M$, a point $p\in M$, and holomorphic functions $f\colon M\to\c^n$ and $h\colon M\to\c$ such that $h$ has a zero of multiplicity $k+1$ at $p$ for some $k\in\z_+$, then the holomorphic map 
\[
     q\mapsto \wt f(q)=\phi_{h(q)}(f(q)),
\]
which is defined on a neighborhood of $p$ in $M$, has a contact of order $k$ with $f$ at the point $p$; that is, $f$ and $\wt f$ have the same $k$-jet at $p$.
\end{lemma}
\begin{proof}
The flow $\phi_s$ of the vector field $V$ at a point $z\in\c^n$ may be expressed as
\[
      \phi_s(z)%=z+\sum_{j\in\n}\frac1{j!}s^j V^j(z)
      =z+sV(z)+O(|s|^2),
\]
%where $V^j=V\circ\stackrel{\text{$j$ times}}{\cdots}\circ V$ for all $j\in\n$
(see e.\ g.\ \cite[\textsection 4.1]{AbrahamMarsdenRatiu1988SV}). Since $h$ has a zero of multiplicity $k+1$ at $p$, the conclusion of the lemma follows.
\end{proof}

We shall use the following notation at several places throughout the paper.
\begin{notation}\label{not:}
	Let $n\ge 3$ be an integer and $\Sgot$ be a (topologically) closed conical complex subvariety of $\c^n$; by {\em conical} we mean that $t\Sgot=\Sgot$ for all $t\in\c_*=\c\setminus\{0\}$. We also assume that $\Sgot$ is contained in no hyperplane of $\c^n$, and $\Sgot_*:=\Sgot\setminus\{0\}$ is smooth and connected (hence irreducible). We also fix a large integer $N\ge n$ and holomorphic vector fields $V_1,\dots,V_N$ on $\c^n$ which are tangential to $\Sgot$ along $\Sgot$, vanish at $0\in\Sgot$, and satisfy 
	\begin{equation}\label{eq:CartanA}
	\span\{V_1(z),\dots,V_N(z)\}=T_{z}\Sgot\quad \text{for all $z\in\Sgot_*$}.
	\end{equation}
	(Such exist by Cartan's theorem A \cite{Cartan1953}.) 
	\begin{equation}\label{eq:CartanA1}
	\text{Let $\phi^j_s$ denote the flow of the vector field $V_j$}
	\end{equation}
	for $j=1,\dots,N$ and small values of the time $s\in\c$.
\end{notation}

\begin{remark}\label{rem:multiplicity}
Throughout the paper we shall say that a holomorphic function has a zero of multiplicity $k\in\n$ at a point to mean that the function has a zero of multiplicity {\em at least} $k$ at the point. When the multiplicity of the zero is exactly $k$ then it will be explicitly mentioned. We will follow the same pattern when claiming that two functions have the same $k$-jet or a contact of order $k$ at a point.
\end{remark}

%%%%%%%%%%
%%%%%%%%%%
%%%%%%%%%% Subsection: Oka manifolds
%%%%%%%%%%   
%%%%%%%%%%
%%%%%%%%%%

\subsection{Oka manifolds}\label{ss:Oka}

In this subsection we recall the notion of Oka manifold and state some of the properties of such manifolds which will be exploited in our argumentation. A comprehensive treatment of Oka theory can be found in \cite{Forstneric2017}; for a briefer introduction to the topic we refer to
 \cite{Larusson2010NAMS,ForstnericLarusson2011NY,Forstneric2013AFSTM,Kutzschebauch2014SPMS}.

\begin{definition}
A complex manifold $Z$ is said to be an {\em Oka manifold} if every holomorphic map from a neighborhood of a compact convex set $K\subset \c^N$ $(N\in\n)$ to $Z$ can be approximated uniformly on $K$ by entire maps $\c^N\to Z$. 
\end{definition}

The central result of Oka theory is that maps $M\to Z$ from a Stein manifold (as, for instance, an open Riemann surface) to an Oka manifold satisfy all forms of the Oka principle (see Forstneri\v c \cite{Forstneric2006AM}). In this paper we shall use as a fundamental tool the following version of the {\em Mergelyan Theorem with Jet-Interpolation} which trivially follows from \cite[Theorems 3.8.1 and 5.4.4]{Forstneric2017}; see also \cite[Theorem 3.2]{Forstneric2004AIF} and \cite[Theorem 4.1]{HormanderWermer1968MS}.

\begin{theorem}\label{th:MTJI}
Let $Z$ be an Oka manifold, let $M$ be an open Riemann surface, and let $S=K\cup\Gamma\subset M$ be an admissible subset in the sense of Definition \ref{def:admissible}. Given a finite subset $\Lambda\subset \mathring{K}$ and an integer $k\in\z_+$, every continuous map $f\colon S\to Z$ which is holomorphic on $\mathring K$ can be approximated uniformly on $S$ by holomorphic maps $M\to Z$ having the same $k$-jet as $f$ at all points in $\Lambda$.
\end{theorem}

As we emphasized in the introduction, the punctured null quadric $\Agot_*\subset \c^n$ (see \eqref{eq:nullquadric} and \eqref{eq:punctured}) directing minimal surfaces in $\r^n$ and null curves in $\c^n$ is an Oka manifold for all $n\ge 3$ (see \cite[Example 4.4]{AlarconForstneric2014IM} or \cite[Example 5.6.2]{Forstneric2017}). Furthermore, for each $j\in\{1,\ldots,n\}$ the complex manifold $\Agot\cap\{z_j=1\}$ is an embedded copy of the complex $(n-2)$-sphere
\[
	\c S^{n-2}=\{w=(w_1,\ldots,w_{n-1})\in\c^{n-1}\colon w_1^2+\cdots+w_{n-1}^2=1\}.
\]
Observe that $\c S^{n-2}$ is homogeneous relative to the complex Lie group $SO(n-1,\c)$, and hence it is an Oka manifold (see \cite{Grauert1957MA} or \cite[Proposition 5.6.1]{Forstneric2017}); see \cite[Example 6.15.7]{Forstneric2017} and \cite[Example 7.8]{AlarconForstneric2014IM} for a more detailed discussion. Moreover, choosing $k\in\{1,\ldots,n\}$, $k\neq j$, the map $h=(h_1,\ldots,h_n)\colon \c\to\Agot$ given by
\[
	h_j(\zeta)=\zeta,\quad h_k(\zeta)=\sqrt{1-\zeta^2},\quad 
	h_l(\zeta)=\frac{\imath}{\sqrt{n-2}}\quad \text{for all }l\neq j,k,\quad \zeta\in\c,
\]
is a local holomorphic section near $\zeta=0\in\c$ of the coordinate projection $\pi_j\colon \Agot\to\c$, $\pi_j(z_1,\ldots,z_n)=z_j$, which satisfies $h(0)\neq 0$. Thus, the null quadric $\Agot\subset\c^n$ meets the requirements in Theorem \ref{th:main-intro3}, including the ones in assertions {\rm (I)} and {\rm (II)}, for all $n\ge 3$.

%%%%%%%%%%
%%%%%%%%%%
%%%%%%%%%%
%%%%%%%%%%   Section: PATHS
%%%%%%%%%%
%%%%%%%%%%

\section{Paths in closed conical complex subvarieties of $\c^n$}\label{sec:paths}

In this section we use Notation \ref{not:}; in particular, $\Sgot\subset\c^n$ $(n\ge3)$ denotes a closed conical complex subvariety which is contained in no hyperplane of $\c^n$ and such that $\Sgot_*=\Sgot\setminus\{0\}$ is smooth and connected. We need the following

\begin{definition}\label{def:degenerateI}
Let $Q$ be a topological space and $n\ge 3$ be an integer. A continuous map $f\colon Q\to\c^n$ is said to be {\em flat} if $f(Q)\subset \c z_0=\{\zeta z_0\colon \zeta\in\c\}$ for some $z_0\in\c^n$; and {\em nonflat} otherwise. The map $f$ is said to be {\em nowhere flat} if $f|_{A}\colon A\to\c^n$ is nonflat for all open subset $\emptyset\neq A\subset Q$.
\end{definition}

It is easily seen that a continuous map $f\colon [0,1]\to\Sgot_*\subset\c^n$ is nonflat if, and only if, 
\[
\span\{T_{f(t)}\Sgot\colon t\in [0,1]\}=\c^n. 
\]

\subsection{Paths on $I:=[0,1]$}
%
% Lemma of sprays
%

In this subsection we prove a couple of technical results for paths $[0,1]\to\Sgot_*$ which pave the way to the construction of period dominating sprays of holomorphic maps of an open Riemann surface into $\Sgot_*$  (see Lemma \ref{lem:pathRiemann} in the next subsection).

\begin{lemma}\label{lem:path-sprays}
Let $f\colon I\to\Sgot_*$ and $\vartheta\colon I\to\c_*$ be continuous maps. Let $\emptyset\neq I'\subset I$ be a closed subinterval and assume that $f$ is nowhere flat on $I'$. There exist continuous functions $h_1,\ldots,h_N\colon I\to\c$, with support on $I'$, and a neighborhood $U$ of $0\in\c^N$ such that the {\em period map} $\Pcal\colon U\to\c^n$ given by
\[
      \Pcal(\zeta)=\int_0^1 \phi^1_{\zeta_1h_1(t)}\circ \cdots \circ \phi^N_{\zeta_Nh_N(t)}(f(t))\vartheta(t)\, dt,\quad \zeta=(\zeta_1,\ldots,\zeta_N)\in U
\]
(see \eqref{eq:CartanA1}),
is well defined and has maximal rank equal to $n$ at $\zeta=0$.
\end{lemma}
\begin{proof}
We choose continuous functions $h_1,\ldots,h_N\colon I\to\c$, with support on $I'$, which will be specified later. Then we define for a small neighborhood $U$ of $0\in\c^N$ a map
\[
    \Phi\colon U\times I\to \Sgot
\]
given by
\[
    \Phi(\zeta,t):=\phi^1_{\zeta_1h_1(t)}\circ \cdots \circ \phi^N_{\zeta_Nh_N(t)}(f(t)),\quad \zeta=(\zeta_1,\ldots,\zeta_N)\in U,\; t\in I.
\]
Note that $\Phi(0,t)=f(t)$ for all $t\in I$; recall that each $V_j$ vanishes at $0$ for all $j\in\{1,\ldots,N\}$. Thus, since $f(I)\subset\Sgot_*$ is compact, we may assume that $U$ is small enough so that $\Phi$ is well defined and takes values in $\Sgot_*$. Furthermore, $\Phi$ is holomorphic in the variable $\zeta$ and its derivative with respect to $\zeta_j$ is
\begin{equation}\label{eq:diPhi}
    \left. \frac{\di \Phi(\zeta,t)}{\di \zeta_j}\right|_{\zeta=0}=h_j(t)V_j(f(t)),\quad j=1,\dots,N. 
\end{equation}
(See \eqref{eq:CartanA} and \eqref{eq:CartanA1}.)
Thus, the period map $\Pcal\colon U\to\c^n$ in the statement of the lemma reads
\[
     \Pcal(\zeta)=\int_0^1 \Phi(\zeta,t)\vartheta(t)\, dt,\quad \zeta\in U.
\]
Observe that $\Pcal$ is holomorphic and, in view of \eqref{eq:diPhi},
\begin{equation}\label{eq:diPzeta}
     \left. \frac{\di \Pcal(\zeta)}{\di \zeta_j}\right|_{\zeta=0}=\int_0^1 h_j(t) V_j(f(t))\vartheta(t)\, dt,\quad j=1,\ldots,N.
\end{equation}

Since $f$ is nowhere flat on $I'$ (see Definition \ref{def:degenerateI}), \eqref{eq:CartanA} guarantees the existence of distinct points $t_1,\dots,t_N\in I'$ such that 
\begin{equation}\label{eq:span}
\span\{V_1(f(t_1)),\dots,V_N(f(t_N))\}=\c^n.
\end{equation}

Now we specify the values of the function $h_j$ in $I'$ $(j=1,\dots,N)$; recall that $\supp(h_j)\subset I'$. We choose $h_j$ with support in a small neighborhood $[t_j-\epsilon,t_j+\epsilon]$ of $t_j$ in $I'$, for some $\epsilon>0$, and such that  
\[
      \int_0^1 h_j(t)\, dt =\int_{t_j-\epsilon}^{t_j+\epsilon} h_j(t)\, dt=1.
\] 
Then, for small $\epsilon>0$, we have that
\[
   \int_0^1 h_j(t) V_j(f(t))\vartheta(t)\, dt \approx V_j(f(t_j))\vartheta(t_j),\quad j=1,\dots,N.
\]
Since $\vartheta(t)\neq 0$, \eqref{eq:span} ensures that the vectors on the right side of the above display span $\c^n$, and hence the same is true for the vectors on the left side provided that $\epsilon>0$ is chosen sufficiently small. This concludes the proof in view of \eqref{eq:diPzeta}.
\end{proof}

%
% Lemma of prescribing periods
%

\begin{lemma}\label{lem:path-periods}
Let $\vartheta\colon I\to\c_*$ be a continuous map. %and $\emptyset\neq I'\subset \mathring I=(0,1)$ be a closed subinterval.
Given points $u_0,u_1\in\Sgot_*$ and $x\in\c^n$,
and a domain $\Omega$ in $\c^n$ containing $0$ and $x$, there exists a continuous function $g\colon I\to\Sgot_*$ which is nowhere flat on a neighborhood of $0$ in $I$ and such that:
\begin{enumerate}[\it i)]
\item $g(0)=u_0$ and $g(1)=u_1$.
\vspace*{1mm}
\item $\int_0^s g(t)\vartheta(t)\, dt\in \Omega$ for all $s\in I$.
\vspace*{1mm}
\item $\int_0^1 g(t)\vartheta(t)\, dt = x.$
\end{enumerate}
\end{lemma}
\begin{proof}
Set $I_0:=[0,\frac12]$ and choose any continuous nowhere flat map $g_0\colon I_0\to\Sgot_*$ such that 
\begin{equation}\label{eq:g0}
      g_0(0)=u_0,\quad \int_0^s g_0(t)\vartheta(t)\, dt\in \Omega\text{ for all $s\in I_0$}. 
\end{equation}
Such a map can be constructed as follows. For any $0<\delta<\frac12$ let $f_\delta\colon I_0\to [\delta,1]$ be the continuous map given by $f_\delta(s)=1-\frac{1-\delta}{\delta}s$ for $s\in [0,\delta]$ and $f_\delta(s)=\delta$ for $s\in [\delta,\frac12]$. Choose any continuous nowhere flat map $\wt g_0\colon I_0\to \Sgot_*$ with $\wt g_0(0)=u_0$. Then $g_0:=f_\delta \wt g_0\colon I_0\to\Sgot_*$ satisfies the requirements for any $\delta>0$ sufficiently small. 

Let $\emptyset \neq I'\subset \mathring I_0$ be a closed subinterval. Thus, Lemma \ref{lem:path-sprays} applied to $g_0$ provides continuous functions $h_1,\ldots,h_N\colon I\to\c$, with support on $I'$, and a neighborhood $U$ of the origin in $\c^N$, such that the period map
\[
    U\ni\zeta\mapsto \Pcal(\zeta)=\int_0^{\frac12} \phi^1_{\zeta_1h_1(t)}\circ \cdots \circ \phi^N_{\zeta_Nh_N(t)}(g_0(t))\vartheta(t)\, dt,\quad     
    \zeta=(\zeta_1,\ldots,\zeta_N)\in\c^N,
\]
has maximal rank equal to $n$ at $\zeta=0$. (See \eqref{eq:CartanA1}.) Set
\[
    \Phi(\zeta,t):=\phi^1_{\zeta_1h_1(t)}\circ \cdots \circ \phi^N_{\zeta_Nh_N(t)}(g_0(t))\in\Sgot,\quad \zeta\in U,\; t\in I_0,
\]
and observe that $\Phi(0,t)=g_0(t)\in\Sgot_*$ for all $t\in I_0$. Then, up to shrinking $U$ if necessary, we have that:
\begin{enumerate}[\rm (a)]
\item $\Phi(U\times I_0)\subset \Sgot_*$ and $\Pcal(U)$ contains a ball in $\c^n$ with radius $\epsilon>0$ centered at $\Pcal(0)=\int_0^{\frac12} g_0(t)\vartheta(t)\, dt\in\Omega$, see \eqref{eq:g0}.
\vspace*{1mm}
\item $\Phi(\zeta,t)=g_0(t)$ for all $(\zeta,t)\in U\times\{0,\frac12\}$; recall that $h_j(0)=h_j(\frac12)=0$ for all $j=1,\ldots,N$.
\vspace*{1mm}
\item $\int_0^s \Phi(\zeta,t)\vartheta(t)\, dt\in\Omega$ for all $\zeta\in U$ and $s\in I_0$; see \eqref{eq:g0}.
\end{enumerate}

To conclude the proof we adapt the argument in \cite[Lemma 7.3]{AlarconForstneric2014IM}. Since the convex hull of $\Sgot$ is $\c^n$ (cf.\ \cite[Lemma 3.1]{AlarconForstneric2014IM}) we may construct a polygonal path $\Gamma\subset\Omega$ connecting $\Pcal(0)$ and $x$; to be more precise, $\Gamma=\bigcup_{j=1}^m \Gamma_j$ where each $\Gamma_j$ is a segment of the form $\Gamma_j= w_j+[0,1]z_j$ for some $w_j\in\c^n$ and $z_j\in \Sgot_*$, the initial point $w_1$ of $\Gamma_1$ is $\Pcal(0)$, the final point $w_m+z_m$ of $\Gamma_m$ is $x$, and the initial point $w_j$ of $\Gamma_j$ agrees with the final one $w_{j-1}+z_{j-1}$ of $\Gamma_{j-1}$ for all $j=2,\ldots,m$. 
Set
\[
     I_j:=\left[\frac12+\frac{j-1}{2m},\frac12+\frac{j}{2m}\right],\quad j=1,\ldots,m,
\] 
and observe that $\bigcup_{j=1}^m I_j=[\frac12,1]$.
For any number $0<\lambda<\frac1{4m}$, set 
\[
     I_j^\lambda:=\left[\frac12+\frac{j-1}{2m}+\lambda,\frac12+\frac{j}{2m}-\lambda\right]\subset I_j,\quad j=1,\ldots,m.
\]
Without loss of generality we may assume that $m\in\n$ is large enough so that 
\begin{equation}\label{eq:bj}
      b_j(\lambda):=\int_{I_j^\lambda}%{\frac12+\frac{j-1}{2m}+\lambda}^{\frac12+\frac{j}{2m}-\lambda}
      \vartheta(t)\, dt\neq 0\quad \text{for all $0<\lambda<\frac1{4m}$,\; $j=1,\ldots,m$};
\end{equation}
recall that $\vartheta$ has no zeroes. Fix a number $0<\lambda<\frac1{4m}$ and set $b_j:=b_j(\lambda)$. Pick a constant $\kappa>\max\{|u_0|,|u_1|,|z_1/b_1|,\ldots,|z_m/b_m|\}$.  Also choose numbers $0<\tau<\mu<\lambda$, which will be specified later, and
consider a continuous map $g_1\colon [\frac12,1]\to\Sgot_*$ satisfying the following conditions:
\begin{enumerate}[\rm (a)]
\item[\rm (d)] $g_1(\frac12)=g_0(\frac12)$ and $g_1(1)=u_1$.
\item[\rm (e)] $g_1(t)=z_j/b_j$ for all $t\in I_j^\lambda$.
\item[\rm (f)] $|g_1(t)|\le \kappa$ for all $t\in [\frac12,1]$.
\item[\rm (g)] $|g_1(t)|\le \tau$ for all $t\in I_j^\tau\setminus I_j^\mu$.
\end{enumerate} 
If $\tau>0$ is chosen sufficiently small, and if $\mu$ is close enough to $\lambda$, then  {\rm (e)}, {\rm (f)}, {\rm (g)}, and \eqref{eq:bj} ensure that 
\begin{enumerate}[\rm (a)]
\item[\rm (h)] the image of the map $[\frac12,1]\ni s\mapsto \Pcal(0)+\int_{\frac12}^s g_1(t)\vartheta(t)\, dt$ is close enough to $\Gamma$ in the Hausdorff distance so that it lies in $\Omega$ and
\item[\rm (i)] $|\Pcal(0)+\int_{\frac12}^1 g_1(t)\vartheta(t)\, dt-x|<\epsilon$, where $\epsilon>0$ is the number appearing in {\rm (a)}.
\end{enumerate}

For $\zeta\in U$, let $g^\zeta\colon I\to\Sgot_*$ denote the function given by $g^\zeta(t)=\Phi(\zeta,t)$ for $t\in [0,\frac12]$ and $g^\zeta(t)=g_1(t)$ for $t\in [\frac12,1]$. Properties {\rm (a)} and {\rm (i)} guarantee the existence of $\zeta_0\in U$ such that $\int_0^{\frac12} g^{\zeta_0}(t)\vartheta(t)\, dt=x-\int_{\frac12}^1 g_1(t)\vartheta(t)\, dt$, and so $\int_0^1 g^{\zeta_0}(t)\vartheta(t)\, dt=x$. Thus $g:=g^{\zeta_0}$ meets {\it iii)}. By \eqref{eq:g0}, {\rm (b)}, and {\rm (d)}, we have that $g$ is continuous and satisfies {\it i)}, whereas {\rm (c)} and {\rm (h)} ensure {\it ii)}. This concludes the proof.
\end{proof}

%%%%%%%%%%
%%%%%%%%%%
%%%%%%%%%%
%%%%%%%%%%   Subsection: On a Riemann surface
%%%%%%%%%%
%%%%%%%%%%

\subsection{Paths on open Riemann surfaces}\label{sec:pathsRiemann}
 
Let us now state and prove the main result of this section; recall that we are using Notation \ref{not:}.

\begin{lemma}\label{lem:pathRiemann}
Let $M$ be an open Riemann surface and let $\theta$ be a holomorphic $1$-form vanishing nowhere on $M$. Let $p_0\in M$ be a point, $C_1,\ldots,C_l$ $(l\in\n)$ be a family of oriented Jordan arcs or closed curves in $M$ that only meet at $p_0$ (i.e. $C_i\cap C_j=\{p_0\}$ for all $i\neq j\in \{1,\ldots,l\}$) and such that $C:=\bigcup_{i=1}^l C_i$ is Runge in $M$.
Also let $f\colon C\to\Sgot_*$ be a continuous map and assume that for each $i\in\{1,\ldots,l\}$ there exists a subarc $\wt C_i\subset C_i$ such that $f$ is nowhere flat on $\wt C_i$.  
Then there exist continuous functions $h_{i,1},\ldots,h_{i,N}\colon C\to\c$, with support on $\wt C_i$, $i=1,\ldots,l$, and a neighborhood $U$ of $0\in(\c^N)^l$ such that the {\em period map} $U\to(\c^n)^l$ whose $i$-th component $U\to\c^n$ is given by
\[
    U\ni\zeta\mapsto  \int_{C_i}\phi^1_{\zeta_1^1h_{1,1}(p)}\circ \cdots \circ \phi^N_{\zeta_N^1h_{1,N}(p)}\circ\cdots\circ \phi^1_{\zeta_1^lh_{l,1}(p)}\circ \cdots \circ \phi^N_{\zeta_N^lh_{l,N}(p)}(f(p)) \theta
\]
(see \eqref{eq:CartanA} and \eqref{eq:CartanA1}), where
\[
     \zeta=(\zeta^1,\ldots,\zeta^l)\in (\c^N)^l,\quad \zeta^i=(\zeta^i_1,\ldots,\zeta^i_N)\in\c^N,
\]
are holomorphic coordinates, is well defined and has maximal rank equal to $nl$ at $\zeta=0$.
\end{lemma}

\begin{proof}

Consider the period map $\Pcal=(\Pcal_1,\ldots,\Pcal_l)\colon \Cscr(C,\c^n)\to (\c^n)^l$ whose $i$-th component is defined by
\begin{equation}\label{eq:PerLem}
\Cscr(C,\c^n)\ni g\mapsto \Pcal_i(g)=\int_{C_i}g\theta,\quad i=1,\dots,l.
\end{equation}
%To prove the lemma, it suffices to find a function $\wh f$, of class $\Ascr(L)$, which is as close as desired to $f$ uniformly on $S$ and satisfies $\Pcal(\wh f)=\Pcal(f)$.

For each $i=1,\ldots,l$, let $\gamma_i\colon I=[0,1]\to C_i$ be a smooth parameterization of $C_i$ such that $\gamma_i(0)=p_0$. If $C_i$ is closed then we choose $\gamma_i$ with $\gamma_i(1)=p_0$; further, up to changing the orientation of $C_i$ if necessary, we assume that the parameterization $\gamma_i$ is compatible with the orientation of $C_i$. Thus,
\begin{equation}\label{eq:PerLemPa}
\Pcal_i(g)=\int_0^1g(\gamma_i(t)) \theta (\gamma_i(t),\dot \gamma_i(t)) \, dt,\quad g\in \Cscr(C,\c^n).
\end{equation} 

Let $\emptyset\neq I_i\subset \mathring I$ be a closed interval such that $\gamma_i(I_i)\subset\wt C_i$. Lemma \ref{lem:path-sprays} applied to $I_i$, $f\circ \gamma_i$, and $\theta(\gamma_i(\cdot),\dot\gamma_i(\cdot))$ provides continuous functions $h_1^i,\ldots,h_N^i\colon I\to\c$, supported on $I_i$, and a neighborhood $U_i$ of $0\in \c^N$ such that the period map $\Psf_i\colon U_i\to\c^n$ given, for any $\zeta^i=(\zeta_1^i,\ldots,\zeta_{N_i}^i)\in U_i$, by
\begin{equation}\label{eq:Psfi}
      \Psf_i(\zeta^i)=\int_0^1 \phi^1_{\zeta_1^ih_1^i(t)}\circ \cdots \circ \phi^N_{\zeta_N^ih_N^i(t)}(f(\gamma_i(t)))\theta(\gamma_i(t),\dot\gamma_i(t))\, dt
\end{equation}
(see \eqref{eq:CartanA} and \eqref{eq:CartanA1}),
is well defined and has maximal rank equal to $n$ at $\zeta^i=0$. Let $U$ be a ball centered at the origin of $(\c^N)^l$ and contained in $U_1\times\cdots\times U_l$.
% Let $\zeta=(\zeta^1,\ldots,\zeta^l)$ be holomorphic coordinates on $(\c^N)^l$, where $\zeta^i=(\zeta_1^i,\ldots,\zeta_N^i)$ are holomorphic coordinates on $\c^N$ for all $i=1,\ldots,l$. 
For each $i\in\{1,\ldots, l\}$ and $j=1,\ldots, N$, we define $h_{i,j}\colon C\to\c$ by $h_{i,j}(\gamma_i(t))=h^i_j(t)$ for all $t\in I$, and $h_{i,j}(p)=0$ for all $p\in C\setminus C_i$. Recall that $h_j^i(0)=0$ and so $h_{i,j}$ is continuous and $h_{i,j}(p_0)=0$.
Define $\Phi\colon U\times C\to\Sgot$ by
\[%\begin{equation}\label{eq:Phi}
\Phi(\zeta,p)=\phi^1_{\zeta_1^1h_{1,1}(p)}\circ \cdots \circ \phi^N_{\zeta_N^1h_{1,N}(p)}\circ\cdots\circ \phi^1_{\zeta_1^lh_{l,1}(p)}\circ \cdots \circ \phi^N_{\zeta_N^lh_{l,N}(p)}(f(p)),
\]%\end{equation}
and, up to shrinking $U$ if necessary, assume that $\Phi(U\times C)\subset \Sgot_*$. 

Let $\Psf\colon U\to (\c^n)^l$ be the period map whose $i$-th component $U\to\c^n$, $i=1,\ldots,l$, is given by
\[
U\ni \zeta \mapsto \int_{C_i} \Phi(\zeta,\cdot) \theta =\Psf_i(\zeta^i),\quad \zeta=(\zeta^1,\ldots,\zeta^l)\in U;
\]
see \eqref{eq:Psfi} and recall that $h_{i,j}$ vanishes everywhere on $C\setminus C_i$. It follows that $\Psf$ has maximal rank equal to $nl$ at $\zeta=0$. This complete the proof.
%and hence the image by $\Psf$ of any open neighborhood of $0\in U\subset (\c^N)^l$ contains an open ball in $(\c^n)^l$ centered at $\Psf(0)=\Pcal(f)$; see \eqref{eq:PerLem}, \eqref{eq:PerLemPa}, and \eqref{eq:Psfi}.
%
%Since $S$ and $C$ are Runge subsets in $M$, Mergelyan theorem enables us to approximate $f$ and each function $h_{i,j}$ by holomorphic functions $\wt f$ and $\wt h_{i,j}$ of class $\Ocal(L)$. Define $\wt \Phi\colon U\times L\to\Sgot$ by 
%\[%\begin{equation}\label{eq:wtPhi}
%\wt\Phi(\zeta,p)=\phi^{1,1}_{\zeta_1^1\wt h_{1,1}(p)}\circ \cdots \circ \phi^{1,N}_{\zeta_N^1\wt h_{1,N}(p)}\circ\cdots\circ \phi^{l,1}_{\zeta_1^l \wt h_{l,1}(p)}\circ \cdots \circ \phi^{l,N}_{\zeta_N^l \wt h_{l,N}(p)}(\wt f(p)),
%\]%\end{equation}
%and consider the period map $\wt \Psf\colon U\to (\c^n)^l$ whose $i$-th component $U\to\c^n$ is given by
%\[
%U\ni \zeta \mapsto \int_{C_i} \wt \Phi(\zeta,\cdot) \theta, \quad \zeta \in U.
%\]
%Thus, for any open ball $0\in W\subset U$, if the approximation of $f$ and each $h_{i,j}$ by $\wt f$ and $\wt h_{i,j}$ is close enough, the range of $\wt\Psf(W)$ also contains $\Pcal(f)$. Therefore, there is $\zeta_0\in U$ close to $0\in(\c^N)^l$ such that $\wh f:=\wt\Phi(\zeta_0,\cdot)$ assumes values in $\Sgot_*$ and solves the lemma.
\end{proof}

%%%%%%%%%%
%%%%%%%%%%
%%%%%%%%%%
%%%%%%%%%%   Section: JET-INTERPOLATION WITH APPROXIMATION
%%%%%%%%%%
%%%%%%%%%%

\section{Jet-interpolation with approximation %and a local structure theorem
}\label{sec:jet}  

We begin this section with some preparations.

\begin{definition}\label{def:simple}
Let $M$ be an open Riemann surface.
An admissible subset $S=K\cup\Gamma\subset M$ (see Definition \ref{def:admissible}) will be said {\em simple} if $K\neq\emptyset$, every component of $\Gamma$ meets $K$, $\Gamma$ does not contain closed Jordan curves, and every closed Jordan curve in $S$ meets only one component of $K$. Further, $S$ will be said {\em very simple} if it is simple, $K$ has at most one non-simply connected component $K_0$, which will be called the {\em kernel component} of $K$, and every component of $\Gamma$ has at least one endpoint in $K_0$; in this case we denote by $S_0$ the component of $S$ containing $K_0$ and call it the {\em kernel component} of $S$.
\end{definition}

A connected admissible subset $S=K\cup\Gamma$ in an open Riemann surface $M$ is very simple if, and only if, $K$ has $m\in\n$  components $K_0,\ldots, K_{m-1}$, where $K_i$ is simply-connected for every $i>0$, and $\Gamma=\Gamma'\cup\Gamma''\cup(\bigcup_{i=1}^{m-1} \gamma_i)$ where $\Gamma'$ consists of components of $\Gamma$ with both endpoints in $K_0$, $\Gamma''$ consists of components of $\Gamma$ with an endpoint in $K_0$ and the other one in $M\setminus K$, and $\gamma_i$ is a component of $\Gamma$ connecting $K_0$ and $K_i$ for all $i=1,\ldots,m-1$. Observe that, in this case, $K_0\cup \Gamma'$ is a strong deformation retract of $S$.
In general, a very simple admissible subset $S\subset M$ is of the form $S=(K\cup\Gamma)\cup K'$ where $K\cup\Gamma$ is a connected very simple admissible subset and $K'\subset M\setminus (K\cup\Gamma)$ is a (possibly empty) union of finitely many pairwise disjoint smoothly bounded compact disks. (See Figure \ref{fig:admissibleverysimple}.)
\begin{figure}[ht]
	\includegraphics[width=12cm]{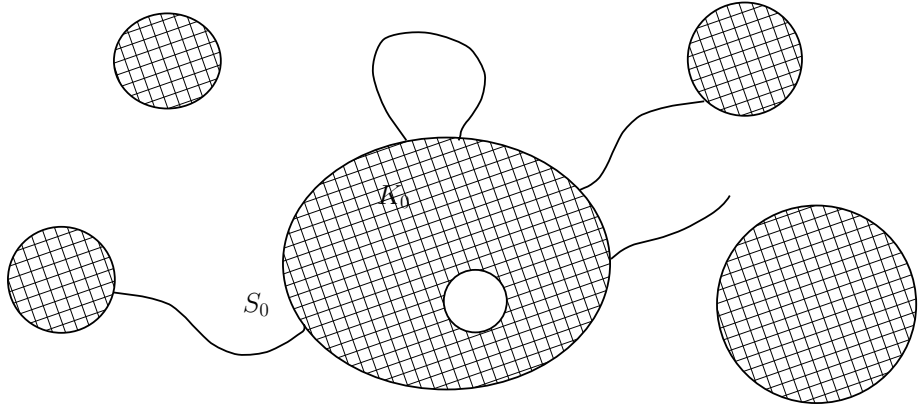}
	\caption{A very simple admissible set.}\label{fig:admissibleverysimple}
\end{figure}

If $S=K\cup\Gamma\subset M$ is admissible, we denote by $\Ascr(S)$ the space of continuous functions $S\to\c$ which are of class $\Ascr(K)$. Likewise, we define the space $\Ascr(S,Z)$ for maps to any complex manifold $Z$.

In the remainder of this section we use Notation \ref{not:}.
%
% Noncritical case
%
\begin{lemma}\label{lem:main-noncritical}%[Noncritical case]
Let $M$ be an open Riemann surface and $\theta$ be a holomorphic $1$-form vanishing nowhere on $M$. Let $S=K\cup\Gamma\subset M$ be a very simple admissible subset and $L\subset M$ be a smoothly bounded compact domain such that $S\subset\mathring L$ and the kernel component $S_0$ of $S$ is a strong deformation retract of $L$ (see Definition \ref{def:simple}). Denote by $l'\in\z_+$ the dimension of the first homology group $H_1(S;\z)=H_1(S_0;\z)\cong H_1(L;\z)$.
Let $K_0,\ldots, K_m$, $m\in\z_+$ denote the components of $K$ contained in $S_0$, where $K_0$ is the kernel component of $K$. 

Let $m'\in\z_+$, $m'\ge m$, and let $p_0,\ldots, p_{m'}$ be distinct points in $S$ such that $p_i\in \mathring K_i$ for all $i=0,\ldots,m$ and $p_i\in\mathring K_0$ for all $i=m+1,\ldots, m'$, and let $C_i$, $i=1,\ldots,m'$, be pairwise disjoint oriented Jordan arcs in $S$ with initial point $p_0$ and final point $p_i$.  Set $l:=l'+m'$. Also let $C_i$, $i=m'+1,\ldots,l$, be smooth Jordan curves in $S$ determining a homology basis of $S$ and such that $C_i\cap C_j=\{p_0\}$ for all $i\neq j\in\{1,\ldots,l\}$ and $C:=\bigcup_{i=1}^l C_i$ is Runge in $M$. (See Figure \ref{fig:admissiblepathperiods}.)

Given $k\in\n$ and a map $f\colon S\to\Sgot_*\subset\c^n$ of class $\Ascr(S)$ which is nonflat on $\mathring K_0$ (see Definition \ref{def:degenerateI}), the following hold:
\begin{enumerate}[\it i)]
\item There exist functions $h_{i,1},\ldots,h_{i,N}\colon L\to\c$, $i=1,\ldots,l$, of class $\Ascr(L)$ and a neighborhood $U$ of $0\in(\c^N)^l$ such that: 
\begin{enumerate}[\it {i}.1)]
\item $h_{i,j}$ has a zero of multiplicity $k$ at $p_r$ for all $j=1,\ldots,N$ and $r=1,\ldots,m'$.
\item Denoting by $\Phi_f\colon U\times S\to\Sgot$ the map
\[
      \Phi_f(\zeta,p)=\phi^1_{\zeta_1^1h_{1,1}(p)}\circ \cdots \circ \phi^N_{\zeta_N^1h_{1,N}(p)}\circ\cdots\circ \phi^1_{\zeta_1^lh_{l,1}(p)}\circ \cdots \circ \phi^N_{\zeta_N^lh_{l,N}(p)}(f(p)),
\]
(see \eqref{eq:CartanA} and \eqref{eq:CartanA1}), where
$\zeta=(\zeta^1,\ldots,\zeta^l)\in (\c^N)^l$ and $\zeta^i=(\zeta^i_1,\ldots,\zeta^i_N)\in\c^N$,
are holomorphic coordinates, the {\em period map} $U\to(\c^n)^l$ whose $i$-th component $U\to\c^n$ is given by
\[
    U\ni\zeta\mapsto  \int_{C_i}\Phi_f(\zeta,\cdot) \theta
\]
has maximal rank equal to $nl$ at $\zeta=0$.
\end{enumerate}
Furthermore, there is a neighborhood $V$ of $g\in \Ascr(S,\Sgot_*)$ such that the map $V\ni g\mapsto\Phi_g$ can be chosen to depend holomorphically on $g$.

\medskip

\item If $\Sgot_*$ is an Oka manifold, then $f$ may be approximated uniformly on $S$ by maps $\wt f\colon L\to\Sgot_*$ of class $\Ascr(L)$ such that:
\begin{enumerate}[\it {ii}.1)]
\item $(\wt f-f)\theta$ is exact on $S$, equivalently, $\displaystyle \int_{C_r} (\wt f-f)\theta=0$ for all $r=m'+1,\ldots,l$.
\item $\displaystyle \int_{C_r} (\wt f-f)\theta=0$ for all $r=1,\ldots,m'$.
\item $\wt f-f$ has a zero of multiplicity $k$ at $p_r$ for all $r=1,\ldots,m'$.
\item No component function of $\wt f$ vanishes everywhere on $M$.
\end{enumerate}
\end{enumerate}
\end{lemma}
\begin{figure}[ht]
	\includegraphics[width=12cm]{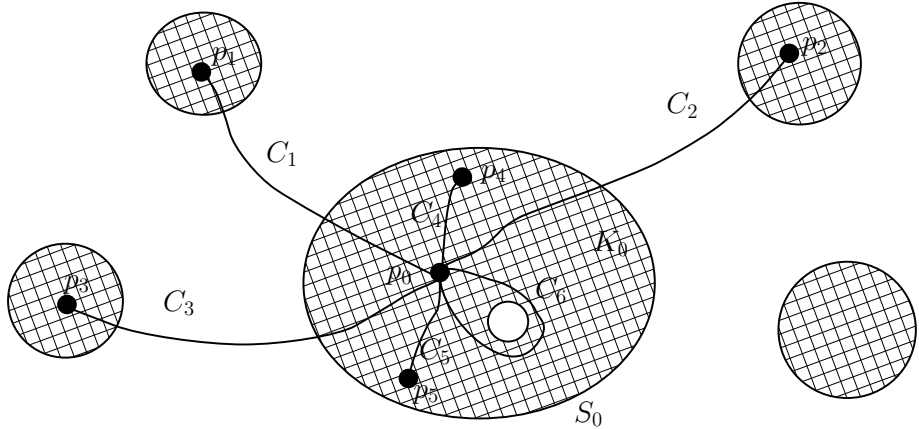}
	\caption{The sets in Lemma \ref{lem:main-noncritical}}\label{fig:admissiblepathperiods}
\end{figure}
Notice that conditions {\it {ii}.1)} and {\it {ii}.2)} in the above lemma may be written as a single one in the form
\[
\int_{C_r} (\wt f-f)\theta=0\quad \text{for all $r=1,\ldots,l$}.
\]
However, we write them separately with the aim of emphasizing that they are useful for different purposes; indeed, {\it {ii}.1)} concerns the period problem whereas {\it {ii}.2)} deals with the problem of interpolation. %These properties will be essential for the following results.
\begin{proof}
Choose $k\in\n$ and let $f\colon S\to\Sgot_*$ be a map of class $\Ascr(S)$ which is nonflat on $K_0$.
Consider the period map $\Pcal=(\Pcal_1,\ldots,\Pcal_l)\colon \Cscr(C,\c^n)\to (\c^n)^l$ whose $i$-th component $\Pcal_i\colon \Cscr(C,\c^n)\to \c^n$ is defined by
\begin{equation}\label{eq:Periods}
\Cscr(C,\c^n)\ni g\mapsto \Pcal_i(g)=\int_{C_i}g\theta,\quad i=1,\dots,l.
\end{equation}
%To prove the lemma, it suffices to find a function $\wh f$, of class $\Ascr(L)$, which is as close as desired to $f$ uniformly on $S$, satisfies $\Pcal(\wh f)=\Pcal(f)$, and verifies that $\wt f-f$ has a zero of degree $k$ at $p_i$ for all $i=1,\ldots,m'$.

Since $S$ is very simple and $f$ is holomorphic and nonflat on $\mathring K_0$, each $C_i$, $i=1,\ldots,l$, contains a subarc $\wt C_i\subset\mathring K_0\setminus\{p_0\}$ such that $f$ is nowhere flat on $\wt C_i$; if $i\in\{m+1,\ldots,m'\}$ then we may choose $\wt C_i\subset C_i\setminus\{p_0,p_i\}$. Thus, Lemma \ref{lem:pathRiemann} applied to the map $f|_C\colon C\to\Sgot_*$, the base point $p_0$, and the curves $C_1,\ldots,C_l$ furnishes functions $g_{i,1},\ldots,g_{i,N}\colon C\to\c$, with support on $\wt C_i$, $i=1,\ldots,l$, and a neighborhood $U$ of $0\in (\c^N)^l$ such that the period map $\Psf\colon U\to(\c^n)^l$ whose $i$-th component $\Psf_i\colon U\to\c^n$ is given by
\[
    \Psf_i(\zeta):=  \int_{C_i}\phi^1_{\zeta_1^1g_{1,1}(p)}\circ \cdots \circ \phi^N_{\zeta_N^1g_{1,N}(p)}\circ\cdots\circ \phi^1_{\zeta_1^lg_{l,1}(p)}\circ \cdots \circ \phi^N_{\zeta_N^lg_{l,N}(p)}(f(p))\, \theta
\]
(see \eqref{eq:CartanA} and \eqref{eq:CartanA1}), is well defined and has maximal rank equal to $nl$ at $\zeta=0$. Since $C\subset M$ is Runge, Theorem \ref{th:MTJI} enables us to approximate each $g_{i,j}$ by functions $h_{i,j}\in\Ocal(M)\subset\Ascr(L)\subset\Ascr(S)$ satisfying condition {\it i.1)}; recall that every function $g_{ij}$ vanishes on a neighborhood of $p_r$ for all $r=1,\ldots,m'$. Furthermore, if the approximation of $g_{i,j}$ by $h_{i,j}$ is close enough then the period map defined in {\it i.2)} also has maximal rank at $\zeta=0$. Finally, by varying $f$ locally (keeping the functions $h_{i,j}$ fixed) we obtain a holomorphic family of maps $f\mapsto \Phi_f$ with the desired properties. This proves {\it i)}.

Let us now prove assertion {\it ii)}, so assume that $\Sgot_*$ is an Oka manifold. Up to adding to $S$ a smoothly bounded compact disk $D\subset M\setminus S$ and extend $f$ to $D$ as a function of class $\Ascr(S)$ all whose components are different from the constant $0$ on $D$, we may assume that no component function of $f$ vanishes everywhere on $S$. Consider the map $\Phi\colon U\times S\to\Sgot$ given in {\it i.2)} and, up to shrinking $U$ if necessary, assume that $\Phi(U\times S)\in\Sgot_*$.
Note that the functions $h_{i,j}$ are defined on $L$ but $f$ only on $S$. By {\it i)}, the period map $\Qcal\colon U\to(\c^n)^l$ with $i$-th component
\[
    \Qcal_i(\zeta)=\int_{C_i}\Phi(\zeta,\cdot)\theta=\Pcal_i(\Phi(\zeta,\cdot)),\quad \zeta\in U
\]
(see \eqref{eq:Periods}), has maximal rank equal to $nl$ at $\zeta=0$. It follows that the image by $\Qcal$ of any open neighborhood of $0\in U\subset (\c^N)^l$ contains an open ball in $(\c^n)^l$ centered at $\Qcal(0)=\Pcal(f)$; see \eqref{eq:Periods}. Since $S\subset M$ is Runge and $\Sgot_*$ is Oka, Theorem \ref{th:MTJI} allows us to approximate $f$ by holomorphic maps $\wh f\colon M\to\Sgot_*$ such that 
\begin{equation}\label{eq:whf-f}
     \text{$\wh f-f$ has a zero of multiplicity $k$ at $p_r$ for all $r=1,\ldots,m'$.}
\end{equation}
Define $\wh \Phi\colon U\times L\to\Sgot$ by 
\begin{equation}\label{eq:whPhi}
\wh\Phi(\zeta,p)=\phi^1_{\zeta_1^1h_{1,1}(p)}\circ \cdots \circ \phi^N_{\zeta_N^1h_{1,N}(p)}\circ\cdots\circ \phi^1_{\zeta_1^l h_{l,1}(p)}\circ \cdots \circ \phi^N_{\zeta_N^l h_{l,N}(p)}(\wh f(p))
\end{equation}
and, up to shrinking $U$ once again if necessary, assume that $\wh\Phi(U\times L)\subset\Sgot_*$.
Consider now the period map $\wh \Qcal\colon U\to (\c^n)^l$ whose $i$-th component $U\to\c^n$ is given by
\[
U\ni \zeta \mapsto \int_{C_i} \wh \Phi(\zeta,\cdot)\, \theta,\quad i=1,\ldots,l.
\]
Thus, for any open ball $0\in W\subset U$, if the approximation of $f$ by $\wh f$ is close enough, the range of $\wh\Qcal(W)$ also contains $\Pcal(f)$. Therefore, there is $\zeta_0\in W\subset U$ close to $0\in(\c^N)^l$ such that 
\begin{equation}\label{eq:wtfPhi}
\wt f:=\wh\Phi(\zeta_0,\cdot)\colon L\to\Sgot_*
\end{equation}
lies in $\Ascr(L)$ and satisfies {\it ii.1)} and {\it ii.2)}; recall that $S_0$ is a strong deformation retract of $L$ and so the curves $C_i$, $i=m'+1,\ldots,l$, determines a basis of $H_1(L;\z)$. To finish the proof, Lemma \ref{lem:jet},  {\it i.1)}, \eqref{eq:whPhi}, and \eqref{eq:wtfPhi} guarantee that $\wt f-\wh f$ has a zero of multiplicity (at least) $k$ at $p_r$ for all $r=1,\ldots,m'$. This and \eqref{eq:whf-f} ensure {\it ii.3)}. Finally, if the approximation of $f$ by $\wt f$ on $S$ is close enough, since no component function of $f$ vanishes everywhere on $S$, then no component function of $\wt f$ vanishes everywhere on $M$, which proves {\it ii.4)} and concludes the proof.
\end{proof}

We now show the following technical result which will considerably simplify the subsequent proofs.

%
% Flat -> Nonflat
%

\begin{proposition}\label{pro:nonflat}
Let $n\ge 3$ be an integer and $\Sgot$ be an irreducible closed conical complex subvariety of $\c^n$ which is not contained in any hyperplane.
Let $M=\mathring M\cup bM$ be a compact bordered Riemann surface, $\theta$ be a holomorphic $1$-form vanishing nowhere on $M$, and $\Lambda\subset\mathring M$ be a finite subset. Choose $p_0\in M\setminus\Lambda$ and, for each $p\in\Lambda$, let $C_p\subset\mathring M$ be a smooth Jordan arc with initial point $p_0$ and final point $p$ such that $C_p\cap C_q=\{p_0\}$ for all $p\neq q\in\Lambda$.

Let $f:M\to\Sgot_*$ be a map of class $\Ascr(M)$ which is flat (see Definition \ref{def:degenerateI}) and $k\in\n$ be an integer.
Then $f$ may be approximated uniformly on $M$ by nonflat maps $\wt f:M\to\Sgot_*$ of class $\Ascr(M)$ satisfying the following properties.
\begin{enumerate}[\rm (i)]
\item $(\wt f-f)\theta$ is exact on $M$.
\item $\displaystyle \int_{C_p}(\wt f- f)\theta = 0$ for all $p\in\Lambda$.
\item $\wt f-f$ has a zero of multiplicity $k$ at all points $p\in\Lambda$.
\end{enumerate}
%Flat $\to$ Nonflat on a compact bordered Riemann surface 
\end{proposition}
\begin{proof}
Without loss of generality we assume that $\Lambda\neq \emptyset$, write $\Lambda=\{p_1,\ldots,p_{l'}\}$, and set $C_i:=C_{p_i}$, $i=1,\ldots,l'$. Choose $C_{l'+1},\ldots,C_{l}$ closed Jordan loops in $\mathring M$ forming a basis of $H_1(M,\z)\cong\z^{l-l'}$ such that $C_i\cap C_j=\{p_0\}$ for all $i,j\in\{1,\ldots,l\}$, $i\neq j$, and $C:=\bigcup_{j=1}^l C_j$ is a Runge subset of $M$; existence of such is ensured by basic topological arguments. Consider smooth parameterizations $\gamma_j\colon [0,1]\to C_j$ of the respective curves verifying $\gamma_j(0)=p_0$ and $\gamma_j(1)=p_j$ for $j=1,\ldots,l'$, and $\gamma_j(0)=\gamma_j(1)=p_0$ for $j=l'+1,\ldots,l$.

Since $f$ is flat there exists $z_0\in\Sgot_*$ such that $f(M)\subset \c_* z_0$. Observe that $\c z_0$ is a proper complex subvariety of $\Sgot$.
We consider the period map $\Pcal=(\Pcal_1,\ldots,\Pcal_l):\Ascr(M)\to(\c^n)^l$ defined by
\begin{equation}\label{eq:periodsss}
\Ascr(M)\ni g\mapsto\Pcal_j(g):=\int_{C_j}g\theta=\int_0^1 g(\gamma_j(t))\theta(\gamma_j(t),\dot\gamma_j(t))\, dt,\quad j=1,\ldots,l.
\end{equation}
Note that a map $g\in\Ascr(M)$ meets {\rm (i)} and {\rm (ii)} if, and only if, $\Pcal(g)=\Pcal(f)$. So, to finish the proof it suffices to approximate $f$ uniformly on $M$ by nonflat maps $\wt f\in\Ascr(M)$ satisfying the latter condition and also {\rm (iii)}.

Choose a holomorphic vector field $V$ on $\c^n$ which is tangential to $\Sgot$ along $\Sgot$, vanishes at $0$, and is not everywhere tangential to $\c_*z_0$ along $f(M)$. Let $\phi_s(z)$ denote the flow of $V$ for small values of time $s\in\c$. Choose a nonconstant function $h_1\colon M\to\c$ of class $\Ascr(M)$ such that $h_1(p_0)=0$. Denote by $\mho$ the space of all functions $h\colon M\to\c$ of class $\Ascr(M)$ having a zero of multiplicity $k\in \n$ at all points $p\in\Lambda$. The following map is well-defined and holomorphic on a small open neighborhood $\mho^*$ of the zero function in $\mho$: 
\[
    \mho^*\ni h \mapsto \Pcal(\phi_{h_1(\cdot)h(\cdot)}(f(\cdot)))\in(\c^n)^l.
\]
Each component $\Pcal_j$, $j=1,\ldots,l$, of this map at the point $h=0$ equals 
\[
    \Pcal_j(\phi_0(f))=\Pcal_j(f)
\] 
(recall that $V$ vanishes at $0\in\c^n$). Since $\mho$ is infinite dimensional, there is a function $h\in\mho$ arbitrarily close to the function $0$ (in particular, we may take $h\in\mho^*$) and nonconstant on $M$, such that
\[
\Pcal(\phi_{h_1(\cdot)h(\cdot)}(f(\cdot)))=\Pcal(f).
\]

Set $\wt f(p)=\phi_{h_1(p)h(p)}(f(p))$, $p\in M$. Assume that $\|h\|_{0,M}$ is sufficiently small so that $\wt f$ is well defined and of class $\Ascr(M)$,  $\wt f$ approximates $f$ on $M$, and $\wt f(p)\in\Sgot_*$ for all $p\in M$. By the discussion below equation \eqref{eq:periodsss}, $\wt f$ meets {\rm (i)} and {\rm (ii)}. On the other hand, since $h$ has a zero of multiplicity $k$ at every point of $\Lambda$ and $h_1$ is not constant, we infer that $hh_1$ also has a zero of multiplicity (at least) $k$ at all points of $\Lambda$. Thus, Lemma \ref{lem:jet} ensures that $\wt f-f$ satisfies {\rm (iii)}. Finally, since $h_1(p_0)=0$ and $V$ vanishes at $0$, we have that $\wt f(p_0)=f(p_0)\in \c_*z_0$, whereas since $h h_1$ is nonconstant on $M$ and $V$ is not everywhere tangential to $\c_* z_0$ along $f(M)$, there is a point $q\in M$ such that $\wt f(q)\notin \c_* z_0$. This proves that $\wt f$ is nonflat, which concludes the proof.
\end{proof}

%
% Main Theorem
%

The following is the main technical result of this paper.

\begin{theorem}\label{th:MtT}
Let $n\ge 3$ be an integer and $\Sgot$ be an irreducible closed conical complex subvariety of $\c^n$ which is not contained in any hyperplane and such that $\Sgot_*=\Sgot\setminus\{0\}$ is smooth and an Oka manifold.
Let $M$ be an open Riemann surface, $\theta$ be a holomorphic $1$-form vanishing nowhere on $M$, $\Kcal\subset M$ be a smoothly bounded Runge compact domain, and $\Lambda\subset M$ be a closed discrete subset. Choose $p_0\in \mathring \Kcal\setminus \Lambda$ and, for each $p\in\Lambda$, let $C_p\subset M$ be an oriented Jordan arc with initial point $p_0$ and final point $p$ such that $C_p\cap C_q=\{p_0\}$ for all $q\neq p\in\Lambda$ and $C_p\subset \Kcal$ for all $p\in\Lambda\cap \Kcal$. Also, for each $p\in\Lambda$, let $\Omega_p\subset M$ be a compact neighborhood of $p$ in $M$ such that $\Omega_p\cap (\Omega_q\cup C_q)=\emptyset$ for all $q\neq p\in \Lambda$. Set $\Omega:=\bigcup_{p\in\Lambda} \Omega_p$.

Let $f\colon \Kcal\cup\Omega \to\Sgot_*$ be a map of class $\Ascr(\Kcal\cup\Omega)$, let $\qgot\colon H_1(M;\z)\to\c^n$ be a group homomorphism, and let $\Zgot\colon \Lambda\to\c^n$ be a map, such that:
\begin{enumerate}[\rm (a)]
\item $\int_\gamma f\theta=\qgot(\gamma)$ for all closed curves $\gamma\subset \Kcal$.
%\item $f=g$ on $\Kcal\cap\Omega$.
\item $\int_{C_p} f\theta =\Zgot(p)$ for all $p\in\Lambda\cap \Kcal$.
\end{enumerate}
Then, for any integer $k\in\n$, $f$ may be approximated uniformly on $\Kcal$ by holomorphic maps $\wt f\colon M\to\Sgot_*$ satisfying the following conditions:
\begin{enumerate}[\rm (A)]
\item $\int_\gamma \wt f\theta=\qgot(\gamma)$ for all closed curves $\gamma\subset M$.
\item $\wt f-f$ has a zero of multiplicity $k$ at $p$ for all $p\in\Lambda$; equivalently, $\wt f$ and $f$ have the same $(k-1)$-jet at every point $p\in\Lambda$.
\item $\int_{C_p} \wt f\theta =\Zgot(p)$ for all $p\in\Lambda$.
\item No component function of $\wt f$ vanishes everywhere on $M$.
\end{enumerate}
\end{theorem}
\begin{proof}
Up to slightly enlarging $\Kcal$ if necessary, we may assume without loss of generality that $\Lambda\cap b\Kcal=\emptyset$. Further, up to shrinking the sets $\Omega_p$, we may also assume that, for each $p\in\Lambda$, either $\Omega_p\subset\mathring \Kcal$ or $\Omega_p\cap\Kcal=\emptyset$. Finally, by Proposition \ref{pro:nonflat} we may assume that $f\colon \Kcal\to\Sgot_*$ is nonflat.

Set $M_0:=\Kcal$ and let $\{M_j\}_{j\in\n}$ be a sequence of smoothly bounded Runge compact domains in $M$ such that 
\[
M_0\Subset M_1\Subset M_2\Subset\cdots \Subset \bigcup_{j\in\n} M_j=M.
\]
Assume also that the Euler characteristic $\chi(M_j\setminus \mathring M_{j-1})$ of $M_j\setminus\mathring  M_{j-1}$ is either $0$ or $-1$, and that $\Lambda\cap bM_j=\emptyset$ for all $j\in\n$. Such a sequence can be constructed by basic topological arguments; see e.\ g.\ \cite[Lemma 4.2]{AlarconLopez2013JGA}. Since $\Lambda$ is closed and discrete, $M_j$ is compact, and $\Lambda\cap bM_j=\emptyset$ for all $j\in\z_+$, then $\Lambda_j:=\Lambda \cap M_j=\Lambda \cap \mathring M_j$ is either empty or finite. 
Without loss of generality we assume that $\Lambda_0\neq \emptyset$ and $\Lambda_j\setminus\Lambda_{j-1}=\Lambda\cap (\mathring M_j\setminus M_{j-1})\neq\emptyset$ for all $j\in\n$, and hence $\Lambda$ is infinite.

Set $f_0:=f|_\Kcal$ and, for each $p\in \Lambda_0\neq\emptyset$, choose an oriented Jordan arc $C^p\subset \mathring M_0$ with initial point $p$ and final point $p_0$, such that 
\begin{equation}\label{eq:C^p0}
    C^p\cap C^q=\{p_0\}\quad \text{for all $p\neq q\in \Lambda_0$.}
\end{equation}
Such curves trivially exist. 

 To prove the theorem we shall inductively construct a sequence of maps $f_j\colon M_j\to\Sgot_*\subset\c^n$ and a family of oriented Jordan arcs $C^p\subset \mathring M_j$, $p\in  \Lambda_j\setminus\Lambda_{j-1}\neq\emptyset$, $j\in\n$, with initial point $p$ and final point $p_0$, meeting the following properties:
\begin{enumerate}[\rm (i$_j$)]
\item $\|f_j-f_{j-1}\|_{0,M_{j-1}}<\epsilon_{j}$ for a certain constant $\epsilon_j>0$ which will be specified later.
\vspace*{1mm}
\item $\int_\gamma f_j\theta=\qgot(\gamma)$ for all closed curves $\gamma\subset M_j$.
\vspace*{1mm}
\item $\int_{C^p} f_j\theta =\qgot(C_p*C^p)-\Zgot(p)$ for all $p\in\Lambda_j$. (Recall that $*$ denotes the product of oriented arcs; see Subsec.\ \ref{ss:RS}.)
%\vspace*{1mm}
%\item $\int_{C_p} f_j\theta =\Zgot(p)$ for all $p\in\Lambda\cap M_0$.
\vspace*{1mm}
\item $f_j-f$ has a zero of multiplicity $k$ at $p$ for all $p\in\Lambda_j$.
\vspace*{1mm}
\item $C^p\cap C^q=\{p_0\}$ for all $p\neq q\in \Lambda_j$.
\vspace*{1mm}
\item No component function of $f_j$ vanishes everywhere on $M_j$.
\end{enumerate} 
(See Figure \ref{fig:admissibleOmega}.) 
\begin{figure}[ht]
	\includegraphics[width=12cm]{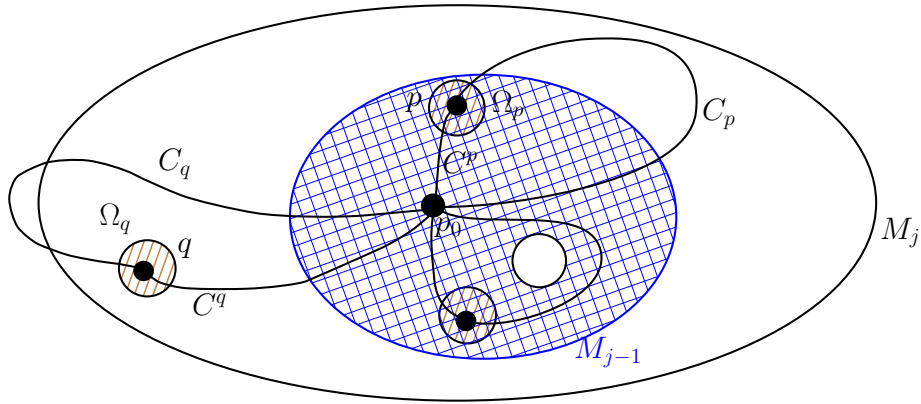}
	\caption{The set $\Kcal\subset M$, the arcs $C_p$, and the domains $\Omega_p$ in Theorem \ref{th:MtT}.}\label{fig:admissibleOmega}
\end{figure}
Assume for a moment that we have already constructed such sequence. Then choosing the sequence $\{\epsilon_j\}_{j\in\n}$ decreasing to zero fast enough, {\rm (i$_j$)} ensures that there is a limit holomorphic map 
\[
      \wt f:=\lim_{j\to\infty} f_j\colon M\to\Sgot_*
\]
which is as close as desired to $f$ uniformly on $\Kcal$,
whereas properties {\rm (ii$_j$)}, {\rm (iii$_j$)}, {\rm (iv$_j$)}, {\rm (v$_j$)}, and {\rm (vi$_j$)} guarantee {\rm (A)}, {\rm (B)}, {\rm (C)},  and {\rm (D)}. This would conclude the proof.

The basis of the induction is given by the nonflat map $f_0=f|_\Kcal$ and the already fixed oriented arcs $C^p$, $p\in\Lambda_0$. Condition {\rm (i$_0$)} is vacuous, {\rm (ii$_0$)}$=${\rm (a)}, {\rm (iii$_0$)} is implied by {\rm (a)} and {\rm (b)}, {\rm (iv$_0$)} is trivial, and {\rm (v$_0$)}$=$\eqref{eq:C^p0}. For the inductive step, we assume that we already have a map $f_{j-1}\colon M_{j-1}\to\Sgot_*$ and arcs $C^p\subset \mathring M_{j-1}$, $p\in \Lambda_{j-1}$, satisfying properties {\rm (ii$_{j-1}$)}--{\rm (v$_{j-1}$)} for some $j\in\n$, and let us construct a map $f_j$ and arcs $C^p$ for $p\in\Lambda_j\setminus \Lambda_{j-1}=\Lambda\cap (\mathring M_j\setminus M_{j-1})$, enjoying conditions {\rm (i$_j$)}--{\rm (vi$_j$)}. We distinguish cases depending on the Euler characteristic $\chi( M_j\setminus \mathring M_{j-1})$.

% % % %
% % % % Noncritical case
% % % %

\smallskip

\noindent{\em Case 1: The noncritical case. Assume that $\chi(M_j\setminus \mathring M_{j-1})=0$}. In this case $M_{j-1}$ is a strong deformation retract of $M_j$. Recall that $\Lambda_j\setminus \Lambda_{j-1}$ is a non-empty finite set. Choose, for each $p\in \Lambda_j\setminus \Lambda_{j-1}$, an oriented Jordan arc $C^p\subset \mathring M_j$ with initial point $p$ and final point $p_0$, so that condition {\rm (v$_j$)} holds; such arcs trivially exist. Up to shrinking $\Omega_p$ if necessary, we assume without loss of generality that $\Omega_p\subset \mathring M_j\setminus M_{j-1}$ for all $p\in \Lambda_j\setminus \Lambda_{j-1}$ and $\Omega_p\cap C^q=\emptyset$ for all $q\in \Lambda_j\setminus \Lambda_{j-1}$, $q\neq p$.

Set 
\[
    K:=M_{j-1}\cup\Big(\bigcup_{p\in\Lambda_j\setminus\Lambda_{j-1}}\Omega_p\Big),\quad\Gamma:=\Big(\bigcup_{p\in\Lambda_j\setminus\Lambda_{j-1}}C^p\Big)\setminus \mathring K,
\]
and, up to slightly modifying the arcs $C^p$, $p\in\Lambda_j\setminus\Lambda_{j-1}$, assume that $S:=K\cup\Gamma\subset \mathring M_j$ is an admissible subset of $M$ (see Definition \ref{def:admissible}). Notice that $S$ is connected and a strong deformation retract of $M_j$; moreover, as admissible set, $S$ is very simple and the kernel component of $K$ is $M_{j-1}$ (see Definition \ref{def:simple}). Thus, Lemma \ref{lem:path-periods} furnishes a map $\varphi\colon S\to\Sgot_*$ of class $\Ascr(S)$ such that:
\begin{enumerate}[\rm (I)]
\item $\varphi=f_{j-1}$ on $M_{j-1}$.
\vspace*{1mm}
\item $\varphi=f$ on $\bigcup_{p\in\Lambda_j\setminus\Lambda_{j-1}}\Omega_p$.
\vspace*{1mm}
\item $\int_{C^p} \varphi\theta=\qgot(C_p*C^p)-\Zgot(p)$ for all $p\in\Lambda_j\setminus\Lambda_{j-1}$.
\end{enumerate}
%Combining {\rm (iii$_{j-1}$)}, {\rm (I)}, and {\rm (III)} we obtain that
%\begin{equation}\label{eq:(III)}
%     \int_{C^p} \varphi \theta=\qgot(C_p*C^p)-\Zgot(p)\quad \text{for all $p\in\Lambda_j$.}
%\end{equation}

Now, given $\epsilon_j>0$, Lemma \ref{lem:main-noncritical}-{\it ii)} applied to $S$, $M_j$, the arcs $C^p$, $p\in\Lambda_j$, the integer $k\in\n$, and the map $\varphi$, provides a map $f_j\colon M_j\to\Sgot_*$ of class $\Ascr(M_j)$ satisfying the following conditions:
\begin{enumerate}[\rm (I)]
\item[\rm (IV)] $\|f_j-\varphi\|_{0,S}<\epsilon_j$.
\vspace*{1mm}
\item[\rm (V)] $(f_j-\varphi)\theta$ is exact on $S$.
\vspace*{1mm}
\item[\rm (VI)] $\int_{C^p}(f_j-\varphi)\theta=0$ for all $p\in\Lambda_j$.
\vspace*{1mm}
\item[\rm (VII)] $f_j-\varphi$ has a zero of multiplicity $k$ at $p$ for all $p\in\Lambda_j$.
\vspace*{1mm}
\item[\rm (VIII)] No component function of $f_j$ vanishes everywhere on $M_j$.
\end{enumerate}
We claim the map $f_j$ meets properties {\rm (i$_j$)}--{\rm (iv$_j$)}; recall that {\rm (v$_j$)} is already guaranteed. Indeed, {\rm (i$_j$)} follows from {\rm (I)} and {\rm (IV)}; {\rm (ii$_j$)} from {\rm (ii$_{j-1}$)}, {\rm (I)}, {\rm (V)}, and the fact that $M_{j-1}$ is a strong deformation retract of $M_j$; {\rm (iii$_j$)} from {\rm (iii$_{j-1}$)}, {\rm (I)}, {\rm (III)}, and {\rm (VI)}; {\rm (iv$_j$)} from {\rm (iv$_{j-1}$)}, {\rm (I)}, {\rm (II)}, and {\rm (VII)}; and {\rm (vi$_j$)}$=${\rm (VIII)}.

% % % %
% % % % Critical case
% % % % 

\smallskip

\noindent {\em Case 2: The critical case. Assume that $\chi(M_{j}\setminus\mathring  M_{j-1})=-1$.} Now, the change of topology is described by attaching to $M_{j-1}$ a smooth arc $\alpha$ in $\mathring M_j\setminus \mathring M_{j-1}$ meeting $M_{j-1}$ only at their endpoints. Thus, $M_{j-1}\cup \alpha$ is a strong deformation retract of $M_j$. Further, we may choose $\alpha$ such that $\alpha\cap\Lambda=\emptyset$ and $S:=M_{j-1}\cup\alpha$ is an admissible subset of $M$, which is very simple (see Definition \ref{def:simple}). Since both endpoints of $\alpha$ lie in $bM_{j-1}$ there is a closed curve $\beta\subset S$ which contains $\alpha$ as a subarc and is not in the homology of $M_{j-1}$. Now, Lemma \ref{lem:path-periods} furnishes a map $\varphi\colon S\to\Sgot_*$ of class $\Ascr(S)$ such that $\varphi=f_{j-1}$ on $M_{j-1}$ and
\[
    \int_{\beta}\varphi\theta = \qgot(\beta).
\]
Choose a smoothly bounded compact domain $L\subset \mathring M_j$ such that $S\subset \mathring L$, $S$ is a strong deformation retract of $L$, and $L\cap (\Lambda_j\setminus\Lambda_{j-1})=\emptyset$. 
Given $\epsilon_j>0$, Lemma \ref{lem:main-noncritical}-{\it ii)} applied to $S$, $L$, the arcs $C^p$, $p\in\Lambda_{j-1}$, the integer $k\in\n$, and the map $\varphi$, provides a map $\wh f\colon L\to\Sgot_*$ of class $\Ascr(L)$ satisfying the following conditions:
\begin{enumerate}[\rm (i)]
\item $\|\wh f-\varphi\|_{0,S}<\epsilon_j/2$.
\vspace*{1mm}
\item $(\wh f-\varphi)\theta$ is exact on $S$.
\vspace*{1mm}
\item $\int_{C^p}(\wh f-\varphi)\theta=0$ for all $p\in\Lambda_{j-1}$.
\vspace*{1mm}
\item $\wh f-\varphi$ has a zero of multiplicity $k$ at $p$ for all $p\in\Lambda_{j-1}$.
\end{enumerate}
Since the Euler characteristic $\chi(M_j\setminus \mathring L)=0$, this reduces the construction to the noncritical case. This finishes the inductive process and concludes the proof of the theorem.
\end{proof}

To finish this section we prove a Runge-Mergelyan type theorem with jet-interpolation for holomorphic maps into Oka subvarieties of $\c^n$ in which a component function is preserved provided that it holomorphically extends to the whole source Riemann surface. This will be an important tool to ensure conditions {\rm (III)} and {\rm (IV)} in Theorem \ref{th:main-intro2} and {\rm (I)} and {\rm (II)} in Theorem \ref{th:main-intro3}.

%
% Coordenada fija
%

\begin{lemma}\label{lem:MtFix}
Let $n\ge 3$ be an integer and $\Sgot$ be an irreducible closed conical complex subvariety of $\c^n$ which is not contained in any hyperplane. Assume that $\Sgot_*=\Sgot\setminus\{0\}$ is smooth and an Oka manifold, and that $\Sgot\cap\{z_1=1\}$ is also an Oka manifold and the coordinate projection $\pi_1\colon \Sgot\to\c$ onto the $z_1$-axis admits a local holomorphic section $h$ near $z_1=0$ with $h(0)\neq 0$.
Let $M$ be an open Riemann surface of finite topology, $\theta$ be a holomorphic $1$-form vanishing nowhere on $M$, $S=K\cup\Gamma\subset M$ be a connected very simple admissible Runge subset (see Definition \ref{def:simple}) which is a strong deformation retract of $M$. Let $\Lambda\subset \mathring K_0$ be a finite subset where $K_0$ is the kernel component of $S$. Choose $p_0\in \mathring K_0\setminus \Lambda$ and, for each $p\in\Lambda$, let $C_p\subset \mathring K_0$ be an oriented Jordan arc with initial point $p_0$ and final point $p$ such that $C_p\cap C_q=\{p_0\}$ for all $q\neq p\in\Lambda$.

Let $f=(f_1,\ldots,f_n)\colon S\to\Sgot_*$ be a continuous map, holomorphic on $K$, such that $f_1$ extends to a holomorphic map $M\to\c$ which does not vanish everywhere on $M$. Assume also that $f|_K\colon K\to\Sgot_*$ is nonflat.
Then, for any integer $k\in\z_+$, $f$ may be approximated in the $\Cscr^0(S)$-topology by holomorphic maps $\wt f=(\wt f_1,\wt f_2,\ldots,\wt f_n)\colon M\to\Sgot_*$ such that
\begin{enumerate}[\rm (i)]
\item $\wt f_1=f_1$ everywhere on $M$.
\item $\wt f-f$ has a zero of multiplicity $k$ at $p$ for all $p\in\Lambda$.
\item $\int_{C_p}(\wt f -f)\theta=0$ for all $p\in\Lambda$.
\item $\int_{\gamma}(\wt f -f)\theta=0$ for all closed curve $\gamma\subset S$.
\end{enumerate}
\end{lemma}
\begin{proof}
We adapt the ideas in \cite[Proof of Theorem 7.7]{AlarconForstneric2014IM}.
Set $\Sgot':=\Sgot\cap\{z_1=1\}$. By dilations we see that $\Sgot\setminus\{z_1=0\}$ is biholomorphic to $\Sgot'\times\c_*$ (and hence is Oka), and  the projection $\pi_1\colon \Sgot'\to\c$ is a trivial fiber bundle with Oka fiber $\Sgot'$ except over $0\in\c$. Write $(f_1,\wh f)=(f_1,f_2,\ldots,f_n)$, that is, $\wh f:=(f_2,\ldots,f_n)\colon S\to \c^{n-1}$. Since $f_1$ is holomorphic and nonconstant on $M$, its zero set $f_1^{-1}(0)=\{a_1,a_2,\ldots\}$ is a closed discrete subset of $M$. The pullback $f_1^*\pi_1\colon E=f^*\Sgot\to M$ of the projection $\pi_1\colon \Sgot\to\c$ is a trivial holomorphic fiber bundle with fiber $\Sgot'$ over $M\setminus f_1^{-1}(0)$, but it may be singular over the points $a_j\in f_1^{-1}(0)$. The map $\wh f\colon S\to\c^{n-1}$ satisfies $\wh f(x)\in \pi_1^{-1}(f_1(x))$ for all $x\in S$, so $\wh f$ corresponds to a section of $E\to M$ over the set $S$.

Now we need to approximate $\wh f$ uniformly by a section $E\to M$ solving the problem of periods and interpolation. (Except for the period and interpolation conditions, a solution is provided by the Oka principle for sections of ramified holomorphic maps with Oka fibers; see \cite{Forstneric2003FM} or \cite[\textsection 6.14]{Forstneric2017}.) We begin by choosing a local holomorphic solution on a small neighborhood of any point $a_j\in M\setminus S$ so that $\wh f(a_j)\neq 0$, and we add these neighborhoods to the domain of holomorphicity of $\wh f$. Then we need to approximate a holomorphic solution $\wh f$ on a smoothly bounded compact set $K\subset M$ by  a holomorphic one on a larger domain $L\subset M$ assuming that $K$ is a strong deformation retract of $L$ and $L\setminus K$ does not contain any point $a_j$. This can be done by applying the Oka principle for maps to the Oka fiber $G'$ of $\pi\colon G\to\c$ over $\c_*$. 
In the critical case we add a smooth Jordan arc $\alpha$ to the domain $K\subset M$ disjoint from the points $a_j$ and such that $K\cup\alpha$ is a strong deformation retract of the next domain. Next, we extend $\wh f$ smoothly over $\alpha$ so that the integral $\int_{\alpha} \wh f\theta$ takes the correct value by applying an analogous result of Lemma \ref{lem:path-periods} but keeping the first coordinate fixed; this reduces the proof to the noncritical case and concludes the proof of the lemma.
\end{proof}

%%%%%%%%%%
%%%%%%%%%%
%%%%%%%%%%
%%%%%%%%%%   Section: General position + completeness + properness
%%%%%%%%%%
%%%%%%%%%%

\section{General position, completeness, and properness results}\label{sec:plus}

In this section we prove several results that flatten the way to the proof of Theorem \ref{th:main-intro3} in Section \ref{sec:maintheorem}. Thus, all the results in this section concern directed holomorphic immersions of open Riemann surfaces into $\c^n$; we point out that the methods of proof easily adapt to give analogous results for conformal minimal immersions into $\r^n$ (see Section \ref{sec:maintheoremMS}).

We begin with the following

\begin{definition}\label{def:GSI}
Let $\Sgot$ be a closed conical complex subvariety of $\c^n$  $(n\ge 3)$, $M$ be an open Riemann surface, and $S=K\cup\Gamma\subset M$ be an admissible subset (see Definition \ref{def:admissible}). By a {\em generalized $\Sgot$-immersion $S\to\c^n$} we mean a map $F\colon S\to\c^n$ of class $\Cscr^1(S)$ whose restriction to $K$ is an $\Sgot$-immersion of class $\Ascr^1(S)$ and the derivative $F'(t)$ with respect to any local real parameter $t$ on $\Gamma$ belongs to $\Sgot_*$.
\end{definition}

We now prove a Mergelyan type theorem for generalized $\Sgot$-immersions which follows from Lemmas \ref{lem:main-noncritical} and \ref{lem:MtFix}; it will be very useful in the subsequent results. 
\begin{proposition}\label{pro:Mergelyan-S}
Let $\Sgot\subset\c^n$ be as in Theorem \ref{th:MtT}. Let $M$ be a compact bordered Riemann surface and let $S=K\cup\Gamma\subset \mathring M$ be a very simple admissible Runge compact subset such that the kernel component $S_0$ of $S$ (see Definition \ref{def:simple}) is a strong deformation retract of $M$. Let $\Lambda\subset\mathring K$ be a finite subset and assume that $\Lambda\cap K'$ consists of at most a single point for each component $K'$ of $K$, $K'\neq K_0$, where $K_0$ is the kernel component of $K$. 
Given an integer $k\in\n$, every generalized $\Sgot$-immersion $F=(F_1,\ldots,F_n)\colon S\to\c^n$ which is nonflat on $\mathring K_0$ may be approximated in the $\Cscr^1(S)$-topology by $\Sgot$-immersions $\wt F=(\wt F_1,\ldots,\wt F_n)\colon M\to\c^n$ such that $\wt F-F$ has a zero of multiplicity $k\in\n$ at all points $p\in\Lambda$ and that $\wt F$ has no constant component function.

Furthermore, if $\Sgot\cap\{z_1=1\}$ is an Oka manifold the coordinate projection $\pi_1\colon \Sgot\to\c$ onto the $z_1$-axis admits a local holomorphic section $h$ near $z_1=0$ with $h(0)\neq 0$, $\Lambda\subset \mathring K_0$, and $F_1$ extends to a nonconstant holomorphic function $M\to\c$, then $\wt F$ may be chosen with $\wt F_1=F_1$.
\end{proposition}
We point out that an analogous result of the above proposition remains true for arbitrary admissible subsets; we shall not prove the most general statement for simplicity of exposition. Anyway, Proposition \ref{pro:Mergelyan-S} will suffice for the aim of this paper. 
\begin{proof}
Let $\theta$ be a holomorphic $1$-form vanishing nowhere on $M$. Set $f=dF/\theta\colon S\to\Sgot_*$ and observe that $f$ is nonflat on $\mathring K_0$ and of class $\Ascr(S)$, and that $f\theta$ is exact on $S$. Fix a point $p_0\in\mathring K_0\setminus\Lambda$. If $S$ is not connected then $S\setminus S_0$ consists of finitely many pairwise disjoint, smoothly bounded compact disks $K_1,\ldots,K_m$. For each $i\in\{1,\ldots,m\}$ choose a smooth Jordan arc $\gamma_i\subset\mathring M$ with an endpoint in $(bK_0)\setminus\Gamma$, the other endpoint in $bK_i$, and otherwise disjoint from $S$. Choose these arcs so that $S':=S\cup(\bigcup_{i=1}^m\gamma_i)$ is an admissible subset of $M$. It follows that $S'$ is connected, very simple, and a strong deformation retract of $M$. By Lemma \ref{lem:path-periods} we may extend $f$ to a map $f'\colon S'\to\Sgot_*$ of class $\Ascr(S')$ such that $F(p_0)+\int_{p_0}^p f\theta =F(p)$ for all $p\in S'$. From now on we remove the primes and assume without loss of generality that $S$ is connected. 

For each $p\in\Lambda$ choose a smooth Jordan arc $C_p\subset S$ joining $p_0$ with $p$ such that $C_p\cap C_q=\{p_0\}$ for all $p\neq q\in\Lambda$. By Lemma \ref{lem:main-noncritical}-{\it ii)} applied to the set $S\subset M$, the map $f$, the integer $k-1\geq 0$, and the arcs $C_p$, $p\in\Lambda$, we may approximate $f$ uniformly on $S$ by a holomorphic map $\wt f\colon M\to\Sgot_*$ such that:
\begin{enumerate}[\rm (a)]
\item $\wt f\theta$ is exact; recall that $f\theta$ is exact on ($S$ and hence on) $S_0$ and that $S_0$ is a strong deformation retract of $M$.
\item $F(p_0)+\int_{C_p}\wt f\theta = F(p_0)+\int_{C_p}f\theta = F(p)$ for all $p\in\Lambda$.
\item $\wt f-f$ has a zero of multiplicity $k-1$ at all points $p\in\Lambda$.
\item No component function of $\wt f$ vanishes everywhere on $M$.
\end{enumerate}
Then, property {\rm (a)} ensures that the map $\wt F\colon M\to\c^n$ defined by
\[
	\wt F(p):= F(p_0) +\int_{p_0}^p\wt f\theta,\quad p\in M,
\]
is a well-defined $\Sgot$-immersion and is as close as desired to $F$ in the $\Cscr^1(S)$-topology. Moreover, properties {\rm (b)} and {\rm (c)} guarantee that $\wt F-F$ has a zero of multiplicity $k$ at all points of $\Lambda$, whereas {\rm (d)} ensures that $\wt F$ has no constant component function. This concludes the first part of the proof.

The second part of the lemma is proved in an analogous way but using Lemma \ref{lem:MtFix} instead of Lemma \ref{lem:main-noncritical}-{\it ii)}. Moreover, in order to reduce the proof to the case when $S$ is connected, we need to extend $f$ to a map $f'$ on $S'$ as above such that the first component of $f'$ equals $dF_1/\theta$; this is accomplished by a suitable analogous of Lemma \ref{lem:path-periods}, we leave the obvious details to the interested reader. This concludes the proof.
\end{proof}

%%%%%%%%%%
%%%%%%%%%%
%%%%%%%%%%
%%%%%%%%%%   Subsection: General position
%%%%%%%%%%
%%%%%%%%%%

\subsection{A general position theorem}\label{ss:gp}

In this subsection we prove a desingularization result with jet-interpolation for directed immersions of class $\Ascr^1$ on a compact bordered Riemann surface. We use Notation \ref{not:}. 
%
% General position for S-immersions
%

\begin{theorem}\label{th:gp-S}
Let $M$ be a compact bordered Riemann surface and $\Lambda\subset \mathring M$ be a finite set. Let $F\colon M\to\c^n$ $(n\ge 3)$ be an $\Sgot$-immersion of class $\Ascr^1(M)$ such that $F|_{\Lambda}$ is injective. Then, given $k\in\n$, $F$ may be approximated uniformly on $M$ by $\Sgot$-embeddings $\wt F\colon M\to\c^n$ of class $\Ascr^1(M)$ such that $\wt F-F$ has a zero of multiplicity $k$ at $p$ for all $p\in\Lambda$.
\end{theorem}
\begin{proof}
Proposition \ref{pro:nonflat} allows us to assume without loss of generality that $F\colon  M\to\c^n$ is non-flat.
We assume that $M$ is a smoothly bounded compact domain in an open Riemann surface $R$. We associate to $F$ the difference map
\[
\delta F\colon M\times M\to\c^n,\quad \delta F(x,y)=F(y)-F(x).
\]
Obviously, $F$ is injective if and only if $(\delta F)^{-1}(0)=D_M=\{(x,x) : x\in M\}$.

Since $F$ is an immersion and $ F|_{\Lambda}\colon\Lambda\to\c^n$ is injective, there is an open neighborhood $U\subset M\times M$ of $D_M\cup (\Lambda\times\Lambda)$ such that $\delta F\neq 0$ everywhere on $\overline{U}\setminus D_M$. To prove the theorem is suffices to find arbitrarily close to $F$ another $\Sgot$-immersion $\wt F\colon M\to\c^n$ of class $\Ascr^1(M)$ such that $\wt F-F$ has a zero of multiplicity $k$ at all points of $\Lambda$ whose difference map $\delta \wt F|_{M\times M\setminus U}$ is transverse to the origin. Indeed, since $\dim_{\c}M\times M = 2 <n$, this will imply that $\delta \wt F$ does not assume the value zero on $M\times M\setminus U$, so $\wt F(y)\neq \wt F(x)$ when $(x,y)\in M\times M\setminus U$. On the other hand, if $(x,y)\in \overline U\setminus D_M$ then $\wt F(y)\neq \wt F(x)$ provided that $\wt F$ is sufficiently close to $F$.

To construct such an $\Sgot$-immersion we will use the standard transversality argument by Abraham \cite{Abraham1963BAMS}. We need to find a neighborhood $\Vcal\subset \c^N$  of the origin in a complex Euclidean space and a map $H\colon\Vcal\times M\to\c ^n$ of class $\Ascr^1(\Vcal\times M)$ such that 
\begin{enumerate}[\rm (a)]
\item $H(0,\cdot)=F$,
\item $H-F$ has a zero of multiplicity $k$ at $p$ for all $p\in\Lambda$, and 
\item the difference map $\delta H\colon \Vcal\times M \times M\to \c ^n$, defined by
\[
     \delta H(\zeta,x,y)= H(\zeta,y)-H(\zeta,x),\quad \zeta\in\Vcal,\ x,y\in M,
\]
is a submersive family of maps in the sense that the partial differential
\[
      d_\zeta\delta H(\zeta,x,y)|_{\zeta=0}\colon T_0\c^N\cong\c^N\to\c^n
\]
is surjective for any $(x,y)\in M\times M\setminus U$.
\end{enumerate}
By openness of the latter condition and compactness of $M\times M\setminus U$ it follows that the partial differential $d_\zeta \delta H$ is surjective for all $\zeta$ in a neighborhood $\Vcal'\subset \Vcal$ of the origin in $\c^N$. Hence, the map $\delta H\colon M \times M\setminus U\to \c^n$ is transverse to any submanifold of $\c^n$, in particular, to the origin $\{0\}\subset\c^n$. The standard argument then shows that for a generic member $H(\zeta,\cdot)\colon M\to\c^n$ of this family, the difference map $\delta H(\zeta,\cdot)$ is also transverse to $0\in\c^n$ on $M\times M\setminus U$. Choosing such a $\zeta$ sufficiently close to $0$ we then obtain the desired $\Sgot$-embedding $\wt F:=H(\zeta,\cdot)$.

To construct a map $H$ as above we fix a nowhere vanishing holomorphic $1$-form $\theta$ on $R$ and write $dF=f\theta$, where $f\colon M\to \Sgot_*$ is a map of class $\Ascr^1(M)$. 
%Pick a neighborhood $U\subset M\times M$ of $D_M\cup(\Lambda\times\Lambda)$ such that $\overline U \cap(\delta F)^{-1}(0)=D_M$. 
We begin with the following.

\begin{lemma}\label{lem:Hpq}
For any point $(p,q)\in M\times M\setminus(D_M\cup(\Lambda\times\Lambda))$ there is a deformation family $H=H^{(p,q)}(\zeta,\cdot)$ satisfying conditions {\rm (a)} and {\rm (b)} above, with $\zeta\in\c^n$, such that the differential $d_\zeta \delta H(\zeta,p,q)|_{\zeta=0}\colon \c^n\to\c^n$ is an isomorphism. 
\end{lemma}

For the proof we adapt the arguments by Alarc\'on and Forstneri\v c in \cite[Lemma 6.1]{AlarconForstneric2014IM} in order to guarantee also the jet-interpolation; i.e. condition {\rm (b)} of the map $H$.

\begin{proof}
Pick $(p,q)\in M\times M\setminus(D_M\cup(\Lambda\times\Lambda))$. We distinguish cases.

\smallskip

\noindent{\em Case 1: Assume that $\{p,q\}\cap\Lambda\neq \emptyset$.} Assume that $p\in\Lambda$ and hence $q\notin\Lambda$; otherwise we reason in a symmetric way. Write $\Lambda=\{p=p_1,\ldots,p_{l'}\}$. Pick a point $p_0\in M\setminus(\Lambda\cup\{q\})$ and choose closed loops $C_j\subset M\setminus\Lambda$, $j=1,\ldots,l''$, forming a basis of $H_1(M,\z)=\z^{l''}$, and
smooth Jordan arcs $C_{l''+j}$ joining $p_0$ with $p_j$, $j=1,\ldots,l'$, such that setting $l:=l'+l''$, we have that $C_i\cap C_j=\{p_0\}$ for any $i,j\in\{1,\ldots,l\}$ and that $C:=\bigcup_{j=1}^l C_j$ is a Runge set in $M$. Also choose another smooth Jordan arc $C_q$ joining $p_0$ with $q$ and verifying $C\cap C_q=\{p_0\}$. Finally let $\gamma_j:[0,1]\to C_j$ ($j=1,\ldots,l$) and $\gamma:[0,1]\to C_q$ be smooth parametrizations of the respective curves verifying $\gamma_j(0)=\gamma_j(1)=p_0$ for $j=1,\ldots,l''$, $\gamma_j(0)=p_0$ and $\gamma_j(1)=p_j$ for $j=l''+1,\ldots,l$, and $\gamma(0)=p_0$ and $\gamma(1)=q$.

Since $F$ is nonflat, there exist tangential fields $V_1,\ldots,V_n$ on $\Sgot$, vanishing at $0$, and points $x_1,\ldots,x_n\in C_q\setminus\{p_0,q\}$ such that, setting $z_i=f(x_i)\in\Sgot_*$, the vectors $V_1(z_1),\ldots,V_n(z_n)$ span $\c^n$. Let $t_i\in (0,1)$ be such that $\gamma(t_i)=x_i$ and $\phi_t^i$ be the flow of the vector field $V_i$ for small values of $t\in\c$ in the sense of Notation \ref{not:}.
Consider for any $i=1,\ldots,n$ a smooth function $h_i\colon C\cup C_q\to \r_+\subset\c$ vanishing on $C\cup\{q\}$; its values on the relative interior of $C_q$ will be specified later. As in the proof of Lemma \ref{lem:main-noncritical}, set $\zeta=(\zeta_1,\ldots,\zeta_n)\in\c^n$ and consider the map
\[
     \psi(\zeta,x)=\phi^1_{\zeta_1h_1(x)}\circ\cdots\circ\phi^n_{\zeta_nh_n(x)}(f(x))\in\Sgot,
     \quad x\in C\cup C_q,
\]
which is holomorphic in $\zeta\in\c^n$. Note that $\psi(0,\cdot)=f\colon M\to\Sgot_*$ (hence $\psi(\zeta,\cdot)$ does not vanish for $\zeta$ in a small neighborhood of the origin) and $\psi(\zeta,x)=f(x)$ for all $x\in C\cup\{q\}$. It follows that 
\[
     \left.\frac{\di \psi(\zeta,x)}{\di \zeta_i}\right|_{\zeta=0}=h_i(x)V_i(f(x)),\quad i=1,\ldots,n.
\]
We choose $h_i$ with support on a small compact neighborhood of $t_i\in(0,1)$ in such a way that
\begin{equation}\label{eq:hivizi}
\int_0^1 h_i(\gamma(t))V_i(f(\gamma(t))) \theta(\gamma(t),\dot{\gamma}(t))\,dt\approx V_i(z_i)\,\theta(\gamma(t_i),\dot{\gamma}(t_i)).
\end{equation}
Assuming that the neighborhoods are sufficiently small then the approximation in \eqref{eq:hivizi} is close enough so that, since the vectors on the right side above form a basis of $\c^n$, the ones in the left side also do.

Fix a number $\epsilon >0$. Theorem \ref{th:MTJI} furnishes holomorphic functions $g_i\colon M\to\c$ such that
\begin{equation}\label{eq:g_ik-1}
     \text{$g_i$ has a zero of multiplicity $k-1$ at all points of $\Lambda$}
\end{equation}
and
\[
     \sup\limits_{C\cup C_q}|g_i-h_i|<\epsilon,\quad i=1,\ldots,n.
\]
Following the arguments in the proof of Lemma \ref{lem:main-noncritical}, we define holomorphic maps
\begin{eqnarray*}
\Psi(\zeta,x,z) & = & \phi^1_{\zeta_1g_1(x)}\circ\cdots\circ\phi^n_{\zeta_ng_n(x)}(z)\in\Sgot,
\\
\Psi_f(\zeta,x) & = & \Psi(\zeta,x,f(x))\in\Sgot_*,
\end{eqnarray*}
where $x\in M$, $z\in\Sgot$, and $\zeta$ belongs to a sufficiently small neighborhood of the origin in $\c^n$. Observe that $\Psi_f(0,\cdot)=f$.
In view of \eqref{eq:hivizi}, if $\epsilon>0$ is small enough then we have that the vectors
\begin{equation}\label{eq:vectors}
   \left.\frac{\di }{\di \zeta_i}\right|_{\zeta=0}\int_0^1 \Psi_f(\zeta,\gamma(t))\theta(\gamma(t),\dot{\gamma}(t))\ dt
   =\int_0^1 g_i(\gamma(t))V_i(f(\gamma(t))) \theta(\gamma(t),\dot{\gamma}(t))dt,
\end{equation}
$i=1,\ldots,n$,
are close enough to $ V_i(z_i)\theta(\gamma(t_i),\dot{\gamma}(t_i))$ so that they also form a basis of $\c^n$.

To finish the proof it remains to perturb $\Psi_f$ in order to solve the period problem and ensure the jet interpolation at the points of $\Lambda$. From the Taylor expansion of the flow of a vector field it follows that
\[
\Psi_f(\zeta,x)=f(x)+\sum\limits_{i=1}^{n}\zeta_ig_i(x)V_i(f(x))+O(|\zeta|^2).
\]
Since $|g_i|<\epsilon$ on $C$ (recall that $h_i=0$ on $C$), the integral of $\Psi_f$ over the curves $C_1,\ldots,C_l$ can be estimated by
\begin{equation}\label{eq:intCj}
\left| \int_{C_j}\big(\Psi_f(\zeta,\cdot)-f\big)\theta\right|  =  \left| \int_{C_j}\Psi_f(\zeta,\cdot)\theta\right| \leq \eta_0\epsilon|\zeta|,\quad j=1,\ldots,l''
\end{equation}
(recall that $\int_{C_j}f\theta=\int_{C_j} dF=0$ for all $j=1,\ldots, l''$, since these curves are closed),
\begin{multline}\label{eq:intCj'}
\left| \int_{C_j}\big(\Psi_f(\zeta,\cdot)-f\big)\theta\right|  = 
\left| F(p_0)-\Big(F(p_j)-\int_{C_j}\Psi_f(\zeta,\cdot)\theta\Big)\right| \leq 
\\
\leq \eta_0\epsilon|\zeta|,\quad j=l''+1,\ldots,l,
\end{multline} 
for some constant $\eta_0>0$ and sufficiently small $\zeta\in \c^n$. Furthermore, \eqref{eq:g_ik-1} guarantees that 
\begin{equation}\label{eq:g_ik-1'}
    \text{$\Psi_f(\zeta,\cdot)-f$ has a zero of multiplicity $k-1$ at all points of $\Lambda$}
\end{equation}
for $\zeta$ in a small neighborhood of the origin (cf.\ Lemma \ref{lem:jet}).

Now, Lemma \ref{lem:main-noncritical}-{\it i)} furnishes holomorphic maps $\Phi(\wt\zeta,x,z)$ and $\Phi_f(\wt\zeta,x)=\Phi(\wt\zeta,x,f(x))$ with the parameter $\wt\zeta$ in a small neighborhood of $0\in\c^{\wt N}$ for some large $\wt N\in\n$ and $x\in M$ such that $\Phi(0,x,z)=z\in\Sgot$ for all $x\in M$ and 
\begin{equation}\label{eq:rara}
     \Phi_f(0,\cdot)=\Phi_{\Psi_f(0,\cdot)}(0,\cdot)=f,
\end{equation} 
and the differential of the associated period map $\wt\zeta\mapsto\Pcal(\Phi_f(\wt \zeta,\cdot))\in\c^{ln}$ (see \eqref{eq:Periods}) at the point $\wt\zeta=0$ has maximal rank equal to $ln$. The same is true if we allow that $f$ vary locally near the given initial map. 
%Moreover, the difference 
%\begin{equation}\label{eq:g_ik-1''}
%     \text{$\Psi_f(\zeta,\cdot)-\Phi_f(\zeta,\cdot)$ has a zero of multiplicity $k-1$ at all points 
%     of $\Lambda$.}
%\end{equation}
Thus, replacing $f$ by $\Psi_f(\zeta,\cdot)$ and considering the map
\[
\c^{\wt N}\times\c^n\times M\ni(\wt\zeta,\zeta,x)\mapsto\Phi(\wt\zeta,x,\Psi_f(\zeta,x))\in\Sgot_*
\]
defined for $x\in M$ and $(\wt\zeta,\zeta)$ in some sufficiently small neighborhood of $0\in\c^{\wt N}\times\c ^n$, the implicit function theorem provides a holomorphic map $\wt\zeta=\rho(\zeta)$ near $\zeta=0\in\c ^n$ with $\rho(0)=0\in\c ^{\wt N}$ such that the map defined by $\Phi(\rho(\zeta),x,\Psi_f(\zeta,x))$ satisfies
\begin{enumerate}[\rm (i)]
\item $\Pcal(\Phi(\rho(\zeta),\cdot,\Psi_f(\zeta,\cdot)))=\Pcal(\Phi(0,\cdot,\Psi_f(0,\cdot)))=\Pcal(\Psi_f(0,\cdot))=\Pcal(f)$, and
\item $\Phi(\rho(\zeta),\cdot,\Psi_f(\zeta,\cdot))-\Psi_f(\zeta,\cdot)$ has a zero of multiplicity $k-1$ at all points of $\Lambda$.
\end{enumerate}
Condition {\rm (ii)} together with \eqref{eq:g_ik-1'} ensure that
\begin{equation}\label{eq:g_ik-1'''}
    \text{$\Phi(\rho(\zeta),\cdot,\Psi_f(\zeta,\cdot))-f$ has a zero of multiplicity $k-1$ at all points of $\Lambda$}
\end{equation}
for all $\zeta$ in a small neighborhood of $0\in\c^n$.
Obviously the map $\rho=(\rho_1,\ldots,\rho_n)$ also depends on $f$. It follows that the integral
\begin{equation}\label{eq:HF}
H_F(\zeta,x)=F(p_0)+\int_{p_0}^x \Phi(\rho(\zeta),\cdot,\Psi_f(\zeta,\cdot))\theta
\end{equation}
is independent of the choice of the arc from $p_0$ to $x\in M$. Moreover,
\begin{equation}\label{eq:HF=F}
      H_F(0,\cdot)=F
\end{equation} 
(see \eqref{eq:rara}) and $H_F(\zeta,\cdot)$ is an $\Sgot$-immersion of class $\Ascr^1(M)$ for every $\zeta\in\c^n$ sufficiently close to zero such that 
\begin{equation}\label{eq:H_F-interpola}
     H_F(\zeta,\cdot)=F\quad  \text{on $\Lambda$}
\end{equation}
(see {\rm (i)}). In addition, from equations \eqref{eq:intCj} and \eqref{eq:intCj'} we have
\[
|\rho(\zeta)| < \eta_1\epsilon|\zeta|
\]
for some $\eta_1>0$. If we call $\wt V_j$ the vector fields and $\wt g_j$ the functions involved in the construction of the map $\Phi$ (see Lemma \ref{lem:main-noncritical}), the above estimate gives
\[
\left| \Phi(\rho(\zeta),x,\Psi_f(\zeta,x))-\Psi_f(\zeta,x)\right|=\left|\sum \rho_j(\zeta)\wt g_j(x)\wt V_j(\Psi_f(\zeta,x))+O(|\zeta|^2) \right| <\eta_2\epsilon|\zeta| 
\]
for some $\eta_2>0$ and all $x\in M$ and all $\zeta$ near the origin in $\c^n$. Clearly, applying this estimate to the arc $C_q$ we have
\[
\left|  \int_0^1\Phi(\rho(\zeta),\gamma(t),\Psi_f(\zeta,\gamma(t)))\theta(\gamma(t),\dot{\gamma}(t))-\int_0^1\Psi_f(\zeta,\gamma(t))\theta(\gamma(t),\dot{\gamma}(t))\right| <\eta_3\epsilon|\zeta|
\] 
for some $\eta_3>0$. Finally, choosing $\epsilon>0$ small enough, the derivatives
\[
\left. \frac{\di}{\di\zeta_i}\right|_{\zeta=0}\int_0^1 \Phi(\rho(\zeta),\gamma(t),\Psi_f(\zeta,\gamma(t)))\,\theta(\gamma(t),\dot{\gamma}(t))\in\c^n,\quad i=1,\ldots,n,
\]
are so closed to the vectors \eqref{eq:vectors} that they also form a basis of $\c^n$. From the definition of $H_F$, \eqref{eq:HF}, and \eqref{eq:H_F-interpola}, we have
\begin{eqnarray*}
\delta H_F(\zeta,p,q) & = & H_F(\zeta,q)-H_F(\zeta,p)
\\
 & = & H_F(\zeta,q)-F(p)
\\
& = & F(p_0)-F(p)+\int_0^1 \Phi(\rho(\zeta),\gamma(t),\Psi_f(\zeta,\gamma(t)))\theta(\gamma(t),\dot{\gamma}(t)),
\end{eqnarray*}
and hence the partial differential
\[
\left. \frac{\di}{\di\zeta}\right|_{\zeta=0}\delta H(\zeta,p,q)\colon \c^n\to \c ^n
\]
is an isomorphism. This, \eqref{eq:HF=F}, {\rm (ii)}, and \eqref{eq:H_F-interpola} show that $H$ satisfies the conclusion of the lemma.

\smallskip

\noindent{\em Case 2: Assume that $\{p,q\}\cap\Lambda= \emptyset$.} In this case setting $\Lambda':=\Lambda\cup\{p\}$ reduces the proof to Case 1. This proves the lemma.
\end{proof}

The family $H_F$ depending on $F$ given in \eqref{eq:HF} is holomorphically dependent also on $F$ on a neighborhood of a given initial $\Sgot$-immersion $F_0$. In particular, if $F(\xi,\cdot)\colon M\to\c^n$ is a family of holomorphic $\Sgot$-immersions depending holomorphically on $\xi\in\c$ such that $F(\xi,\cdot)-F$ has a zero of multiplicity $k$ at all points $p\in\Lambda$ for any $\xi$, then $H_{F(\xi,\cdot)}(\zeta,\cdot)$ depends holomorphically on $(\zeta,\xi)$. This allows us to compose any finite number of such deformation families by an inductive process. For the case of two families suppose that $H=H_F(\zeta,\cdot)$ and $G=G_F(\xi,\cdot)$ are deformation families with $H_F(0,\cdot)=G_F(0,\cdot)=F$ and such that $H_F(\zeta,\cdot)-F$ and $G_F(\xi,\cdot)-F$ have a zero of multiplicity $k\in\n$ at all points of $\Lambda$ for all $\zeta$ and $\xi$ respectively. Then, we define the composed deformation family by
\[
     (H\#G)_F(\zeta,\xi,x):=G_{H_F(\zeta,\cdot)}(\xi,x),\quad x\in M.
\]
Obviously, we have that $(H\#G)_F(0,\xi,\cdot)=G_F(\xi,\cdot)$ and $(H\#G)_F(\zeta,0,\cdot)=H_F(\zeta,\cdot)$, and $H\#G-F$ has a zero of multiplicity $k$ at $p$ for all $p\in\Lambda$. The operation $\#$ is associative but not commutative. 
%(This operation is similar to the composition of sprays that was introduced by Gromov {\color{red} references}).

To finish the proof of Theorem \ref{th:gp-S}, Lemma \ref{lem:Hpq} gives a finite open covering $\{U_i\}_{i=1}^m$ of the compact set $M\times M\setminus U$ and deformation families $H^i=H^i(\zeta^i,\cdot)\colon M\to\c^n$, with $H^i(0,\cdot)=F$, where $\zeta^i=(\zeta^i_1,\ldots,\zeta^i_{\eta_i})\in\Omega_i\subset\c^{\eta_i}$ for positive integers $\eta_i\in\n$ and $i=1,\ldots,m$. It follows that the difference map $\delta H^i(\zeta^i,p,q)$ is submersive at $\zeta^i=0$ for all $p,q\in U_i$. By taking $\zeta=(\zeta^1,\ldots,\zeta^m)\in\c^N$, with $N=\sum_{i=1}^m\eta_i$, and setting
\[
     H(\zeta,x):=(H^1\#H^2\#\cdots\#H^m)(\zeta^1,\ldots,\zeta^m,x)
\]
we obtain a deformation family such that $H(0,\cdot)=F$, $H(q,\cdot)-F$ has a zero of multiplicity $k$ at $p$ for all $p\in\Lambda$, and $\delta H$ is submersive everywhere on $M\times M\setminus U$ for all $\zeta\in\c ^N$ sufficiently close to the origin. This concludes the proof.
\end{proof}

%%
%% General position for conformal minimal immersions
%%
%
%We now state an analogous result to Theorem \ref{th:gp-S} for conformal minimal immersions.
%
%\begin{theorem}\label{th:gp-minimal}
%Let $M$ be a compact bordered Riemann surface and $\Lambda\subset \mathring M$ be a finite set. Let $\Xi\colon \Lambda\to\r^n$ $(n\ge 5)$ be an injective map and $X\colon M\to\r^n$ be a conformal minimal immersion such that $X=\Xi$ on $\Lambda$. Then, given $k\in\n$, $X$ may be approximated uniformly on $M$ by conformal minimal embeddings $\wt X\colon M\to\r^n$ such that $\wt X$ and $X$ have a contact of order $k$ at $p$ for all $p\in\Lambda$.
%\end{theorem}

%%%%%%%%%%
%%%%%%%%%%
%%%%%%%%%%   Subsection: COMPLETENESS
%%%%%%%%%%
%%%%%%%%%%

\subsection{A completeness lemma}\label{ss:complete}

In this subsection we develop an intrinsic-extrinsic version of the arguments by Jorge and Xavier from \cite{JorgeXavier1980AM} in order to prove the following

\begin{lemma}\label{lem:complete-S}
Let $\Sgot\subset\c^n$ $(n\ge 3)$ be as in Lemma \ref{lem:MtFix}. Let $M$ be a compact bordered Riemann surface and $K\subset \mathring M$ be a smoothly bounded compact domain which is Runge and a strong deformation retract of $M$. Also let $\Lambda\subset \mathring K$ be a finite subset and $p_0\in\mathring K\setminus\Lambda$ be a point.
Then, given an integer $k\in\n$ and a positive number $\tau>0$, every $\Sgot$-immersion $F\colon K\to\c^n$ of class $\Ascr^1(K)$ may be approximated in the $\Cscr^1(K)$-topology by $\Sgot$-immersions $\wt F\colon M\to\c^n$ of class $\Ascr(M)$ satisfying the following conditions:
\begin{enumerate}[\rm (I)]
%\item $\wt F$ approximates $F$ in the $\Cscr^0(K)$-topology.
%\item $\wt F(p)=F(p)$ for all $p\in\Lambda$.
\item $\wt F-F$ has a zero of multiplicity $k$ at all points $p\in\Lambda$.
%\item $\dsf(bM,bK)>\tau$.
\item $\dist_{\wt F}(p_0,bM)>\tau$.
\end{enumerate}
\end{lemma}
\begin{proof}
Without loss of generality we assume that $M$ is a smoothly bounded compact domain in an open Riemann surface $\wt M$. By Proposition \ref{pro:Mergelyan-S} we may assume that $F$ is holomorphic on $M$ and that $f_1$ does not vanish everywhere on $M$. Fix a holomorphic $1$-form $\theta$ vanishing nowhere on $\wt M$ and set $dF=f\theta$ where $f=(f_1,\ldots,f_n)\colon M\to\Sgot_*$ is a holomorphic map. 

Since $K$ is a strong deformation retract of $M$ then $\mathring M\setminus K$ consists of a finite family of pairwise disjoint open annuli. Thus, there exists a finite family of pairwise disjoint, smoothly bounded, compact disks $L_1,\ldots, L_m$ in $\mathring M\setminus K$ satisfying the following property:  if $\alpha\subset M\setminus\bigcup_{j=1}^m L_j$ is a arc connecting $p_0$ with $bM$ then
\begin{equation}\label{eq:lengthpath}
     \int_{\alpha}|f_1\theta| > \tau.
\end{equation}
Recall that $f_1\neq 0$. (Such disks can be found as pieces of labyrinths of Jorge-Xavier type (see \cite{JorgeXavier1980AM}) on the annuli forming $\mathring M\setminus K$; see \cite[Proof of Lemma 4.1]{AlarconFernandezLopez2013CVPDE} for a detailed explanation). 
Set $L:=\bigcup_{j=1}^m L_j$.

For each $j=1,\ldots,m$, choose a Jordan arc $\gamma_j\subset\mathring M$ with an endpoint in $K$, the other endpoint in $L_j$, and otherwise disjoint from $K\cup L$, such that $\gamma_i\cap\gamma_j=\emptyset$ for all $i\neq j\in\{1,\ldots,m\}$ and the set
\[
	S:=K\cup L\cup\Gamma,
\]
where $\Gamma:=\bigcup_{j=1}^m\gamma_j$, is an admissible subset of $M$. It follows that $S$ is a connected very simple admissible subset of $M$ with kernel component $K$ (see Definition \ref{def:simple}), and such that $K$ is a deformation retract of $S$ (hence of $M$). 
Take a map $h=(h_1,\ldots,h_n)\colon S\to\Sgot_*$ of class $\Ascr(S)$ satisfying the following conditions:
\begin{enumerate}[\rm (a)]
\item $h_1=f_1|_{S}$.
\item $h|_K=f|_K$.
\item $|\int_\alpha h\theta| >\tau$ for all arcs $\alpha\subset S$ with initial point $p_0$ and final point in $ L$.
\end{enumerate}
Existence of such a map is clear; we may for instance choose $h$ close to $0\in\c^n$ on each component of $ L$ and such that $|\int_{\gamma_j} h\theta|$ is very large for every component $\gamma_j$ of $\Gamma$. Also choose for each $p\in\Lambda$ a smooth Jordan arc $C_p\subset \mathring K$ with initial point $p_0$ and final one $p$, and assume that $C_p\cap C_q=\{p_0\}$ for all $p\neq q\in\Lambda$.  Then, Lemma \ref{lem:MtFix} provides a holomorphic map $\wt f=(\wt f_1,\wt f_2,\ldots,\wt f_n)\colon M\to\Sgot_*$ such that
\begin{enumerate}[\rm i)]
\item $\wt f$ approximates $h$ on $S$,
\item $\wt f_1=f_1$ everywhere on $M$,
\item $\wt f-h$ has a zero of multiplicity $k-1$ at $p$ for all $p\in\Lambda$,
\item $\int_{C_p}(\wt f-h)\theta=0$ for all $p\in\Lambda$, and
\item $\int_\gamma(\wt f-h)\theta =0$ for all closed curves $\gamma\subset S$.
\end{enumerate}

Since $f\theta=dF$ is exact, properties {\rm (b)} and  {\rm v)} and the fact that $K$ is a strong deformation retract of $M$ guarantee that $\wt f\theta$ is exact on $M$ as well.  Therefore, the map $\wt F=(\wt F_1,\wt F_2,\ldots,\wt F_n)\colon M\to\c^n$ defined by
\[
     \wt F(p):=F(p_0)+\int_{p_0}^p\wt f\theta,\quad p\in M,
\]
is well defined and an $\Sgot$-immersion of class $\Ascr^1(M)$. We claim that if the approximation in {\rm i)} is close enough then $\wt F$ satisfies the conclusion of the lemma. Indeed, properties {\rm i)} and {\rm (b)} guarantee that $\wt F$ approximates $F$ as close as desired in the $\Cscr^1(K)$-topology. On the other hand, {\rm iii)}, {\rm iv)}, and {\rm (b)} ensure that $\wt F - F$ has a zero of multiplicity $k$ at all points of $\Lambda$, which proves {\rm (I)}.
Finally, in order to check condition {\rm (II)}, let $\alpha\subset M$ be an arc with initial point $p_0$ and final one in $bM$. Assume first that $\alpha\cap L\neq\emptyset$ and let $\wt\alpha\subset\alpha$ be a subarc with initial point $p_0$ and final point $q$ for some $q\in L$. Then we have
\[
	\length(\wt F(\alpha))  >  \length(\wt F(\wt \alpha)) 
	\ge |\wt F(q)-\wt F(p_0)|
	=  \big| \int_{p_0}^q \wt f \theta \big|
	\stackrel{\text{\rm i)}}{\approx} 
	\big| \int_{p_0}^q h \theta \big| \stackrel{\text{\rm (c)}}{>} \tau.
\]
Assume that, on the contrary, $\alpha\cap L=\emptyset$. In this case,
\[
	\length(\wt F(\alpha))  = \int_\alpha |\wt f \theta| \ge \int_\alpha |\wt f_1\theta|
	\stackrel{\text{\rm ii)}}{=} \int_\alpha |f_1\theta| \stackrel{\eqref{eq:lengthpath}}{>}\tau.
\]
This proves {\rm (II)} and completes the proof.
\end{proof}

%%%%%%%%%%
%%%%%%%%%%
%%%%%%%%%%   Subsection: Properness
%%%%%%%%%%
%%%%%%%%%%

\subsection{A properness lemma}\label{ss:proper}

Recall that given a vector $x=(x_1,\ldots,x_n)$ in $\r^n$ or $\c^n$ we denote $|x|_\infty=\max\{|x_1|,\ldots,|x_n|\}$; see Section \ref{sec:prelim} for notation.

\begin{lemma}\label{lem:proper-S}
Let $n\ge 3$ be an integer and $\Sgot$ be an irreducible closed conical complex subvariety of $\c^n$ which is not contained in any hyperplane. Assume that $\Sgot_*=\Sgot\setminus\{0\}$ is smooth and an Oka manifold, and that $\Sgot\cap\{z_j=1\}$ is an Oka manifold and the coordinate projection $\pi_j\colon \Sgot\to\c$ onto the $z_j$-axis admits a local holomorphic section $h_j$ near $z_j=0$ with $h_j(0)\neq 0$ for all $j=1,\ldots,n$. Let $M$ be a compact bordered Riemann surface and $K\subset \mathring M$ be a smoothly bounded compact domain which is Runge and a strong deformation retract of $M$. Also let $\Lambda\subset \mathring K$ be a finite subset,
 $F\colon K\to\c^n$ be an $\Sgot$-immersion of class $\Ascr^1(K)$, let $\tau>\rho>0$ be numbers, and assume that
\begin{equation}\label{eq:maxFK}
|F(p)|_\infty>\rho\quad \text{for all $p\in bK$}.
\end{equation} 
Then, given an integer $k\in\n$, $F$ may be approximated in the $\Cscr^1(K)$-topology by $\Sgot$-immersions $\wt F\colon M\to\c^n$ of class $\Ascr^1(M)$ satisfying the following conditions:
\begin{enumerate}[\rm (I)]
\item $\wt F-F$ has a zero of multiplicity $k$ at all points $p\in\Lambda$.
\item $|\wt F(p)|_\infty>\rho$ for all $p\in M\setminus\mathring K$.
\item $|\wt F(p)|_\infty>\tau$ for all $p\in bM$.
\end{enumerate}
\end{lemma}
\begin{proof}
Without loss of generality we assume that $M$ is a smoothly bounded compact domain in an open Riemann surface $\wt M$. By Proposition \ref{pro:Mergelyan-S} we may assume that $F=(F_1,\ldots,F_n)$ is holomorphic on $\wt M$. Since $K$ is a strong deformation retract of $M$, we have that $M\setminus \mathring K$ consists of finitely many pairwise disjoint compact annuli. For simplicity of exposition we assume that $\Acal:=M\setminus\mathring K$ is connected (and hence a single annulus); the same proof applies in general by working separately on each connected component of $M\setminus \mathring K$. We denote by $\alpha$ the boundary component of $\Acal$ contained in $bK$ and by $\beta$ the one contained in $bM$; both $\alpha$ and $\beta$ are smooth Jordan curves.

From inequality \eqref{eq:maxFK} there exist an integer $l\ge 3$, subsets $I_1,\ldots,I_n$ of $\z_l$ (where $\z_l=\{0,1,\ldots,l-1\}$ denotes the additive cyclic group of integers modulus $l$), and a family of compact connected subarcs $\{\alpha_j: j\in \z_l \}$ of $bK$, satisfying the following properties:
\begin{enumerate}[\rm ({a}1)]
\item $\bigcup_{j\in\z_l} \alpha_j=\alpha$.
\item $\alpha_j$ and $\alpha_{j+1}$ have a common endpoint $p_j$ and are otherwise disjoint.
\item $\bigcup_{a=1}^n I_a=\z_l$ and $I_a\cap I_b=\emptyset$ for all $a\neq b\in\{1,\ldots,n\}$.
\item If $j\in I_a$ then $|F_a(p)|>\rho$ for all $p\in\alpha_j$, $a=1,\ldots,n$.
\end{enumerate}
(Possibly $I_a=\emptyset$ for some $a\in\{1,\ldots,n\}$.)
%From {\rm (a4)} we infer that
%\begin{enumerate}
%\item[\rm (a5)] if $j\in I_a$ and $j+1\in I_b$, $a\neq b$, then $|F_a(p_j)|>\rho$ and $|F_b(p_j)|>\rho$. 
%\end{enumerate}

Consider for each $j\in \z_l$ a smooth embedded arc $\gamma_j\subset \Acal$ with the following properties:
\begin{itemize}
\item $\gamma_j$ joins $\alpha\subset bK$ with $\beta\subset bM$ and intersects them transversely.
\item $\gamma_j\cap\alpha=\{p_j\}$.
\item $\gamma_j\cap\beta$ consists of a single point, namely, $q_j$.
\item The arcs $\gamma_j$, $j\in\z_l$, are pairwise disjoint.
\end{itemize}
Consider the admissible set 
\[
	S:=K\cup\big(\bigcup_{j\in \z_l}\gamma_j\big)\subset M
\]
and fix a point $ x_0 \in\mathring K\setminus\Lambda$.
Let $z=(z_1,\ldots,z_n)$ be the coordinates on $\c^n$ and recall that $\pi_a\colon \c^n \to\c$ is the $a$-th coordinate projection $\pi_a(z)=z_a$ for all $a=1,\ldots,n$. Let $\theta$ be a holomorphic $1$-form vanishing nowhere on $\wt M$, and let $f\colon S\to\Sgot_*$ be a map of class $\Ascr(S)$ such that $f=dF/\theta$ on $K$ and the map $\wt G\colon S\to\c^n$ given by
\[
	\wt G(p)=F(p_0)+\int_{ x_0 }^p f\theta,
\]
which is well defined since $K$ is a deformation retract of $S$, satisfies the following conditions:
\begin{enumerate}[\rm ({b}1)]
\item $\wt G=F$ on $K$ and on a neighborhood of $p_j$ for all $j\in\z_l$.
\item If $j\in I_a$ then $|\pi_a(\wt G(z))|>\rho$ for all $z\in\gamma_{j-1}\cup\gamma_j$, $a=1,\ldots,n$.
\item If $j\in I_a$ then $|\pi_a(\wt G(q_{j-1}))|>\tau$ and $|\pi_a(\wt G(q_j))|>\tau$, $a=1,\ldots,n$.
\end{enumerate} 

Existence of such an $f$ is guaranteed by {\rm (a4)}.
%It follows that
%\begin{enumerate}
%\item[{\rm (b4)}] if $j\in I_a$ and $j+1\in I_b$, $a\neq b\in\{1,\ldots,n\}$ then:
%\begin{itemize}
%\item $|\pi_a(\wt G(z))|>\rho$ and $|\pi_b (\wt G(z))|>\rho$ for all $z\in\gamma_j$, and
%\item $|\pi_a(\wt G(q_j))|>\rho +1$ and $|\pi_b(\wt G(q_j))|>\rho +1$.
%\end{itemize}
%\end{enumerate}
Theorem \ref{th:MtT} provides a map $g\colon M\to\Sgot_*$ of class $\Ascr(M)$ such that $g\theta$ is exact on $M$, and the $\Sgot$-immersion $G=(G_1,\ldots,G_n)\colon M\to\c^n$ of class $\Ascr^1(M)$ given by $G(p)=F( x_0 )+\int_{ x_0 }^p g\theta$ enjoys the following properties:
\begin{enumerate}[\rm ({c}1)]
\item $G$ approximates $\wt G$ in the $\Cscr^1(K)$-topology.
\item $G-\wt G$ has a zero of multiplicity $k$ at all point of $\Lambda$.
\item If $j\in I_a$ then $|G_a(p)|>\rho$ for all $p\in\gamma_{j-1}\cup\alpha_j\cup\gamma_j$, $a=1,\ldots,n$.
\item If $j\in I_a$ then $|G_a(p)|>\tau$ for $p\in\{q_{j-1},q_j \}$, $a=1,\ldots,n$.
\end{enumerate}
Property {\rm (c3)} follows from {\rm (a4)} and {\rm (b2)} whereas {\rm (c4)} follows from {\rm (b3)}, provided that the approximation of $f$ by $g$ is close enough.
%We infer that
%\begin{enumerate}
%\item[\rm (c5)] if $j\in I_a$ and $j+1\in I_b$, $a\neq b\in\{1,\ldots,n\}$ then
%\begin{itemize}
%\item $|G_a(p)|>\rho$ and $|G_b(p)|>\rho$ for all $p\in\gamma_j$, and
%\item $|G_a(q_j)|>\rho +1$ and $|G_b(q_j)|>\rho +1$.
%\end{itemize}
%\end{enumerate}

For each $j\in \z_l$ let $\beta_j\subset\beta$ denote the subarc of $\beta$ with endpoints $q_{j-1}$ and $q_j$ which does not contain $q_i$ for any $i\in \z_l\setminus\{j-1,j\}$. It is clear that 
\begin{equation}\label{eq:betaj}
\bigcup_{j\in\z_l}\beta_j=\beta.
\end{equation}
%See Figure \ref{Fig:setproper1}.
%\begin{figure}[ht]
%	\includegraphics[width=11cm]{setproper1.eps}
%	\caption{Description for $l=4$.}\label{Fig:setproper1}
%\end{figure}
%
Also denote by $D_j\subset\Acal$ the closed disk bounded by the arcs $\gamma_{j-1}$, $\alpha_j$, $\gamma_j$, and $\beta_j$; see Figure \ref{Fig:setproper}. It follows that 
\begin{equation}\label{eq:Dj}
\Acal=\bigcup_{j\in\z_l} D_j.
\end{equation}

Call $H_0:=G=(H_{0,1},\ldots,H_{0,n})$ and $I_0:=\emptyset$. We shall construct a sequence of $\Sgot$-immersions $H_b=(H_{b,1},\ldots,H_{b,n})\colon M\to\c^n$, $b=1,\ldots,n$, of class $\Ascr^1(M)$ satisfying the following requirements for all $b\in \{1,\ldots,n\}$:
\begin{enumerate}[\rm ({d}1$_b$)]
\item $H_b$ approximates $H_{b-1}$ in the $\Cscr^1$-topology on $M\setminus (\bigcup_{j\in I_b}\mathring D_j)$.
\item $H_b-H_{b-1}$ has a zero of multiplicity $k$ at all points of $\Lambda$.
\item If $j\in \bigcup_{i=1}^b I_i$ then $|H_b(p)|_\infty>\rho$ for all $p\in D_j$.
\item If $j\in \bigcup_{i=1}^b I_i$ then $|H_b(p)|_\infty>\tau$ for all $p\in \beta_j$.
\item If $j\in I_a$ then $|H_{b,a}(p)|>\rho$ for all $p\in\gamma_{j-1}\cup\alpha_j\cup\gamma_j$, $a=1,\ldots,n$.
\item If $j\in I_a$ then $|H_{b,a}(p)|>\tau$ for $p\in\{q_{j-1},q_j \}$, $a=1,\ldots,n$.
\end{enumerate}

We claim that the $\Sgot$-immersion $\wt F:=H_n\colon M\to\c^n$ satisfies the conclusion of the lemma. Indeed, $\wt F$ approximates $F$ in the $\Cscr^1(K)$-topology by properties {\rm (b1)}, {\rm (c1)} and {\rm (d1$_{1}$)}--{\rm (d1$_n$)}; condition {\rm (I)} is guaranteed by {\rm (d2)}, {\rm (c2)}, and {\rm (b1)}; condition {\rm (II)} by {\rm (d3$_n$)}, {\rm (a3)}, and \eqref{eq:Dj}; and condition {\rm (III)} by {\rm (d4$_n$)}, {\rm (a3)}, and \eqref{eq:betaj}. So, to conclude the proof it suffices to construct the sequence $H_1,\ldots, H_n$ satisfying the above properties. We proceed by induction. Assume that we already have $H_0,\ldots,H_{b-1}$ for some $b\in\{1,\ldots,n\}$ with the desired properties and let us construct $H_b$. Notice that  {\rm (d5$_0$)}$=${\rm (c3)} and {\rm (d6$_0$)}$=${\rm (c4)} formally hold.
By continuity of $H_{b-1}$ and conditions {\rm (d5$_{b-1}$)} and {\rm (d6$_{b-1}$)}, for each $j\in I_b$ there exists a closed disk $\Omega_j\subset D_j\setminus (\gamma_{j-1}\cup\alpha_j\cup\gamma_j)$ such that the following hold.
\begin{enumerate}[\rm (i)]
\item $\Omega_j\cap \beta_j$ is a compact connected Jordan arc.
\item $|H_{b-1,b}(p)|>\rho$ for all $p\in\Upsilon_j:=\overline{D_j\setminus \Omega_j}$.
\item $|H_{b-1,b}(p)|>\tau$ for all $p\in \overline{\beta_j\setminus \Omega_j}$.
\end{enumerate}
Next, for each $j\in I_b$ choose a smooth embedded arc $\lambda_j\subset \Upsilon_j\setminus(\gamma_{j-1}\cup\gamma_j)$ with an endpoint in $\alpha_j$ and the other one in $\Omega_j$ and otherwise disjoint from $b\Upsilon_j$ (see Figure \ref{Fig:setproper}). Moreover, choose each $\lambda_j$ so that the set
\[
	S_b:=\big(M\setminus \bigcup_{j\in I_b}\mathring \Upsilon_j\big)\cup \big(\bigcup_{j\in I_b} \lambda_j\big)
\]
is admissible. Notice that $S_b$ is connected and very simple in the sense of Definition \ref{def:simple}.
\begin{figure}[ht]
	\includegraphics[width=11cm]{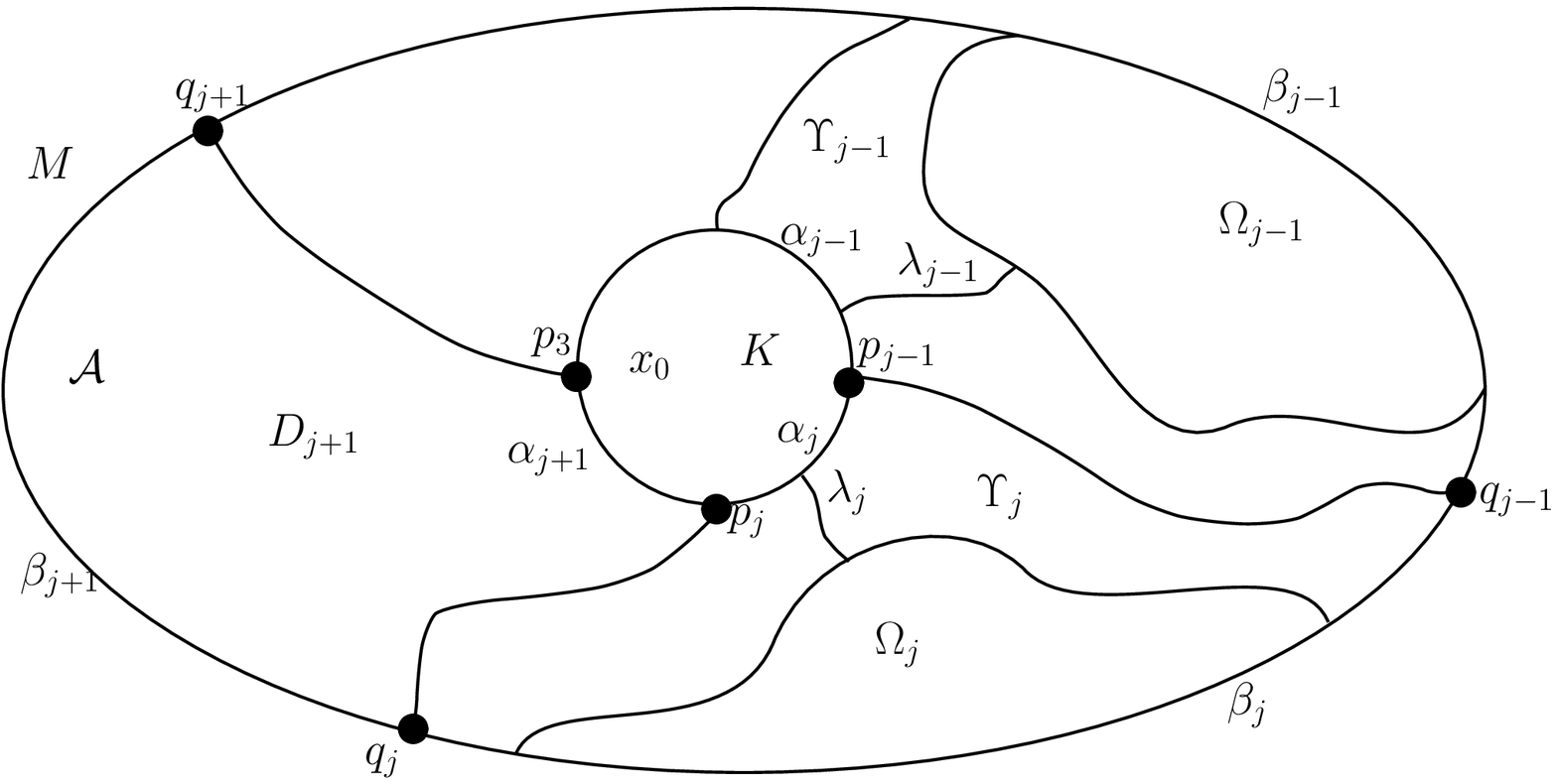}
	\caption{The annulus $\Acal$.}\label{Fig:setproper}
\end{figure}

Set $h=(h_1,\ldots,h_n)=dH_{b-1}/\theta$ and let $\wt h=(\wt h_1,\ldots,\wt h_n)\colon S_b\to\Sgot_*$ be a map of class $\Ascr(S_b)$ such that:
\begin{enumerate}[\rm (i)]
\item[\rm (iv)] $\wt h=h$ on $M\setminus (\bigcup_{j\in I_b}\mathring D_j)$.
\item[\rm (v)] $\wt h_b=h_b$ on $S_b$.
\item[\rm (vi)] The map $\wt H\colon S_b\to\c^n$ given by 
\[
	\wt H(p)=H_{b-1}( x_0 )+\int_{ x_0 }^p \wt h\theta,\quad p\in S_b,
\]
satisfies $|\wt H(p)|_\infty>\tau$ for all $p\in \bigcup_{j\in I_b}\Omega_j$. 
\end{enumerate}
To construct such a map $\wt h$ we may for instance choose $\wt h=h$ on $M\setminus \bigcup_{j\in I_b}\mathring \Upsilon_j$ and suitably define it on $\bigcup_{j\in I_b} \lambda_j$. Now, Lemma \ref{lem:MtFix} furnishes a map $\phi\colon M\to\Sgot_*$ of class $\Ascr(M)$ such $\phi\theta$ is exact on $M$ (take into account that $\wt h\theta=h\theta=dH_{b-1}$ on $K$ and that $K$ is a deformation retract of $M$) and the map $H_b\colon M\to\c^n$ given by
\[
	H_b(p)=H_{b-1}( x_0 )+\int_{ x_0 }^p \phi\theta,\quad p\in M,
\]
is an $\Sgot$-immersion of class $\Ascr^1(M)$ enjoying the following properties:
\begin{enumerate}[\rm (i)]
\item[\rm (vii)] $H_b$ is as close as desired to $\wt H$ in the $\Cscr^1$-topology on $M\setminus (\bigcup_{j\in I_b}\mathring D_j)$.
\item[\rm (viii)] $H_{b,b}=H_{b-1,b}$ (take into account {\rm (v)}).
\item[\rm (ix)] $H_b-\wt H$ has a zero of multiplicity $k$ at all points of $\Lambda$.
\end{enumerate}
Since $\wt H=H_{b-1}$ on $M\setminus (\bigcup_{j\in I_b}\mathring D_j)\supset\Lambda\cup(\bigcup_{j\in\z_l}\gamma_{j-1}\cup\alpha_j\cup\gamma_j)$, then {\rm (d1$_b$)}$=${\rm (vii)}, {\rm (d2$_b$)}$=${\rm (ix)}, and, if the approximation in {\rm (vii)} is close enough, {\rm (d5$_b$)} and {\rm (d6$_b$)} follow from {\rm (d5$_{b-1}$)} and {\rm (d6$_{b-1}$)}, respectively. 

Pick $j\in\bigcup_{i=1}^b I_i$ and $p\in D_j$. If $j\notin I_b$ then {\rm (d3$_{b-1}$)} and {\rm (vii)} ensure that $|H_b(p)|_\infty>\rho$. On the other hand, if $j\in I_b$ then {\rm (ii)} and {\rm (viii)} ensure that $|H_b(p)|\ge |H_{b,b}(p)|>\rho$ provided that $p\in\Upsilon_j$, whereas {\rm (vi)} and {\rm (vii)} guarantee that $ |H_b(p)|>\tau>\rho$ provided that $p\in \Omega_j$. This proves {\rm (d3$_b$)}.

Finally, choose $j\in\bigcup_{i=1}^b I_i$ and $p\in \beta_j$. As above, if $j\notin I_b$ then {\rm (d4$_{b-1}$)} and {\rm (vii)} give that $|H_b(p)|_\infty>\tau$. Likewise, if $j\in I_b$ then {\rm (iii)} and {\rm (viii)} ensure that $|H_b(p)|\ge |H_{b,b}(p)|>\tau$ provided that $p\in\beta_j\setminus\Omega_j$, whereas {\rm (vi)} and {\rm (vii)} imply that $ |H_b(p)|>\tau$ provided that $p\in \beta_j\cap\Omega_j$. This proves {\rm (d4$_b$)} and concludes the proof of the lemma.
\end{proof}

%%%%%%%%%%
%%%%%%%%%%
%%%%%%%%%%
%%%%%%%%%%   MAIN THEOREM
%%%%%%%%%%
%%%%%%%%%%

\section{Proof of Theorem \ref{th:main-intro3}}\label{sec:maintheorem}

As in the proof of Theorem \ref{th:MtT} we can assume that $\Lambda\cap bK=\emptyset$ and also that for each $p\in\Lambda$ we have that either $\Omega_p\subset \mathring K$ or $\Omega_p\cap K=\emptyset$.

Set $M_0:=K$ and let $\{M_j\}_{j\in\n}$ be an exhaustion of $M$ by smoothly bounded Runge compact domains in $M$ such that:
\begin{itemize}
\item $M_0\Subset M_1\Subset\cdots\Subset\bigcup_{j\in\n} M_j=M$.
\item $\chi(M_j\setminus\mathring M_{j-1})\in\{-1,0\}$ for all $j\in\n$.
\item $b M_j\cap \Lambda=\emptyset$ for all $j\in\n$ and so, up to shrinking the sets $\Omega_p$ if necessary, we may assume that $\Omega_p\subset\mathring M_j$ or $\Omega_p\cap M_j=\emptyset$ for all $p\in\Lambda$.
\end{itemize}
The existence of such sequence is guaranteed as in the proof of Theorem \ref{th:MtT}.
The set $\Lambda_j:=\Lambda\cap M_j=\Lambda\cap\mathring M_j$, $j\in\z_+$, is empty or finite; without loss of generality we may assume that $\Lambda_0\neq\emptyset$ and $\Lambda_j\setminus\Lambda_{j-1}\neq \emptyset$ for all $j\in\n$, and hence $\Lambda$ is infinite. Observe that $\Lambda_{j-1}\subsetneq\Lambda_j$ for all $j\in\n$. Fix a sequence $\{\epsilon_j\}_{j\in\n}\searrow 0$ which will be specified later, set 
\[
	F_0:=F|_{M_0}\colon M_0\to\c^n
\]
and, by Proposition \ref{pro:nonflat} and Theorem \ref{th:gp-S}, assume without loss of generality that $F_0$ is nonflat and, if $F|_\Lambda$ is injective, an embedding.

%%%
%%% First part
%%%

\subsection{Proof of the first part of the theorem}\label{ss:1th} 

For the first part of the theorem we shall construct a sequence $\{F_j\}_{j\in\n}$ of nonflat $\Sgot$-immersions $F_j\colon M_j\to\c^n$ of class $\Ascr^1(M_j)$ satisfying:
\begin{enumerate}[\rm (i$_j$)]
\item $\|F_j-F_{j-1}\|_{1,M_{j-1}}<\epsilon_j$.
\item $F_j-F$ has a zero of multiplicity $k\in\n$ at every point $p\in\Lambda_j$.
\item If $F|_{\Lambda}$ is injective then $F_j$ is an embedding.
\end{enumerate}

We proceed by induction. The basis is given by the $\Sgot$-immersion $F_0$, which clearly meets {\rm (ii$_0$)} and {\rm (iii$_0$)}; condition {\rm (i$_0$)} is vacuous. For the inductive step assume that we have an $\Sgot$-immersion $F_{j-1}\colon M_{j-1}\to\c^n$ of class $\Ascr^1(M_{j-1})$ meeting {\rm (i$_{j-1}$)}, {\rm (ii$_{j-1}$)}, and {\rm (iii$_{j-1}$)} for some $j\in\n$, and let us furnish $F_j\colon M_j\to\c^n$ enjoying the corresponding properties. We distinguish two different cases depending on the Euler characteristic of $M_j\setminus \mathring M_{j-1}$.

\smallskip

\noindent {\it Noncritical case: Assume that $\chi(M_j\setminus\mathring  M_{j-1})=0$}. It follows that $M_{j-1}$ is a strong deformation retract of $M_j$, and then Proposition \ref{pro:Mergelyan-S} applied to the data
\[
	S=M_{j-1}\cup \big(\bigcup_{p\in\Lambda_j\setminus\Lambda_{j-1}} \Omega_p\big),
	\quad S_0=M_{j-1},\quad \Lambda=\Lambda_j,\quad k,
\]
and the generalized $\Sgot$-immersion $S\to\c^n$ agreeing with 
$F_{j-1}$ on $M_{j-1}$ and with $F$ on $\bigcup_{p\in\Lambda_j\setminus\Lambda_{j-1}} \Omega_p$,
provides an $\Sgot$-immersion $F_j\colon M_j\to\c^n$ of class $\Ascr^1(M_j)$ that meets {\rm(i$_{j}$)} and {\rm (ii$_{j}$)}. Finally, if $F|_{\Lambda}$ is injective then Theorem \ref{th:gp-S} enables us to choose $F_j$ being an embedding; this ensures {\rm(iii$_{j}$)}.

\smallskip

\noindent {\it Critical case: Assume that $\chi(M_j\setminus\mathring  M_{j-1})=-1$}. We then have that the change of topology is described by attaching to $M_{j-1}$ a smooth arc $\alpha$ in $\mathring M_j\setminus \mathring M_{j-1}$ meeting $M_{j-1}$ only at their endpoints. Thus, $M_{j-1}\cup\alpha$ is a strong deformation retract of $M_j$. Further, we may choose $\alpha$ such that $\alpha\cap\Lambda=\emptyset$; and $S:= M_{j-1}\cup\alpha$ is an admissible subset of $M$, which is clearly very simple (see Definition \ref{def:simple}). We use Lemma \ref{lem:path-periods} to extend $F_{j-1}$ to $S$ as a generalized $\Sgot$-immersion. By Proposition \ref{pro:Mergelyan-S}, we may approximate $F_{j-1}$ in the $\Cscr^1(M_{j-1}\cup\alpha)$-topology by nonflat $\Sgot$-immersions on a small compact tubular neighborhood $M_j'\Subset M_j$ of $M_{j-1}\cup\alpha$ having a contact of order $k$ with $F$ at all points of $\Lambda_j$. Since $\chi(M_j\setminus\mathring M_j')=0$, this reduces the proof to the previous case and hence concludes the recursive construction of the sequence $\{F_j\}_{j\in\n}$.

Finally, if the number $\epsilon_j>0$ is chosen sufficiently small at each step in the recursive construction, properties {\rm(i$_{j}$)}, {\rm(ii$_{j}$)}, and {\rm (iii$_{j}$)} ensure that the sequence $\{F_j\}_{j\in\n}$ converges uniformly on compacta in $M$ to an $\Sgot$-immersion
\[
      \wt F:=\lim_{j\to\infty}F_j\colon M\to\c^n,
\]
which is as close as desired to $F$ uniformly on $K$, is injective if $F|_{\Lambda}$ is injective, and such that $\wt F-F$ has a zero of multiplicity $k$ at all points of $\Lambda$.

%%%
%%% Proof of (I)
%%%

\subsection{Proof of assertion {\rm (I)}}\label{ss:(I)}

Suppose that the assumptions in assertion {\rm (I)} hold. Fix a point $p_0\in \mathring K\setminus \Lambda$. We shall now construct a sequence of $\Sgot$-immersions $F_j\colon M_j\to\c^n$ of class $\Ascr^1(M_j)$, $j\in\n$, satisfying conditions {\rm (i$_j$)}--{\rm (iii$_j$)} above and also
\begin{enumerate}[\rm (i)$_j$]
\item[\rm (iv$_j$)] $\dist_{F_j}(p_0,bM_j)>j$ for all $j\in\n$.
\end{enumerate}

Observe that $F_0=F|_{M_0}$ meets {\rm (iv$_0$)} since it is an immersion and $p_0\in \mathring K$. For the inductive step assume that we already have $F_{j-1}$ satisfying {\rm (i$_j$)}--{\rm (iv$_j$)} for some $j\in\n$ and, reasoning as above, construct an $\Sgot$-immersion $F_j'\colon M_j\to\c^n$ meeting {\rm (i$_j$)}, {\rm (ii$_j$)}, and {\rm (iii$_j$)}. Let $M_j'\subset \mathring M_j$ be a smoothly bounded compact domain which is Runge and a strong deformation retract of $M_j$ and contains $M_{j-1}\cup\Lambda_j$ in its relative interior. Then, Lemma \ref{lem:complete-S} applied to the data
\[
	M=M_j,\quad K=M_j',\quad \Lambda=\Lambda_j,\quad k,\quad \tau=j, 
	\quad\text{and}	\quad F=F_j'|_{M_j'},
\]
gives an $\Sgot$-immersion $F_j\colon M_j\to\c^n$ of class $\Ascr^1(M_j)$ meeting {\rm (ii$_j$)}, {\rm (iv$_j$)}, and also {\rm (i$_j$)} provided that the approximation of $F_j'$ by $F_j$ on $M_j'$ is close enough; Theorem \ref{th:gp-S} enables us to assume that $F_j$ also meets {\rm (iii$_j$)}. This closes the induction and concludes the construction of the sequence $\{F_j\}_{j\in\n}$ with the desired properties.

As above, if the number $\epsilon_j>0$ is chosen sufficiently small at each step in the recursive construction, properties {\rm(i$_{j}$)}-{\rm(iii$_{j}$)} ensure that the sequence $\{F_j\}_{j\in\n}$ converges uniformly on campacta in $M$ to an $\Sgot$-immersion $\wt F:=\lim_{j\to\infty}F_j\colon M\to\c^n$ which is as close as desired to $F$ uniformly on $K$, is injective if $F$ is injective, and such that $\wt F-F$ has a zero of multiplicity $k$ at all points of $\Lambda$. In addition, property {\rm(iv$_{j}$)} ensures that 
\[
	\lim_{j\to\infty}\dist_{\wt F}(p_0,bM_j)=+\infty
\]
whenever the number $\epsilon_j>0$ is chosen small enough at each step in the recursive process.
This implies that $\wt F$ is complete and concludes the proof of assertion {\rm (I)}.

%%%
%%% Proof of (II)
%%%

\subsection{Proof of assertion {\rm (II)}} \label{ss:(II)} 

Suppose that the assumptions in assertion {\rm (II)} hold. Observe that $F|_\Lambda\colon\Lambda\to\c^n$ is a proper map if, and only  if, $(F|_\Lambda)^{-1}(C)$ is finite for any compact set $C\subset\c^n$, or, equivalently, if either the closed discrete set $\Lambda$ is finite or for some (and hence for any) ordering $\Lambda=\{p_1,p_2,p_3,\ldots\}$ of $\Lambda$, the sequence $\{F(p_1),F(p_2),F(p_3),\ldots\}$ is divergent in $\c^n$. Since we are assuming that $\Lambda$ is infinite, there is $j_0\in \n$ such that 
\begin{equation}\label{eq:Gp=no0}
F(p)\neq 0\quad \text{for all $p\in\Lambda\setminus \Lambda_{j_0}$}. 
\end{equation}

In a first step we construct for each $j\in\{0,\ldots, j_0 \}$ an $\Sgot$-immersion $F_j\colon M_j\to\c^n$ of class $\Ascr^1(M_j)$ satisfying conditions {\rm (i$_j$)}--{\rm (iii$_j$)} above; we reason as in Subsec.\ \ref{ss:1th}.  Now, up to a small deformation of $M_{j_0}$ if necessary, we may assume without loss of generality that $F_{j_0}$ does not vanish anywhere on $bM_{j_0}$, and hence there exists $\rho_{j_0}>0$ such that
\begin{equation}\label{eq:Fj0}
|F_{j_0}(p)|_{\infty}> \rho_{j_0}>0\quad \text{for all $p\in bM_{j_0}$}.
\end{equation}
Set
\begin{equation}\label{eq:rhoj}
\rho_j:= \min\{|F(p)|_{\infty}\colon p\in\Lambda_j\setminus\Lambda_{j-1} \} \quad \text{for all $j\geq j_0+1$.}
\end{equation}
Recall that $\Lambda_j\setminus\Lambda_{j-1}\neq\emptyset$ for all $j\in\n$. In view of \eqref{eq:Gp=no0} and \eqref{eq:Fj0} we have that $\rho_j>0$ for all $j\geq j_0$. Moreover, since $F|_\Lambda$ is proper then
\begin{equation}\label{eq:rhotoinfty}
      \lim\limits_{j\to+\infty}\rho_j =+\infty.
\end{equation}
%Thus,
%\begin{equation}\label{eq:varsigma}
%	0<\varsigma:=\inf\{\rho_j\colon j\ge j_0\}=\min\{\rho_j\colon j\ge j_0\}.
%\end{equation}	

In a second step, we shall construct a sequence of $\Sgot$-immersions $F_j\colon M_j\to\c^n$ of class $\Ascr^1(M_j)$, for $j\geq j_0+1$, enjoying conditions {\rm (i$_j$)}--{\rm (iii$_j$)} and also
\begin{enumerate}[\rm ({v$_j$}.1)]
	\item $|F_j(p)|_{\infty}>\frac12\min\{ \rho_{j-1}, \rho_j \}$ for all $p\in M_j\setminus \mathring M_{j-1}$, and
	\item $|F_j(p)|_\infty>\rho_j$ for all $p\in bM_j$.
\end{enumerate}

We proceed in an inductive way. The basis of the induction is accomplished by $F_{j_0}$; recall that it meets {\rm (i$_{j_0}$)}--{\rm (iii$_{j_0}$)} whereas property {\rm (v$_{j_0}$.1)} is vacuous and property {\rm (v$_{j_0}$.2) follows from \eqref{eq:Fj0}. For the inductive step, assume that we already have $F_{j-1}\colon M_{j-1}\to\c^n$ for some $j\geq j_0+1$ satisfying {\rm (i$_{j-1}$)}--{\rm (iii$_{j-1}$)}, {\rm (v$_{j-1}$.1)}, and {\rm (v$_{j-1}$.2) and let us construct an $\Sgot$-immersion $F_j\colon M_j\to\c^n$ of class $\Ascr^1(M_j)$ with the corresponding requirements.
		
By \eqref{eq:rhoj} and up to a shrinking of the set $\Omega_p$ if necessary, we may assume that
\begin{equation}\label{eq:Gq}
	|F(q)|_{\infty}>\frac{\rho_j}2 \quad \text{for all points $q$ in $\Omega^j:=\bigcup_{p\in\Lambda_j\setminus\Lambda_{j-1}}\Omega_p$}\neq\emptyset.
\end{equation}
Next, choose a smooth Jordan arc $C_p$ for each $p\in\Lambda_j\setminus\Lambda_{j-1}$ with the initial point in $bM_{j-1}$, the final point in $b\Omega_p$, and otherwise disjoint from $M_{j-1}\cup\Omega^j$, and such that
\[
    S':=M_{j-1}\cup \Omega^j\cup \big( \bigcup_{p\in \Lambda_j\setminus\Lambda_{j-1}} C_p\big)
\]
is a very simple admissible subset of $M_j$; in particular $C_p\cap C_q=\emptyset$ if $p\neq q$.
If $\chi(M_j\setminus\mathring M_{j-1})=-1$ we then also choose another smooth Jordan arc $\alpha\subset\mathring M_j$ with its two endpoints in $bM_{j-1}$ and otherwise disjoint from $S'$ such that 
$S'\cup\alpha$ 
is admissible and a strong deformation retract of $M_j$. If $\chi(M_j\setminus\mathring M_{j-1})=0$ we set $\alpha:=\emptyset$. In any case, the set
\[
	S:=S'\cup\alpha\subset \mathring M_j
\]
is admissible in $M$ and a strong deformation retract of $M_j$.
Set
\[
	C:=\alpha\cup\Big(\bigcup_{p\in \Lambda_j\setminus\Lambda_{j-1}} C_p\Big),
\]
and observe that $S=(M_{j-1}\cup\Omega^j)\cup C$. 
Consider a generalized $\Sgot$-immersion $\wt F_j\colon S\to\c^n$ of class $\Ascr^1(S)$ 
such that:
\begin{enumerate}[\rm (A.1)]
\item[\rm (A.1)] $\wt F_j|_{M_{j-1}}=F_{j-1}$.
\item[\rm (A.2)] $\wt F_j|_{\Omega^j}=F|_{\Omega^j}$.
\item[\rm (A.3)] $|\wt F_j(q)|_{\infty}>\frac12\min\{\rho_{j-1},\rho_j \}$ for all $q\in C$. 
\end{enumerate}
To ensure {\rm (A.3)} we use Lemma \ref{lem:path-periods}; take into account {\rm (v$_{j-1}$.2)} and \eqref{eq:Gq}. Thus, {\rm (v$_{j-1}$.2)}, \eqref{eq:Gq}, and {\rm (A.3)} guarantee that 
\begin{enumerate}[\rm (A.1)]
\item[ \rm (A.4)] $|\wt F_j(p)|_{\infty}>\frac12\min\{\rho_{j-1},\rho_j \}>0$ for all $p\in S\setminus \mathring M_{j-1}=(bM_{j-1})\cup\Omega^j\cup C$.
\end{enumerate}

Since $S\subset \mathring M_j$ is Runge and a strong deformation retract of $M_j$, Proposition \ref{pro:Mergelyan-S} applied to the data
\[
    M=M_j, \quad S, \quad \Lambda=\Lambda_j,\quad k,\quad \text{and}\quad F=\wt F_j
\]
gives a nonflat $\Sgot$-immersion $\wh F_j\colon M_j\to\c^n$ of class $\Ascr^1(M_j)$ such that
\begin{enumerate}[\rm (B.1)]
	\item[\rm (B.1)] $\wh F_j$ is as close as desired to $\wt F_j$ in the $\Cscr^1(S)$-topology.
	\item[\rm (B.2)] $\wh F_j -\wt F_j$ has a zero of multiplicity $k\in\n$ at every point $p\in\Lambda_j$.
\end{enumerate}

If the approximation in {\rm (B.1)} is close enough then, in view of {\rm (A.4)}, there exists a small compact neighborhood $N$ of $S$ in $\mathring M_j$, being a smoothly bounded compact domain and a strong deformation retract of $M_j$, and such that
\begin{enumerate}[\rm (B.1)]
	\item[\rm (B.3)] $|\wh F_j(p)|_{\infty}>\frac12\min\{\rho_{j-1},\rho_j \}>0 \quad 
      \text{for all $p\in N\setminus \mathring M_{j-1}$}.$
\end{enumerate}
Notice that $\Lambda\cap (M_j\setminus\mathring N)=\emptyset$ and hence we may apply Lemma \ref{lem:proper-S} to the data
\[
      M=M_j,\quad K=N,\quad \Lambda=\Lambda_j,\quad F=\wh F_j,\quad 
	\rho=\frac12\min\{\rho_{j-1},\rho_j \}, \quad \tau =\rho_j,\quad k,
\]
obtaining an $\Sgot$-immersion $F_j\colon M_j\to\c^n$ of class $\Ascr^1(M_j)$ such that:
\begin{enumerate}[\rm (C.1)]
	\item $F_j$ is as close as desired to $\wh F_j$ in the $\Cscr^1(N)$-topology.
	\item[\rm (C.2)]  $F_j-\wh F_j$ has a zero of multiplicity $k$ at every point $p\in\Lambda_j\subset\mathring N$.
	\item[\rm (C.3)] $|F_j(p)|_\infty>\frac12\min\{\rho_{j-1},\rho_j \}$ for all $p\in M_j\setminus \mathring N$.
	\item[\rm (C.4)] $|F_j(p)|_\infty>\rho_j$ for all $p\in bM_j$.
\end{enumerate}

We claim that, if the approximations in {\rm (B.1)} and {\rm (C.1)} are close enough, the $\Sgot$-immersion $F_j\colon M_j\to\c^n$ satisfies properties {\rm (i$_j$)}--{\rm (iii$_j$)}, {\rm (v$_j$.1)}, and {\rm (v$_j$.2)}. Indeed, {\rm (A.1)} ensures {\rm (i$_j$)}; properties {\rm (A.2)}, {\rm (B.2)}, and {\rm (C.2)} guarantee {\rm (ii$_j$)}; condition {\rm (v$_j$.1)} follows from {\rm (A.4)}, {\rm (B.3)}, and {\rm (C.3)}; and condition {\rm (v$_j$.2)} is implied by {\rm (C.4)}.
Finally, if $F|_{\Lambda}$ is injective then, by Theorem \ref{th:gp-S}, we may assume without loss of generality that $F_j$ is an embedding. This closes the inductive step and concludes the recursive construction of the sequence $\{F_j\}_{j\ge j_0+1}$ meeting the desired requirements.

As above, choosing the number $\epsilon_j>0$ $(j\in\n)$ small enough at each step in the construction, properties {\rm(i$_{j}$)}-{\rm(iii$_{j}$)} ensure that the sequence $\{F_j\}_{j\in\n}$ converges uniformly on compact subsets of $M$ to an $\Sgot$-immersion $\wt F:=\lim_{j\to\infty}F_j\colon M\to\c^n$ which is as close as desired to $F$ uniformly on $K$, is injective if $F|_{\Lambda}$ is injective, and such that $\wt F-F$ has a zero of multiplicity $k$ at all points of $\Lambda$. 
Furthermore, properties {\rm (v$_j$.1)} and \eqref{eq:rhotoinfty} imply that $\wt F$ is a proper map. Indeed, take a number $R>0$ and a sequence $\{q_m\}_{m\in\n}$ that diverges on $M$, and let us check that there is $m_0\in\n$ such that $|\wt F(q_m)|_{\infty}>R$ for all $m\ge m_0$. 
Indeed, set
\[%\begin{equation}\label{eq:varepsilon}
	\varepsilon:= \sum_{j\ge 1} \epsilon_j <+\infty
\]%\end{equation}
and observe that, by %\eqref{eq:varepsilon} and 
properties {\rm (i$_j$)},
\begin{equation}\label{eq:wF-Fj}
	\|\wt F-F_j\|_{1,M_j}  <\varepsilon\quad \text{for all $j\in\z_+$}.
\end{equation}
On the other hand, in view of \eqref{eq:rhotoinfty} there is an integer $j_1\ge j_0+1$ such that 
\begin{equation}\label{eq:rhoR}
	\rho_{j-1}> 2(R+\varepsilon) \quad \text{for all $j\ge j_1$}.
\end{equation}
Now, since the sequence $\{p_m\}_{m\in \n}$ diverges on $M$ and $\{M_j\}_{j\in\n}$ is an exhaustion of $M$, there is $m_0\in\n$ such that 
\[
	p_m\in M\setminus M_{j_1}\quad \text{for all $m\ge m_0$}.
\]
Thus, for any $m\ge m_0$ there is an integer $j_m\ge j_1$ such that $q_m\in M_{j_m}\setminus \mathring M_{j_m-1}$, and so
\begin{eqnarray*}
	|\wt F(q_m)|_\infty & \ge & |F_{j_m}(q_m)|_\infty - |F_{j_m}(q_m)-\wt F(q_m)|_\infty 
	\\
	& \stackrel{\text{{\rm (v$_{j_m}$.1)}, \eqref{eq:wF-Fj}} }{>} & 
	\frac12 \min\{\rho_{{j_m}-1},\rho_{j_m}\} - \varepsilon
	\;  \stackrel{\eqref{eq:rhoR}}{>} \; R.
\end{eqnarray*}
This proves that $\wt F\colon M\to\c^n$ is a proper map and concludes the proof of Theorem \ref{th:main-intro3}.

%%%%%%%%%%
%%%%%%%%%%
%%%%%%%%%%
%%%%%%%%%%   MAIN THEOREM MS
%%%%%%%%%%
%%%%%%%%%%

\section{Sketch of the proof of Theorem \ref{th:main-intro2} and proof of Theorem \ref{th:main-intro}}\label{sec:maintheoremMS}

In this section we briefly explain how the arguments in Sections \ref{sec:plus} and \ref{sec:maintheorem} which have enabled us to prove Theorem \ref{th:main-intro3} may be adapted in order to guarantee Theorem \ref{th:main-intro2}; we shall leave the obvious details of the proof to the interested reader. Afterward, we will use Theorem \ref{th:main-intro2} to prove Theorem \ref{th:main-intro}.

First of all recall that, as pointed out in Subsec.\ \ref{ss:Oka}, for any integer $n\ge 3$ the punctured null quadric $\Agot_*\subset\c^n$ (see \eqref{eq:nullquadric} and \eqref{eq:punctured}) directing minimal surfaces in $\r^n$ is an Oka manifold and satisfies the assumptions in assertions {\rm (I)} and {\rm (II)} in Theorem \ref{th:main-intro3}. 
 Thus, Theorem \ref{th:MtT} and Lemma \ref{lem:MtFix} hold for $\Sgot=\Agot$.

The first step in the proof of Theorem \ref{th:main-intro2} consists of providing an analogous of Proposition \ref{pro:Mergelyan-S} for {\em generalized conformal minimal immersions} in the sense of \cite[Definition 5.2]{AlarconForstnericLopez2016MZ}. In particular we need to show that if we are given $M$, $S$, and $\Lambda$ as in Proposition \ref{pro:Mergelyan-S} then, for any integer $k\in\z_+$, every generalized conformal minimal immersion $X\colon S\to\r^n$ $(n\ge 3)$ %which is nonflat on $\mathring K_0$ 
may be approximated in the $\Cscr^1(S)$-topology by conformal minimal immersions $\wt X\colon M\to\r^n$ of class $\Cscr^1(M)$ such that $\wt X$ and $X$ have a contact of order $k$ at every point in $\Lambda$ and the flux map $\Flux_{\wt X}$ equals $\Flux_X$ everywhere in the first homology group $H_1(M;\z)$. To do that we reason as in the proof of Proposition \ref{pro:Mergelyan-S} but working with the map $f:=\di X/\theta\colon S\to\Agot_*$. Since $f\theta$ does not need to be exact (only its real part does) we replace conditions {\rm (a)} and {\rm (b)} in the proof of the proposition by the following ones:
\begin{itemize}
\item $(\wt f-f)\theta$ is exact on $S$.
\item $X(p_0)+2\int_{C_p}\Re(\wt f\theta)=X(p_0)+2\int_{C_p} \Re(f\theta)=X(p)$ for all $p\in\Lambda$.
\end{itemize}
It can then be easily seen that the conformal minimal immersion $\wt X\colon M\to\r^n$ of class $\Cscr^1(M)$ given by
\[
	\wt X(p):=X(p)+2\int_{C_p}\Re(\wt f\theta),\quad p\in M,
\]
is well defined and enjoys the desired properties.

In a second step and following the same spirit, we need to furnish a general position theorem, a completeness lemma, and a properness lemma for conformal minimal immersions of class $\Cscr^1$ on a compact bordered Riemann surface, which are analogues of Theorem \ref{th:gp-S}, Lemma \ref{lem:complete-S}, and Lemma \ref{lem:proper-S}, respectively. In this case the general position of conformal minimal surfaces is embedded in $\r^n$ for all integers $n\ge 5$; to adapt the proof of Theorem \ref{th:gp-S} to the minimal surfaces framework we combine the argument in \cite[Proof of Theorem 1.1]{AlarconForstnericLopez2016MZ} with the new ideas in Subsec.\ \ref{ss:gp} which allow us to ensure the interpolation condition. Likewise, the analogues of Lemmas \ref{lem:complete-S} and \ref{lem:proper-S} for conformal minimal surfaces can be proved by adapting the proofs of the cited lemmas in Subsec.\ \ref{ss:complete} and \ref{ss:proper}, respectively; the required modifications follow the pattern described in the previous paragraph: at each step in the proofs we ensure that the real part of the $1$-forms is exact and that the periods of the imaginary part agree with the flux map of the initial conformal minimal immersion. Furthermore, obviously, we are allowed to use only the real part in order to ensure the increasing of the intrinsic diameter of the surface to achieve completeness (cf.\ Lemma \ref{lem:complete-S} {\rm (II)}) and the increasing of the $|\cdot|_\infty$-norm near the boundary to guarantee properness (cf.\ Lemma \ref{lem:proper-S} {\rm (II)} and {\rm (III)}). For the former we just replace condition {\rm (c)} in the proof of Lemma \ref{lem:complete-S} (which determines an extrinsic bound) by the following one:
\begin{itemize}
\item $|2\int_\alpha \Re(h\theta)| >\tau$ for all arcs $\alpha\subset S$ with initial point $p_0$ and final point in $ L$.
\end{itemize}
For the latter the adaptation is done straightforwardly since all the bounds are of the same nature, namely, extrinsic.

Finally, granted the analogues for conformal minimal surfaces in $\r^n$ of Proposition \ref{pro:Mergelyan-S}, Theorem \ref{th:gp-S}, Lemma \ref{lem:complete-S}, and Lemma \ref{lem:proper-S}, the proof of Theorem \ref{th:main-intro2} follows word by word, up to trivial modifications similar to the ones discussed in the previous paragraphs, the one of Theorem \ref{th:main-intro3} in Section \ref{sec:maintheorem}. It is perhaps worth to point out that in the noncritical case in the recursive construction (see Subsec.\ \ref{ss:1th}) we now have to extend a conformal minimal immersion of class $\Cscr^1(M_{j-1})$ to a generalized conformal minimal immersion on the admissible set $S=M_{j-1}\cup\alpha\subset \mathring M_j$ whose flux map equals $\pgot$ for every closed curve in $S$ (here $\pgot\colon H_1(M;\z)\to\r^n$ denotes the group homomorphism given in the statement of Theorem \ref{th:main-intro2} whereas $M_{j-1}$, $\alpha$, and $M_j$ are as in Subsec.\ \ref{ss:1th}); this can be easily done as in \cite[Proof of Theorem 1.2]{AlarconForstnericLopez2016MZ}. This concludes the sketch of the proof of Theorem \ref{th:main-intro2}; as we announced at the very beginning of this section, we leave the details to the interested reader.

%
% Proof of Theorem 1.1
%

To finish the paper we show how Theorem \ref{th:main-intro2} can be used in order to prove the following extension to Theorem \ref{th:main-intro} in the introduction.

\begin{corollary}\label{co:}
Let $M$ be an open Riemann surface and $\Lambda\subset M$ be a closed discrete subset. Consider also an integer $n\ge 3$ and maps $\Zgot\colon \Lambda\to\r^n$ and $G\colon \Lambda\to Q_{n-2}=\{[z_1\colon\cdots\colon z_n]\in\cp^{n-1}\colon z_1^2+\cdots+z_n^2=0\}\subset\cp^{n-1}$. Then there is a conformal minimal immersion $\wt X\colon M\to\r^n$ satisfying $\wt X|_\Lambda = \Zgot$ and whose generalized Gauss map $G_{\wt X}\colon M\to \cp^{n-1}$ equals $G$ on $\Lambda$. 
\end{corollary}
\begin{proof}
For each $p\in\Lambda$ let $\Omega_p$ be a smoothly bounded simply-connected compact neighborhood of $p$ in $M$, and assume that $\Omega_p\cap\Omega_q=\emptyset$ whenever that $p\neq q\in\Lambda$. Set $\Omega:=\bigcup_{p\in\Lambda}\Omega_p$ and let $X\colon \Omega\to\r^n$ be any conformal minimal immersion  of class $\Cscr^1(\Omega)$ such that $X|_\Lambda=\Zgot$ and the generalized Gauss map $G_X|_\Lambda =G$. (Such always exists; we may for instance choose $X|_{\Omega_p}$ to be a suitable planar disk for each $p\in\Lambda$). Also fix a smoothly bounded simply-connected compact domain $K\subset M\setminus\Lambda$, up to shrinking the sets $\Omega_p$ if necessary assume that $K\subset M\setminus\Omega$, and extend $X$ to $\Omega\cup K\to\r^n$ such that $X|_K\colon K\to\r^n$ is any conformal minimal immersion of class $\Cscr^1(K\cup\Omega)$. Applying Theorem \ref{th:main-intro2} to these data, any group homomorphism $H_1(M;\z)\to\r^n$, and the integer $k=1$, we obtain a conformal minimal immersion $\wt X\colon M\to\r^n$ which has a contact of order $1$ with $X|_\Omega$ at every point in $\Lambda$. Thus, $\wt X|_\Lambda=X|_\Lambda=\Zgot$ and the generalized Gauss map $G_{\wt X}|_{\Lambda}=[\di \wt X]|_\Lambda=[\di X]|_\Lambda=G_X|_\Lambda=G$. This concludes the proof.
\end{proof}

%%%%%%%%%%
%%%%%%%%%%
%%%%%%%%%%
%%%%%%%%%%   THANKS
%%%%%%%%%%
%%%%%%%%%%

\subsection*{Acknowledgements}
Research partially supported by the MINECO/FEDER grant no. MTM2014-52368-P and MTM2017-89677-P, Spain. 

The authors wish to express their gratitude to Franc Forstneri\v c and Francisco J. L\'opez for helpful discussions which led to improvement of the paper.

%%%%%%%%%%
%%%%%%%%%%
%%%%%%%%%%
%%%%%%%%%%   THE BIBLIOGRAPHY
%%%%%%%%%%
%%%%%%%%%%

%{\bibliographystyle{abbrv} \bibliography{bibACL}}

%%%%%%%%%%
%%%%%%%%%%
%%%%%%%%%%
%%%%%%%%%%   AFFILIATIONS
%%%%%%%%%%
%%%%%%%%%%

\bigskip

\noindent Antonio Alarc\'{o}n

\noindent Departamento de Geometr\'{\i}a y Topolog\'{\i}a e Instituto de Matem\'aticas (IEMath-GR), Universidad de Granada, Campus de Fuentenueva s/n, E--18071 Granada, Spain.

\noindent  e-mail: {\tt alarcon@ugr.es}

\bigskip

\noindent Ildefonso Castro-Infantes

\noindent Departamento de Geometr\'{\i}a y Topolog\'{\i}a, Universidad de Granada, Campus de Fuentenueva s/n, E--18071 Granada, Spain.

\noindent  e-mail: {\tt icastroinfantes@ugr.es}

\end{document}